\documentclass[10pt,a4paper]{article}
\usepackage[latin2]{inputenc}
\usepackage[english]{babel}
\usepackage{eucal}
\usepackage[english]{babel}
\usepackage{amsmath,amsfonts,amstext,amsbsy,amsopn,amssymb,amsthm,amscd}
\usepackage{amsxtra,latexsym}
\usepackage{exscale}
\usepackage{bef_alex-modified}
\numberwithin{equation}{section}
\usepackage{graphicx,mathrsfs}
\usepackage{psfrag}
\usepackage{amsfonts}
\usepackage{amsthm,amsmath,amssymb}
\usepackage{color}
\usepackage{amssymb,mathrsfs,enumerate}
\usepackage{paralist,eucal}
\usepackage{color}
\definecolor{dred}{rgb}{.8,0,0}
\definecolor{ddmagenta}{rgb}{0.7,0,0.9}
\definecolor{ddcyan}{rgb}{0,0.2,1.0}
\definecolor{dblue}{rgb}{0,0,0.7}
\definecolor{ddgreen}{rgb}{0,0.4,0.4}

\def\trait #1 #2 #3 {\vrule width #1pt height #2pt depth #3pt}
\def\fin{
    \trait .3 5 0
    \trait 5 .3 0
    \kern-5pt
    \trait 5 5 -4.7
    \trait 0.3 5 0
\medskip}

\newcommand{\bele}{\begin{lemm}\begin{sl}}
\newcommand{\enle}{\end{sl}\end{lemm}}
\newcommand{\bedef}{\begin{defi}\begin{sl}}
\newcommand{\eddef}{\end{sl}\end{defi}}
\newcommand{\bete}{\begin{teor}\begin{sl}}
\newcommand{\ente}{\end{sl}\end{teor}}
\newcommand{\beos}{\begin{osse}\begin{rm}}
\newcommand{\eddos}{\end{rm}\end{osse}}
\newcommand{\bepr}{\begin{prop}\begin{sl}}
\newcommand{\empr}{\end{sl}\end{prop}}
\newcommand{\bepro}{\begin{prob}\begin{rm}}
\newcommand{\empro}{\end{rm}\end{prob}}
\newcommand{\bede}{\begin{defin}\begin{sl}}
\newcommand{\edde}{\end{sl}\end{defin}}
\newcommand{\beco}{\begin{coro}\begin{sl}}
\newcommand{\enco}{\end{sl}\end{coro}}
\newcommand{\behy}{\begin{hypo}\begin{sl}}
\newcommand{\enhy}{\end{sl}\end{hypo}}

\newcommand{\thspace}{\hspace{3mm}}

\newcommand{\RR}{\mathbb{R}}

\newcommand{\beeq}[1]{\begin{equation}\label{#1}}
\newcommand{\eddeq}{\end{equation}}
\newcommand{\beeqa}[1]{\begin{eqnarray}\label{#1}}
\newcommand{\eddeqa}{\end{eqnarray}}
\newcommand{\beal}[1]{\begin{align}\label{#1}}
\newcommand{\eddal}{\end{align}}
\newcommand{\bespl}[1]{\begin{split}\label{#1}}
\newcommand{\edspl}{\end{split}}
\newcommand{\bega}[1]{\begin{gather}\label{#1}}
\newcommand{\edga}{\end{gather}}
\newcommand{\beeqax}{\begin{eqnarray*}}
\newcommand{\eddeqax}{\end{eqnarray*}}

\newcommand{\no}{\nonumber}

\newcommand{\weaksto}{{\rightharpoonup^*}}
\newcommand{\weakto}{\rightharpoonup}

\newcommand{\hn}{\|_{L^2(\Omega)}}
\newcommand{\vn}{\|_{\V}}

\newcommand{\dt}{\partial_t}

\newcommand{\dn}{\partial_{\bf n}}

\newcommand{\itt}{\int_0^t}

\newcommand{\io}{\int_\Omega}

\newcommand{\e}{\varepsilon}

\newcommand{\mezzo}{{\frac{1}{2}}}

 \DeclareMathOperator{\dive}{div}

\let\TeXchi\chi
\def\chi{{\setbox0 \hbox{\mathsurround0pt
$\TeXchi$}\hbox{\raise\dp0 \copy0 }}}

\newcommand{\ub}{\mathbf{u}}
\newcommand{\uu}{\mathbf{u}}
\newcommand{\vv}{\mathbf{v}}
\newcommand{\ww}{\mathbf{w}}
\newcommand{\vb}{\mathbf{v}}

\newcommand{\teta}{\vartheta}

\newcommand{\aein}{\text{a.e. in\,}}
\newcommand{\eeta}{{\mbox{\boldmath$\eta$}}}
\newcommand{\zzeta}{{\mbox{\boldmath$\zeta$}}}
\newcommand{\tensore}{\varepsilon({\bf u})}
\newcommand{\tensoret}{\varepsilon({\mathbf{u}_t})}

\newcommand{\forae}{\text{for a.a.}}
\newcommand{\foraa}{\text{for a.a.}}

\def\fine{\hfill\kern4pt \vrule height4pt depth0pt width4pt }

\numberwithin{equation}{section}
\numberwithin{equation}{section}

\newcommand{\Ha}{L^2 (\Omega;\R^d)}
\newcommand{\V}{H^1 (\Omega)}
\newcommand{\Vp}{W^{1,p} (\Omega)}

\newcommand{\boZ}{H_{0}^1(\Omega;\R^d)}
\newcommand{\boY}{H_{0}^2(\Omega;\R^d)}

\newcommand{\pairing}[4]{ \sideset{_{#1 }}{_{ #2}}  {\mathop{\langle #3 , #4  \rangle}}}
\newcommand{\bilh}[3]{a_{\mathrm{el}}({#1}#2,#3)}
\newcommand{\bilj}[3]{a_{\mathrm{vis}}({#1}#2,#3)}
\newcommand{\oph}[2]{\mathcal{E}\left({#1}{#2}\right)}
\newcommand{\ophname}{\mathcal{E}}
\newcommand{\opjname}{\mathcal{V}}
\newcommand{\opj}[2]{\mathcal{V}\left({#1}{#2}\right)}
\newcommand{\opchi}{\mathcal{B}}
\newcommand{\argmin}{\mathrm{Argmin}}

\newcommand{\heat}{\mathsf{c}}
\newcommand{\condu}{\mathsf{K}}
\newcommand{\ental}{w}
\newcommand{\w}{\ental}

\newcommand{\mmu}{{\mbox{\boldmath$\mu$}}}

\newcommand{\utau}[1]{\uu_{\tau}^{#1}}
\newcommand{\wtau}[1]{\w_{\tau}^{#1}}
\newcommand{\btau}[1]{b_{\tau}^{#1}}
\newcommand{\vtau}[1]{v_{\tau}^{#1}}
\newcommand{\chitau}[1]{\chi_{\tau}^{#1}}
\newcommand{\chitaue}[1]{\chi_{\tau,\eps}^{#1}}
\newcommand{\xitau}[1]{\xi_{\tau}^{#1}}
\newcommand{\zetau}[1]{\zeta_{\tau}^{#1}}
\newcommand{\ftau}[1]{\mathbf{f}_{\tau}^{#1}}
\newcommand{\gtau}[1]{g_{\tau}^{#1}}
\newcommand{\Thetatau}[1]{{\Theta^*}_{\tau}^{#1}}
\newcommand{\dtau}[2]{\mathrm{D}_{\tau,{#1}}(#2)}
\newcommand{\duetau}[2]{\mathrm{D}_{\tau,{#1}}^2(#2)}

\newcommand{\pwc}[2]{\overline{#1}_{#2}}
\newcommand{\pwl}[2]{{#1}_{#2}}
\newcommand{\pwwll}[2]{\widehat{#1}_{#2}}
\newcommand{\upwc}[2]{\underline{#1}_{#2}}
\newcommand{\supp}{\mathrm{supp}}

\newcommand{\DDD}[3]{\begin{array}[t]{c}#1\vspace*{-1em}\\_{#2}\vspace*{-.5em}\\_{#3}\end{array}}
\newcommand{\ddd}[3]{\DDD{\begin{array}[t]{c}\underbrace{#1}\vspace*{.6em}\end{array}}{\text{\footnotesize #2}}{\text{\footnotesize #3}}}
\newcommand{\down}{\downarrow}
\newcommand{\ciro}{\mathcal{C}_\rho}
\newcommand{\BV}{\mathrm{BV}}
\newcommand{\expo}{\mathfrak{s}}

\newenvironment{rcomm}{\color{dred} \textsf{R:}\,}{\color{black}}
\newenvironment{bcomm}{\color{dblue}}{\color{black}}

\newenvironment{new}{\color{ddmagenta}}{\color{black}}
\newcommand{\bne}{\begin{new}}
\newcommand{\ene}{\end{new}}
\newcommand{\beric}{\begin{rcomm}}
\newcommand{\eric}{\end{rcomm}}
\newcommand{\bebe}{\begin{bcomm}}
\newcommand{\ebe}{\end{bcomm}}
\newenvironment{rickynew}{\color{ddmagenta}}{\color{black}}
\newcommand{\berin}{\begin{rickynew}}
\newcommand{\erin}{\end{rickynew}}

\newenvironment{rickyrev}{\color{red}}{\color{black}}
\newcommand{\berr}{\begin{rickyrev}}
\newcommand{\err}{\end{rickyrev}}

\newenvironment{newrickyrevnew}{\color{ddmagenta}}{\color{black}}
\newcommand{\nberr}{\begin{newrickyrevnew}}
\newcommand{\nerr}{\end{newrickyrevnew}}

\newenvironment{newrickyrev}{\color{green}}{\color{black}}
\newcommand{\berrn}{\begin{newrickyrev}}
\newcommand{\errn}{\end{newrickyrev}}

\newenvironment{nuovissimoricky}{\color{ddmagenta}}{\color{black}}
\newcommand{\bnvsr}{\begin{nuovissimoricky}}
\newcommand{\envsr}{\end{nuovissimoricky}}

\begin{document}

\title{A degenerating PDE system for
phase transitions and damage}

 \author{%
Elisabetta Rocca \footnote{Dipartimento di Matematica,
Universit\`a di Milano,  Via Saldini 50, 20133 Milano, Italy,
E-Mail: ~~{\tt elisabetta.rocca@unimi.it}. The work of E.R. was
supported by the FP7-IDEAS-ERC-StG Grant \#256872
(EntroPhase).} \and Riccarda Rossi\footnote{Sezione di Matematica del Dipartimento DICATAM,  Universit\`a di Brescia, Via Valotti 9, 25133 Brescia,
Italy, E-Mail: ~~{\tt riccarda.rossi@ing.unibs.it}. R.R. was
partially supported by a MIUR-PRIN 2008 grant for the project
``Optimal mass transportation, geometric and functional
inequalities and applications'' and by the FP7-IDEAS-ERC-StG Grant \#256872
(EntroPhase).} }

\date{}

 \maketitle
%


\begin{abstract}
In this paper, we analyze a PDE system 
arising in the modeling of phase transition and damage  phenomena
in thermoviscoelastic materials. The resulting evolution equations
in the unknowns $\teta$ (absolute temperature), $\uu$ (displacement), and $\chi$
(phase/damage parameter)
are strongly nonlinearly coupled.
Moreover, the momentum equation for $\uu$ contains  $\chi$-dependent
elliptic operators, which degenerate at the \emph{pure phases} (corresponding to the values
$\chi=0$ and $\chi=1$), making the  \emph{whole} system degenerate.

 That is why, we have to resort to  a suitable weak solvability  notion
for the analysis of the problem: it consists of the weak formulations
 of the heat and momentum equation, and,  for the phase/damage parameter $\chi$, of a generalization of the principle of
 virtual powers, partially mutuated from the theory of rate-independent damage processes.

 To   prove an existence result for this weak formulation,
 an approximating problem is
introduced, where the elliptic degeneracy  of the displacement equation  is  ruled out:
in the framework of damage models,
this corresponds to  allowing for \emph{partial damage}
 only.
 For such
an approximate system, global-in-time existence and well-posedness results are established in various
cases.
 Then, the passage to the
limit to the
degenerate system is performed via suitable variational techniques.
\end{abstract}

\vspace{4mm}

\noindent {\bf AMS (MOS) subject clas\-si\-fi\-ca\-tion:}\thspace
35K65, 35K92, 35R11,  80A17, 74A45.

\section{Introduction}
We consider the following PDE system
\begin{align}
& \heat(\teta) \teta_t +\chi_t \teta +\rho\teta \dive (\ub_t)
-\dive(\condu( \teta) \nabla\teta) = g \quad\hbox{in
}\Omega\times (0,T),\label{eq0}
\\
&\ub_{tt}-\dive(a(\chi)\mathrm{R}_v\tensoret+b(\chi)\mathrm{R}_e\tensore-\rho\teta\mathbf{1})={\bf
f}\quad\hbox{in }\Omega\times (0,T),\label{eqI}
\\
&\chi_t +\mu \partial I_{(-\infty,0]}(\chi_t)-\dive({\bf
d}(x,\nabla\chi))+W'(\chi) \ni - b'(\chi)\frac{\tensore \mathrm{R}_e \tensore}2 +
\teta \quad\hbox{in }\Omega \times (0,T),\label{eqII}
\end{align}
which describes a thermoviscoelastic system occupying a reference
domain $\Omega \subset \R^d$, $d \in \{2,3\}$,  supplemented with suitable initial and boundary conditions.
 The symbols $\teta$ and $\uu$ respectively denote the
absolute temperature of the system and the vector of \emph{small
displacements}. Depending on the choices of the functions $a$ and
$b$,
 we obtain a model
\begin{compactenum}
 \item[-] for \emph{phase transitions}: in this case, $\chi$ is the order parameter, standing for the local proportion
of one of the two phases;
\item[-]  for \emph{damage}: in this case, $\chi$ is the damage parameter,
 assessing
 the soundness of the material.
\end{compactenum}
We will assume that $\chi$  takes values
 between $0$ and $1$, choosing $0$ and $1$
 as reference values:
\begin{compactenum}
\item[-] for the \emph{pure phases} in phase change models (for
example, $\chi=0$ stands for the solid phase and $\chi=1$ for the
liquid one in solid-liquid phase transitions, and  one has $0 <
\chi < 1$ in  the so-called \emph{mushy regions});
\item[-] for the completely \emph{damaged} $\chi=0$ and the \emph{undamaged} state $\chi =1$,
respectively, in damage models, while $0 <
\chi < 1$ corresponds to \emph{partial damage}.
\end{compactenum}

\subsection{The model}
\label{ss:1.1}
 Let us now briefly  illustrate  the derivation of the
 PDE system \eqref{eq0}--\eqref{eqII}.  We shall systematically  refer
 for more details to
\cite{rocca-rossi1}, where we dealt with the case of phase transitions in thermoviscoelastic materials,
and just
underline here the main differences with respect to the discussion in \cite{rocca-rossi1}.

 Equation \eqref{eqI},   governing  the evolution of the displacement $\ub$,
  is the classical balance
equation for macroscopic movements (also known as the {\em stress-strain
relation}),  in which inertial effects are  taken into account as well.   It is derived from
the principle of virtual power (cf.~\cite{fremond}),  which yields
\begin{equation}
\ub_{tt}-\dive\sigma =\mathbf{f}\quad \hbox{in }
\Omega\times (0,T),  \label{equieq1}
\end{equation}
 where
the symbol $\dive$ stands
both for the scalar and for the  vectorial divergence operator,
$\sigma
$ is the stress tensor, and
 $\mathbf{f}$  an  exterior volume force. For $
 \sigma$, we adopt    the well-known constitutive law
\begin{equation}\label{constisigma}
\sigma =\sigma^{\mathrm{nd}}+\sigma^{\mathrm{d}}=\frac{\partial {\cal F} }{\partial
\varepsilon ( \mathbf{u})}+\frac{\partial {\cal P} }{\partial
\varepsilon (\mathbf{u}_{t})},
\end{equation}
with $\varepsilon(\ub)$ the linearized symmetric strain tensor, which in the
(spatially) three-dimensional case is given by
 $\varepsilon _{ij}(
\mathbf{u}):=(u_{i,j}+u_{j,i})/2$, $i,j=1,2,3$ (with the commas we
denote space derivatives).  Hence, the   explicit
 expression of
$\sigma$  depends on
the form   of the free energy functional ${\cal F}$ and  of the pseudopotential of dissipation
${\cal P}$.  The former is a function of the
state variables, namely 
 $\chi$, its gradient $\nabla\chi$, the absolute temperature $\teta$, and the linearized
symmetric strain tensor $\varepsilon(\ub)$. According to  Moreau's approach
(cf.~\cite{fremond} and references therein),  we include dissipation in the model by means of
 the latter potential, which depends on the dissipative
variables $\nabla \teta$, $\chi_t$, and  $\varepsilon (\mathbf{u}_{t})$. We will  make precise  our choice
for $\mathcal{F}$ and $\mathcal{P}$ below,
cf.\  \eqref{psi} and \eqref{fi}.

{We shall
supplement \eqref{equieq1} with
 a
  zero Dirichlet boundary condition
  on the boundary of $\Omega$
\begin{equation}\label{bouu}
\ub=
{\bf 0}\quad \hbox{on }\partial\Omega\times (0,T)\,,
\end{equation}
yielding a  {\em pure displacement} boundary value problem for
$\uu$, according to the terminology of~\cite{ciarlet}. However,
 our analysis carries over to other kinds of boundary conditions
on $\uu$,} 
 see Remark~\ref{rem-other-b.c.}.

Following Fr\'emond's
perspective,
 \eqref{eqI}
is  coupled with the equation of microscopic movements for the phase
variable $\chi$ (cf.\ \cite[p.\ 5]{fremond}), leading to \eqref{eqII}.
Let  ${B}$ (a density of energy function) and
$\mathbf{H}$ (an energy flux vector) represent the internal
microscopic forces  responsible for the mechanically induced heat
sources, and let us denote by  $B^{\text{d}}$ and
$\mathbf{H}^{\text{d}}$ their dissipative parts, and by
$B^{\text{nd}}$ and $\mathbf{H}^{\text{nd}}$ their non-dissipative
parts. Standard constitutive
relations yield
\begin{align}\label{constiB}
B=\,& \,B^{\text{nd}}+B^{\text{d}}=\frac{\partial {\cal F} }{
\partial\chi}+\frac{\partial {\cal P} }{\partial\chi_t}\,,\\
\label{constiH} \mathbf{H}=\,&
\,\mathbf{H}^{\text{nd}}+\mathbf{H}^{\text{d}}=\frac{\partial {\cal F} }{
\partial\nabla\chi}+\frac{\partial {\cal P} }{\partial\nabla\chi_t}\,.
\end{align}
Then, if the volume
amount of mechanical energy provided to the domain by the external
actions (which do not involve macroscopic motions) is zero, the
equation for the microscopic motions can be written as
\begin{equation}\label{mombal}
B-\dive{\mathbf{H}}=0\quad \hbox{in }\Omega \times (0,T),
\end{equation}
 where  $B$ and $\mathbf{H}$ will be  specified according to
the expression of  $\mathcal{F}$ and  $\mathcal{P}$.
{The natural  boundary condition for
this equation of motion is
$$
\mathbf{H}\cdot \mathbf{n}=0\quad \hbox{on }\partial \Omega \times
(0,T), $$
where ${\bf n}$ is the outward unit normal to
$\partial\Omega$. Thus (cf.\ \eqref{psi}) we obtain  the  homogeneous Neumann boundary
condition on $\chi$
\begin{equation}\label{bouchi}
\dn\chi=0\quad\hbox{on }\partial\Omega\times (0,T).
\end{equation}}

Finally, equation \eqref{eq0} is derived from the internal energy balance
\begin{equation}\label{internbal}
e_t+\dive {\bf q}=g+\sigma :\tensoret+B\chi_t+\mathbf{H}\cdot\nabla\chi_t \quad\hbox{in }\Omega\times(0,T),
\end{equation}
where $g$ denotes a heat source and $e$ and ${\bf q}$ are obtained from $\mathcal{F}$
 and $\mathcal{P}$ by means of the standard
constitutive relations
\begin{equation}\label{constieq}
e=\calF-\teta\frac{\partial \calF}{\partial\teta}, \quad {\bf
q}=\frac{\partial \calP}{\partial\nabla\teta}.
\end{equation}
{We couple equation \eqref{internbal} with a  no-flux boundary condition:
\[{\bf q}\cdot{\bf n}=0 \quad\hbox{on }\partial\Omega\times (0,T)\]
implying (cf.\ \eqref{fi}) the homogeneous Neumann  boundary condition
\begin{equation}\label{boteta}
\dn\teta=0\quad\hbox{on }\partial\Omega\times (0,T).
\end{equation}}

From the above relations and the following choices for the free energy functional and of the pseudopotential
of dissipation (cf.~\eqref{psi} and \eqref{fi}), we derive the
PDE system \eqref{eq0}--\eqref{eqII} within the \emph{small perturbation assumption} \cite{germain} (i.e. neglecting
the quadratic
terms $|\chi_t|^2+a(\chi)\tensoret \mathrm{R}_v\tensoret$ on the right-hand side of the heat equation).
In agreement with Thermodynamics (cf.\ \cite{fremond,fne} and  \cite[Sec.\ 4, 6]{fre-newbook}), we choose
  the volumetric free
energy ${\cal F}$  of the form
\begin{equation}\label{psi}
{\cal F} (\teta, \tensore
,\chi,\nabla\chi)= \int_{\Omega}\left( f(\teta)+b(\chi)\frac{\varepsilon(\ub) \mathrm{R}_e\varepsilon(\ub)}{2}
+\phi(x,\nabla\chi) +W(\chi)  -\teta\chi  -\rho\teta\hbox{tr}(\varepsilon(\ub)) \right) \dd x\,,
\end{equation}
where $f$ is a concave function of $\teta$.
Notice that
the symmetric, positive-definite elasticity tensor
$\mathrm{R}_e$ is pre-multiplied by
 a  function $b$ of the phase/damage parameter $\chi$. In particular,
 \begin{compactenum}
 \item[-]
  in
the case of phase transitions in viscoelastic materials, a
meaningful choice for $b$ is $b(\chi)=1-\chi$, or a function
vanishing at $1$ \cite[Sec.\ 4.5, pp.\ 42-43]{fre-newbook}.
  This reflects the fact that we  have
the
 full
elastic contribution of $b(\chi)\varepsilon(\ub)
\mathrm{R}_e\varepsilon(\ub)$ only in the non-viscous phase,
 and that  such a contribution  is null  in the viscous one (i.e.~when $\chi=1$);
 \item[-]
for damage models a significant choice is
instead $b(\chi)=\chi$ (cf.~\cite{fne} and \cite[Sec.\ 6.2, pp.\ 102-103]{fre-newbook}  for further comments on this
topic). The term $\chi\frac{\varepsilon(\ub) \mathrm{R}_e\varepsilon(\ub)}{2}$ represents the classical
elastic contribution in which the stiffness of the material decreases as $\chi$ approaches $0$, i.e.~during the evolution
of damage.
\end{compactenum}
The term $\phi(x,\nabla\chi)+W(\chi)$  is a
\textsl{mixture} or \textsl{interaction free-energy}. We shall suppose that
 $\phi: \Omega \times \R^d \to [0,+\infty)$
is a normal integrand, such that for almost all $x \in \Omega$ the function
$\phi(x,\cdot): \R^d \to [0,+\infty)$ is convex,  $\mathrm{C}^1$,  with $p$-growth, and $p >d$.
 Hence,
 the field
${\bf d}(x,\cdot)= \nabla \phi(x,\cdot)\,:\,\RR^d\to \RR^d$, $x\in \Omega$,
leads to a
$p$-Laplace type operator in \eqref{eqII}.
The prototypical example is $\phi(x,\nabla \chi)= \frac1p |\nabla \chi|^p$, yielding
${\bf d}(x, \nabla\chi):=|\nabla\chi|^{p-2}\nabla\chi$.
Let us point out that the gradient of  $\chi$
accounts for
interfacial energy effects in phase transitions, and  for the influence of damage at a material point,
undamaged in its neighborhood, in  damage models.
In this sense we can say that the term $\frac{1}{p}|\nabla\chi|^p$ models
nonlocality of the phase transition or the damage process,
i.e.\   the feature that  a particular point  is influenced by its surrounding.
 In damage, this leads to possible hardening or softening effects (cf.\ also \cite{bmr} for further comments on this topic).
Gradient regularizations of $p$-Laplacian type
are often adopted  in the mathematical papers on damage  (see for example \cite{bobo2,bosch,hk1,MieRou06,mrz,mt}),
and in the
 modeling literature as well
  (cf., e.g., \cite{fremond,fne,la}).
  In a different context,  a $p$-Laplacian elliptic regularization with $p>d$ has also been exploited
  in
\cite{abels}, in order to study a diffuse interface model for the
flow of two viscous incompressible Newtonian fluids in a bounded domain.

In the following, we will also scrutinize another kind of elliptic regularization in {\eqref{eqII}},
given by the {\sl nonlocal}  $s$-Laplacian
operator on the Sobolev-Slobodeckij space $W^{s,2}(\Omega)$,  hereafter denoted by $A_s$
 (cf.\ \eqref{As} later on for its precise definition).
 Recently,  fractional Laplacian operators have been widely investigated
(cf., e.g., \cite{CafSoug, Valdi} and the  references therein), and used
 in connection with  real-world applications, such as thin obstacle problems,
 finance, material sciences,
but also phase transition and damage phenomena (cf., e.g. \cite{GL} and \cite{Knees-Rossi-Zanini}).
For analytical reasons, we will have to assume $s>d/2$, which ensures the (compact)
embedding $ W^{s,2}(\Omega) \Subset \mathrm{C}^0 (\overline{\Omega})$, in the same way as
$W^{1,p}(\Omega) \Subset \mathrm{C}^0 (\overline{\Omega})$ for $p>d$. This property will play
a crucial role in the \emph{degenerate} limit to complete damage, as it did in~\cite{mrz} within the
rate-independent context, cf.\ Remark \ref{rmk:explanation} for more details.

As for the potential $W$, we  suppose that
\begin{equation}
\label{dabliu}
W = \widehat\beta + \widehat \gamma,
\end{equation}
with $\widehat\beta: \R \to (-\infty,+\infty]$
 convex  and possibly
nonsmooth, and
 $\widehat\gamma : \R \to \R$ smooth  and possibly nonconvex.
We will take the domain of $\widehat{\beta}$ to be
contained in $[0,1]$. Note that, in this way, the values outside
$[0,1]$ (which indeed are not physically meaningful for the order
parameter $\chi$, denoting a phase or damage proportion) are
excluded. Typical examples of functionals which we can include in
our analysis are the logarithmic potential
\begin{equation}
\label{logW} W(r)= r \ln (r) + (1-r) \ln(1-r) -c_1 r^2 -c_2 r -c_3
\quad \text{for } r \in (0,1),
\end{equation}
where $c_1$ and $ c_2$  are positive constants, as well as the sum
of the indicator function $\widehat \beta:= I_{[0,1]}$ with a \emph{nonconvex}
$\widehat\gamma$. In such a  case,
 in
\eqref{eqII}
the derivative
 $W'$ needs to be understood as the subdifferential $\partial W= \partial\widehat\beta +\widehat\gamma' $ in the sense of convex analysis.

The term $\rho\teta\hbox{tr}(\varepsilon(\ub))$ in
\eqref{psi} accounts for the thermal expansion of the system, with
 the thermal expansion coefficient $\rho$  assumed
to be constant (cf., e.g., \cite{KreRoSprWilm}).
 Indeed, one could consider
  more general functions $\rho$
depending, e.g., on the phase parameter $\chi$ and vanishing when $\chi=0$.
  This would be meaningful  especially in damage models, where the
terms associated with
 deformations should disappear once the material is completely damaged (cf., e.g., \cite{bobo2}). We will discuss
the mathematical difficulties   attached to  this extension  in Section~\ref{mathdiff}.

For the pseudo-potential $\mathcal{P}$, following  \cite[Sec.\ 4, 6]{fre-newbook}  we take
\begin{align}\label{fi}
{\cal P} (\nabla\teta, \chi_t, \varepsilon (\mathbf{u}_{t}))=\,&
\,\frac{\mathsf{K}(\teta)}{2}|\nabla\teta|^2
+{\frac{1}{2}}{|\chi_t|^2}  + \mu I_{(-\infty,0]}(\chi_t)
+a(\chi)\frac{\e(\ub_{t})\mathrm{R}_v\e(\ub_t)}{2}\,,
\end{align}
where  
$\mathrm{R}_v$  is  a symmetric and positive
definite viscosity matrix,
premultiplied by a function $a$ of $\chi$. In particular,
for phase change models, one can take for  example  $a(\chi)= \chi$. The underlying physical
interpretation is that
the viscosity term $\chi \e(\ub_{t})\mathrm{R}_v\e(\ub_t)$ vanishes
when we are in the non-viscous phase, i.e.\ in the solid phase $\chi=0$.  Also in damage models
the choice $a(\chi)=\chi$ is considered, cf.\ e.g.\ \cite{mrz}.
The heat conductivity   function $\mathsf{K}$ will be assumed  continuous; for the analysis of
system \eqref{eq0}--\eqref{eqII},
we will need to impose some compatibility conditions on the growth  of $\mathsf{K}(\teta)$ and of  the
heat capacity function   $\mathsf{c}(\teta)=-\teta f{''}(\teta)$
 in \eqref{eq0},  see 
 Hypothesis (II)
  in Section~\ref{ss:assumptions}.
 Furthermore, in \eqref{fi} $\mu \geq 0$ is a non-negative coefficient: for $\mu>0$
we encompass in our model the \emph{unidirectionality} constraint $\chi_t \leq 0$ a.e.\ in $\Omega \times (0,T)$.
In fact,   throughout the paper  we are going to use the term \emph{irreversible}
in connection with the case in which the
 process under consideration is unidirectional, which is indeed typical of
damage phenomena.

With straightforward computations, from \eqref{constiB}--\eqref{mombal} and using the
form of the free energy functional \eqref{psi} and of the pseudopotential of dissipation \eqref{fi}, we
derive equations \eqref{eq0}--\eqref{eqII},
 neglecting the quadratic contributions in the velocities on the
right-hand side in \eqref{eq0} by means of the aforementioned \emph{small perturbation assumption}~\cite{germain}.
{This is a simplification needed
from the analytical point of view in order to solve the problem.}
Indeed, in a forthcoming paper we plan to tackle the PDE system \eqref{eq0}--\eqref{eqII}, featuring in addition
these quadratic terms in the temperature equation. To do so, we are going to resort to specific techniques, partially mutuated from
\cite{fpr09}, however confining the analysis to some particular cases.

In fact, to our knowledge only few results are available on
diffuse interface  models in thermoviscoelasticity (i.e.\ also accounting for the evolution of the
displacement variables, besides the temperature and the
order parameter): among others, we quote \cite{fr,fralloys,rocca-rossi1,rocca-rossi2}.
In all of these papers, the small perturbation assumption is adopted. For,
without it in the spatial
three-dimensional case existence results seem to be out of reach, at the moment, even when
the equation for displacements is neglected
(whereas the existence of solutions to the \emph{full} phase change model in the unknowns $\teta$ and $\chi$ has been obtained in $1D$ in \cite{ls1}). This has led to the development of suitable \emph{weak solvability} notions
to handle (the usually neglected) quadratic terms, like in \cite{fpr09} (where however $\ub$ is
 still   taken constant).
 Also in \cite{roubicek-SIAM},
a PDE system coupling the displacement and the temperature equation (with quadratic nonlinearities)  and a \emph{rate-independent} flow rule for an internal dissipative variable $\chi$ (such as the damage parameter) has been
analyzed. Rate-independence means that the evolution equation for $\chi$ has no longer the \emph{gradient flow} structure of \eqref{eqII}:  the term $\chi_t$ therein
is replaced by $\mathrm{Sign}(\chi_t)$, viz.\ in the pseudo-potential $\mathcal{P}$, instead of
the quadratic contribution $\frac12 |\chi_t|^2$ we have the $1$-homogeneous dissipation term $|\chi_t|$.
In the frame of the (weak) \emph{energetic formulation} for rate-independent systems \cite{Mie05}, suitably adapted to the
temperature-dependent case, in \cite{roubicek-SIAM} existence results have been obtained.
A temperature-dependent, \emph{full} model
for (rate-dependent) damage has been addressed
in \cite{bobo2} as well, with local-in-time existence results.

\subsection{Mathematical difficulties and related literature}
\label{mathdiff}

The main difficulties attached to the analysis of system \eqref{eq0}--\eqref{eqII}
are:
\begin{compactenum}
\item[{\bf 1)}] the \emph{elliptic degeneracy} of the momentum
equation \eqref{eqI}: in particular, we allow for the
 positive coefficients
$a(\chi)$ and $b(\chi)$ to tend to zero simultaneously; 
\item[{\bf 2)}]
the \emph{highly nonlinear coupling} between the single equations, resulting in the
 the quadratic terms $\chi_t \teta $,  $\teta \dive(\uu_t)$,  and  $|\e(\ub)|^2$
 in the heat and phase equations \eqref{eq0} and
 \eqref{eqII}, respectively;
 \item[{\bf 3)}]  the \emph{poor regularity} of the temperature variable, which brings
 about
  difficulties in dealing with the
 coupling between equations \eqref{eq0} and \eqref{eqI}
 when we consider the thermal expansion terms (i.e. we take $\rho\neq0$);
\item[{\bf 4)}]  the \emph{doubly nonlinear} character of \eqref{eqII}, due to the
nonsmooth graph $\partial \widehat \beta$ and the nonlinear
 operator $-\mathrm{div}(\mathbf{d}(\nabla \chi)) \sim -\Delta_p \chi$
  (which on the other hand has a key regularizing
role). Furthermore,
if we set  $\mu>0$  in \eqref{eqII} to enforce an irreversible evolution for $\chi$, the simultaneous
presence
of the terms  $-\mathrm{div}(\mathbf{d}(\nabla \chi))$ and $\partial I_{(-\infty,0]}(\chi_t)$  makes it difficult
to derive suitable estimates for $\chi$, also due to the low regularity of the right-hand side of
\eqref{eqII}.
\end{compactenum}
We now partially survey
how each of these problems
has been handled in the recent literature.

As for \textbf{1)}, in \cite{rocca-rossi1,rocca-rossi2}
we have focused on the \emph{phase transition}
  case, in which $a(\chi)=\chi$ and $b(\chi)=1-\chi$. We have proved the local-in time (in the $3D$-setting)
  and the global-in-time (in the $1D$-setting) well-posedness of a  
  system in thermoviscoelasticity analogous to \eqref{eq0}--\eqref{eqII} (with the
Laplacian  instead of the $p-$Laplacian in \eqref{eqII}, in the case $\mu=0$ and  $\rho=0$, and
for constant heat capacity  and heat conductivity in
\eqref{eq0}). The main idea in \cite{rocca-rossi1,rocca-rossi2}
to handle the possible elliptic degeneracy of \eqref{eqI} is in fact to \emph{prevent it}. Specifically, we have shown that, if the initial datum
$\chi_0$ stays away from the values points $0$ and $1$, so does $\chi$ during this evolution,
guaranteeing
 that the
operators in \eqref{eqI} are uniformly elliptic.
 This \emph{separation property} is proved
 by exploiting a sufficient coercivity of $W$ at the thresholds $0$ and $1$, which for example holds true for
 the logarithmic potential \eqref{logW}.

In \cite{bosch,bss} an  \emph{isothermal} (irreversible) model for damage
 has been considered: therein,
because of the elliptic degeneracy of \eqref{eqI}, the authors only prove a local-in-time existence result.
For (isothermal) \emph{rate-independent} damage models \cite{MieRou06,bmr,mrz,mt},
 the results change significantly:
in this realm, only poor time-regularity of the solution component $\chi$ is to be expected, because
the $1$-homogeneous  dissipation contribution in $\chi_t$
 to $\mathcal{P}$ just ensures $\mathrm{BV}$-estimates for the function  $t \mapsto \chi(x,t)$. That is why,
 one has to resort to the aforementioned  notion of \emph{energetic solution} \cite{Mie05}, in which no time-derivatives of $\chi$
 are featured. Therefore, this concept is very flexible for analysis, and  has allowed for
 handling the (degenerate) case of \emph{complete damage} in \cite{bmr,mrz} by means of a
 specially devised formulation we will refer to later.

Concerning problem \textbf{2)}, as already mentioned
existence results have been obtained in \cite{fpr09} for
a \emph{full} model of phase transitions (in the reversible case $\mu=0$
and for constant $\ub$), even featuring the term
$|\chi_t|^2$ on the right-hand side of the temperature equation. Therein,
  a suitable notion of weak solution is addressed, consisting of the phase equation, coupled with a total
  energy balance and  a weak entropy inequality, for which existence is proved by
relying on an iterative
regularization procedure.
This technique cannot be applied to system \eqref{eq0}--\eqref{eqII}. Nonetheless,
 let us mention that a key assumption in \cite{fpr09} is a suitable growth of the heat conductivity $\mathsf{K}$.
Following \cite{roubicek-SIAM,rossi-roubi}, here we will combine it with conditions on
the heat capacity  coefficient $\mathsf{c}$ to handle the quadratic nonlinearities $\chi_t \teta$
 and  $\teta\dive (\ub_t)$   in \eqref{eq0}.

 Due to the lack of ``good" a priori estimates
 for $\teta$
  mentioned in
  \textbf{3)},
  we will not be able to encompass in our analysis the case of a non-constant
  thermal expansion coefficient $\rho$, e.g.\
  $\rho(\chi)= \chi$, which would still be interesting for damage \cite{bobo2}.
  Indeed, such a choice would lead to an additional term of the
   type $\teta\chi_t\dive (\uu)$ in the heat equation, which
   we would not be able to handle without resorting to further
   regularizations, and possibly proving only local-in-time existence results.
   Nonetheless, let us stress
   that, especially in case of phase transition phenomena,
the choice of a constant $\rho$  is quite reasonable  (cf., e.g., \cite{KreRoSprWilm}).

As for \textbf{4)}, in \cite{hk1} (dealing with Cahn-Hilliard systems coupled with elasticity and damage processes; see
also \cite{hk2}), the authors have devised a weak formulation of
\eqref{eqII} (in the irreversible case $\mu=1$) which  has allowed them to circumvent its triply nonlinear
character. Such a formulation strongly relies on the
 special choice
$\widehat{\beta}(\chi)= I_{[0,+\infty)}(\chi)$ (which, joint with the irreversibility constraint,
 still ensures that $\chi$ takes values in the meaningful interval $[0,1]$,
provided that $\chi_0 \in [0,1]$).
It consists of a \emph{one-sided} variational inequality (i.e.\ with test functions
having a fixed sign), and of an \emph{energy inequality}, see \eqref{weaksol-intro} later.

\subsection{Our results}
Unlike \cite{rocca-rossi1, rocca-rossi2}, here we shall not
enforce  separation of $\chi$
from the threshold values $0$ and $1$, and accordingly  we will  allow for general initial configurations
of $\chi$. Then, it is not to be expected that
either of the coefficients $a(\chi)$ and $b(\chi)$
stay away from $0$, which results in the elliptic degeneracy of
the displacement equation \eqref{eqI}.
To handle it, we shall approximate
system \eqref{eq0}--\eqref{eqII}
with a non-degenerating one, where we replace \eqref{eqI} with
\begin{equation}
\label{eqI-delta}
\ub_{tt}-\dive((a(\chi)+\delta)\mathrm{R}_v\tensoret+ (b(\chi)+\delta) \mathrm{R}_e\tensore-\rho\teta{\bf 1})={\bf
f}\quad\hbox{in }\Omega\times (0,T), \qquad \text{for } \delta>0.
\end{equation}
{Let us note  that,  to rule out the degeneracy,
it is sufficient to truncate away from zero only the coefficient $a(\chi)$ of the \emph{viscous} part of the elliptic operator in the momentum
equation. However,
for technical reasons which will become apparent in Section
\ref{sec:5}   (cf.\ Rmk.\ \ref{rmk:added}),
 when addressing the asymptotic analysis  as $\delta \down 0$ to the \emph{degenerate} limit, we will need to
 truncate the coefficient
 $b(\chi)$  as well, resulting in \eqref{eqI-delta}.}
 In the analysis of \eqref{eqI-delta}, we will distinguish the cases $\rho =0$ and $\rho \neq 0$: let us stress that, in the latter,
 there is an additional coupling between the heat and the momentum balance equations, which
 needs to be carefully
handled and indeed requires strengthening of some of our assumptions.
Furthermore, to avoid overburdening the paper we will tackle the case   $\rho \neq 0$
only for the \emph{reversible} system (i.e.\ with $\mu=0$).
More specifically,
in Theorems \ref{teor1}, \ref{teor3} and \ref{teor4},
 we will establish global-in-time
 existence results for the
non-degenerating system (\ref{eq0}, \ref{eqI-delta}, \ref{eqII}) with $\rho=0$,
both in the \emph{reversible}
 and in the \emph{irreversible} cases.
 We will work
under quite general
assumptions on $\heat$ and $\condu$,  basically requiring
that $\heat$ and $\condu$ are bounded from below  and above by the sum of a bounded function and function behaving like a small power of $\teta$ (cf.\ Hypotheses (I) and (II) in Sec.\ \ref{ss:assumptions}).
In Theorem \ref{teor1bis}, we will handle the case $\rho \neq 0, \, \mu=0$ and
prove the existence of global solutions to (\ref{eq0}, \ref{eqI-delta}, \ref{eqII}),
under the
more restrictive assumption that $\condu$ is bounded from below   and above  by a function behaving like
$\teta^{2+\nu}$,
 with  $0 \leq (d-2)/(d+2)<\nu <1$,  
  cf.\ Hypothesis (VIII).
  A continuous dependence result,  yielding uniqueness of solutions,
  for the non-degenerating \emph{isothermal} reversible system, possibly with $\rho \neq 0$,
  will be given in Theorem \ref{teor2}.
    Finally,
    we will address the
degenerate limit $\delta \down 0$ in Theorem \ref{teor5} in a less
general setting, in particular confining ourselves to the case
$\rho=0$. In what follows, we give more details on Thms.\
\ref{teor1}--\ref{teor5}.

Our first main result Thm.\ \ref{teor1} states the
existence of solutions to
system (\ref{eq0}, \ref{eqI-delta}, \ref{eqII}) {with $\rho=0$,  in the  \emph{reversible} case $\mu=0$,} with the heat equation
\eqref{eq0}  suitably reformulated
by means of
an \emph{enthalpy} transformation
(cf.\ Sec.\ \ref{ss:assumptions}),
switching from the temperature variable $\teta$ to the enthalpy $w$.  Already in the proof of  this global-in-time  existence result,
a key role is played by the aforementioned $p$-growth assumption on the function $\phi$ \eqref{psi}
with $p>d$. In fact, it enables us to derive an estimate for $\chi$ in $L^\infty (0,T;W^{1,p}(\Omega))$,
which in turns   allows for
a suitable regularity estimate on the displacement variable $\uu$, leading to
 a global-in-time bound   on the quadratic
nonlinearity $|\tensore|^2$ on the right-hand side of \eqref{eqII}.
For further details we refer to the proof of Thm.\ \ref{teor1}, developed
by passing to the limit in a carefully designed time-discretization scheme  and exploiting
\textsc{Boccardo\&Gallou\"{e}t}-type estimates on $\teta$.

Relying on the stronger Hyp.\ (VIII),  in the case $\rho \neq 0, \, \mu =0$ we will obtain enhanced estimates
on (the sequence, constructed by time discretization, approximating) $\teta$,  cf.\ also
 Remark \ref{rmk:afterThm2} later on.
These bounds and the related enhanced convergences will enable us to
handle the (passage to the limit in the time discretization of the) thermal expansion terms in \eqref{eq0}
and \eqref{eqI-delta}. In this way, we will conclude the proof of
  the existence Theorem  \ref{teor1bis} for system  (\ref{eq0}, \ref{eqI-delta}, \ref{eqII})
in the case $\rho \neq 0$.

In the reversible and \emph{isothermal} case,   continuous dependence of the solutions on the
initial and problem data
 is proved in Thm.\  \ref{teor2} under  a slightly more restrictive
condition on the field $\phi$, which is however satisfied in the prototypical case of the
$p$-Laplacian operator. 

 As
already mentioned, in  the \emph{irreversible} case  $\mu>0$  a major difficulty
in the analysis  of
system (\ref{eq0}, \ref{eqI-delta}, \ref{eqII}) stems from
 the simultaneous presence in \eqref{eqII} of the multivalued operators
$\partial I_{(-\infty,0]}(\chi_t)$ and  $\beta(\chi) = \partial
\widehat \beta (\chi)$, (cf.\ \eqref{dabliu}), as well as of the
$p$-Laplacian type operator $-\mathrm{div}(\mathbf{d}(x,\nabla
\chi))$, which still has  a key role in providing global-in-time
estimates for $|\tensore|^2$. To tackle this problem, following the
approach of \cite{hk1} we restrict to the yet meaningful case
$\widehat{\beta}=I_{[0,+\infty)}$ and consider a suitable weak
formulation of \eqref{eqII}. It consists (cf.\ Definition
\ref{def-weak-sol} later on) of the \emph{one-sided} variational
inequality
\begin{subequations}
\label{weaksol-intro}
\begin{equation}
\label{weaksol-intro-1}
\begin{aligned}
 &
 \int_\Omega  \Big( \chi_t(t) \varphi     +
\mathbf{d}(x,\nabla\chi(t)) \cdot \nabla \varphi  + \xi(t) \varphi +
\gamma(\chi(t)) \varphi   + b'(\chi(t))\frac{\varepsilon(\ub(t))
\mathrm{R}_e\varepsilon(\ub(t))}{2}\varphi  -\teta(t) \varphi \Big)
\, \mathrm{d}x    \geq 0 \\ & \quad   \text{for all }  \varphi \in
W^{1,p}(\Omega) \hbox{  with  }   \varphi \leq 0, \ \foraa\, t \in (0,T),  \qquad
\text{ with } \chi_t \leq 0, \  \xi \in
\partial I_{[0,+\infty)}(\chi),
\end{aligned}
\end{equation}
 and of  the following energy
inequality for all $t \in (0,T]$, for $s=0$,  and for almost all $0
< s\leq t$:
\begin{equation}
\label{weaksol-intro-2}
\begin{aligned}
 &  \int_s^t   \int_{\Omega} |\chi_t|^2 \dd x \dd r    +
  \int_{\Omega}\left( \phi(x,\nabla\chi(t)){+} W(\chi(t)) \right)\dd x\\ & \quad  \leq
 \int_{\Omega}\left( \phi(x,\nabla\chi(s)){+} W(\chi(s)) \right)\dd x
  +\int_s^t  \int_\Omega \chi_t \left(- b'(\chi)
  \frac{\varepsilon(\ub)\mathrm{R_e}\varepsilon(\ub)}2
+\teta\right)\dd x \dd r.
\end{aligned}
\end{equation}
\end{subequations}
In Sec.\ \ref{glob-irrev}, several comments and remarks shed
light on this weak solvability notion for \eqref{eqII}. In
particular, Proposition \ref{more-regu} shows that, if $\chi$ is
regular enough, \eqref{weaksol-intro}  and the subdifferential
inclusion \eqref{eqII} are equivalent.  In Theorem \ref{teor3}
we state the existence of global-in-time solutions to the weak
formulation of system (\ref{eq0}, \ref{eqI-delta}, \ref{eqII}) with
$\mu>0$, consisting of the (weakly formulated) enthalpy equation, of
\eqref{eqI-delta} and of \eqref{weaksol-intro}. The proof is again
carried out via a time-discretization procedure, combined with
Yosida-regularization techniques.

 Finally,  Theorem \ref{teor4}
focuses on the isothermal case, i.e.\ with a fixed temperature profile.
In this setting, we succeed in proving enhanced regularity for
$\chi$,
thus solving \eqref{eqII} in a stronger sense than \eqref{weaksol-intro}.
 In the particular case $\phi(x,\nabla \chi)=\frac1p |\nabla \chi|^p$
 the crucial estimate consists in  testing
\eqref{eqII} by $\partial_t (A_p \chi +\beta(\chi))$ (where for simplicity we
write $\beta$ as single-valued). This enables us to estimate separately the terms
$\partial I_{(-\infty,0]}(\chi_t)$ (again written as single-valued),  $A_p \chi$, and $\beta(\chi)$
in $L^\infty (0,T;L^2(\Omega))$, which is the  key step for proving the existence of solutions to the
\emph{pointwise}
subdifferential
inclusion \eqref{eqII}.

Uniqueness results for the
\emph{irreversible} system, even in the isothermal case, do not seem to be at hand,
due to the triply nonlinear character of equation \eqref{eqII}, cf.\ also
Remark~\ref{uni-irrev} ahead. Nonetheless, both in the reversible and in the irreversible case,
in Thms.\ \ref{teor1},   \ref{teor1bis}  and \ref{teor3}  we
will prove positivity  of the temperature $\teta$. In fact,
under suitable conditions on the initial temperature,
 for  $\mu>0$ we will also obtain a strictly
positive lower bound for $\teta$.

For the analysis of the degenerate limit $\delta \down 0$ of (\ref{eq0}, \ref{eqI-delta}, \ref{eqII}),
we have carefully adapted to the present setting techniques from
\cite{bmr} and \cite{mrz}. These two papers  deal with \emph{complete damage} in the fully rate-independent case,
and, respectively,
for a system featuring  a  rate-independent damage flow rule for $\chi$
and a displacement equation with viscosity and inertia according to
Kelvin-Voigt rheology. In particular, we have
extended the results from \cite{mrz} to the case of a \emph{rate-dependent} equation for $\chi$,
also coupled with the temperature equation.
Following \cite{bmr, mrz}, the key observation
is that, for any family $(\w_\delta,\uu_\delta,\chi_\delta)_\delta$
of  solutions to (\ref{eq0}, \ref{eqI-delta}, \ref{eqII}) (where $\w$ denotes the \emph{enthalpy}),
it is possible to deduce  for the quantities
$
\mmu_\delta:= \sqrt{ a(\chi_\delta)+\delta  } \, \eps(\partial_t \uu_\delta)$ and  $ \eeta_\delta:= \sqrt{ b(\chi_\delta)+\delta } \, \eps(\uu_\delta)$
the estimates
\[
\| \mmu_\delta \|_{L^2(0,T;L^2(\Omega;\R^{d \times d}))}, \ \| \eeta_\delta \|_{L^\infty(0,T;L^2(\Omega;\R^{d \times d}))} \leq C
\]
for a positive constant independent of $\delta$.
Therefore, there exist $\mmu $ and
$\eeta $ such that,
 up to a subsequence
$\mmu_\delta \weakto \mmu$ in $L^2(0,T;L^2(\Omega;\R^{d \times
d}))$ and $\eeta_\delta \weaksto \eeta$ in
$L^\infty(0,T;L^2(\Omega;\R^{d \times d}))$ as $\delta
\down 0$. According the terminology of \cite{mrz}, we refer
to $\mmu$ and $\eeta$, respectively, as
 the viscous and elastic \emph{quasi-stresses}.

In Theorem \ref{teor5} we will focus on the degenerate limit $\delta
\down 0$, confining the discussion to the
 case where  $\rho=0$ and
 $\mu>0$  (viz.\ the map
 $t\mapsto \chi(t,x)$ is nonincreasing for all $x\in
 \overline{\Omega}$).
We refer to Remark \ref{rmk:explanation} for a thorough justification of these
choices.
 Passing to the limit as $\delta \down 0$ in \eqref{eqI-delta}
and exploiting the above convergences  for $(\mmu_\delta)_\delta$ and $(\eeta_\delta)_\delta$
 we will  prove that there exist a triple $(\uu,\mmu,\eeta)$ solving the \emph{generalized} momentum
 balance
 \begin{subequations}
 \label{degen-intro}
 \begin{equation}
 \label{degen-intro-1}
 \ub_{tt}-\dive(\sqrt{a(\chi)}\mathrm{R}_v\mmu +\sqrt{b(\chi)}\mathrm{R}_e\eeta)={\bf
f}\quad\hbox{in }\Omega\times (0,T),
\end{equation}
such that the quasi-stresses fulfill
\begin{equation}
 \label{degen-intro-2}
 \mmu=  \sqrt{a(\chi)} \,\eps(\uu_t), \ \eeta  = \sqrt{b(\chi)} \, \eps(\uu)
 \ \text{a.e. in any open set } A \subset \Omega \times (0,T)
 \text{ s.t. }   A \subset \{ \chi >0\}.
 \end{equation}
 In addition to \eqref{degen-intro-1}--\eqref{degen-intro-2},
the notion of weak solution  to system \eqref{eq0}--\eqref{eqII}
arising in the limit $\delta \down 0$ consists of the (weak
formulation of the) enthalpy equation, of the \emph{one-sided}
variational inequality
\begin{equation}
 \label{degen-intro-3}
 \begin{aligned}
&  \int_0^T \int_\Omega \Big( \left(\chi_t+
\gamma(\chi)\right)\varphi  + \mathbf{d}(x,\nabla\chi) \cdot \nabla
\varphi \Big)   \leq \int_0^T \int_\Omega \left(
 -\frac1{2  b(\chi) } \eeta \,\mathrm{R}_e\, \eeta + \teta \right) \varphi \dd x
\dd t
 \\ &
 \qquad   \quad   \text{for all }  \varphi \in   L^p
 (0,T;W^{1,p}(\Omega))
\cap L^\infty (Q) \text{ with }  \varphi \geq 0 \text{ and }\mathrm{supp}(\varphi) \subset \{
\chi>0\},
\end{aligned}
 \end{equation}
 \end{subequations}
and of a \emph{generalized} total energy inequality,  featuring the quasi-stresses $\eeta$ and $\mmu$.
While referring to Remark \ref{rmk:comparison-with-hk} for more comments
 in this direction,
 we may observe here that \eqref{degen-intro-3} is in fact the  integrated  version in terms of \emph{quasi-stresses}
  of
 the variational inequality
 \eqref{weaksol-intro-1}.

\paragraph{Plan of the paper.}  In the next Section~\ref{s:main}
 we introduce the variational formulation for the
initial boundary value problem associated to the PDE system
\eqref{eq0}--\eqref{eqII}, as well as our main assumptions. Then, we state Theorems \ref{teor1}--\ref{teor4}
 on
the existence/uniqueness of solutions for  the reversible and the irreversible \emph{non-degenerating}
systems (i.e.\ $\delta>0$). The existence Thms.\ \ref{teor1},
 \ref{teor1bis},  \ref{teor3}, and \ref{teor4}  rely on the time-discretization
procedure of Section \ref{s:time-discrete}; their proof is carried  out
by passing to the limit with the time discretization
in Sections \ref{ss:4.1},   \ref{sec:4.2new},   \ref{ss:5.1}, and \ref{ss:5.2}. The continuous dependence Thm.\ \ref{teor2}
is proved in
Section \ref{CD}. Finally, Section
\ref{sec:5} is devoted to the  passage to the degenerate limit
$\delta \down 0$.

\noindent
The following table summarizes
  our results

\begin{center}
\begin{tabular}{|p{0.25\textwidth}|p{0.25\textwidth}|p{0.25\textwidth}|}
\hline
{\bf Results}&     $\mu=0$      & $\mu=1$   \\
\hline
&&\\
 $\rho=0$,  $\delta>0$ &  Theorem~\ref{teor1} (Sec.~\ref{ss:glob-rev}): $\, \exists$ %
 &  Theorem~\ref{teor3} (Sec.~\ref{glob-irrev}):  $\exists\,$
 \\
\hline
&&\\
 $\rho\neq 0$,  $\delta>0$& Theorem~\ref{teor1bis} (Sec.~\ref{ss:glob-rev}):  $\exists\,$
 &
 \\
\hline
&&\\
 $\rho=0$, $\teta$ constant, $\delta>0$& Theorem~\ref{teor2} (Sec.~\ref{ss:glob-rev}): uniqueness 
    &  Theorem~\ref{teor4} (Sec.~\ref{glob-irrev}):
  improved regularity
   \\
\hline
&&\\
 $\rho\neq 0$, $\teta$ constant, $\delta>0$
 & Theorem~\ref{teor2} (Sec.~\ref{ss:glob-rev}): uniqueness 
  &  Theorem~\ref{teor4} (Sec.~\ref{glob-irrev}):
  improved regularity
\\
\hline
&&\\
 $\rho=0$, 
  $\delta\down0$ &   %
  & Theorem~\ref{teor5}  (Sec.~\ref{sec:5}): $\exists\,$ degenerate case \\
\hline
\end{tabular}
\end{center}
\section{Setup and results for the non-degenerating system}
\label{s:main}

\subsection{Notation and preliminaries}
\label{ss:prelims}
\begin{notation}
\label{not:2.1}
\upshape
 Throughout the paper, given a Banach space $X$
we shall denote by $\|\cdot\|_{X}$ 
its norm,
and use the symbol $\pairing{}{X}{\cdot}{\cdot}$ for
the duality pairing between $X'$ and $X$.

Hereafter, we shall suppose that
\[
\Omega\subset\RR^d, \quad d\in \{2,3\} \ \ \text{is
 a bounded connected domain,   with $\mathrm{C}^2$-boundary $\partial\Omega$.}
\]
We will  identify both $L^2 (\Omega)$ and
$\Ha$ with their dual spaces, and denote by $(\cdot,\cdot)$ the
scalar product in $\R^d$,  by  $(\cdot,\cdot)_{L^2(\Omega)}$
both the scalar product in
 $L^2(\Omega)$,
and  in $\Ha$, and by
 $\boZ$ and $\boY$
the  spaces
$$
\begin{aligned}
&
\boZ:=\{\vv \in H^1(\Omega;\R^d) \,:\ \vv= 0 \ \hbox{ on }\partial\Omega
\,\}, \text{ endowed with the norm } \| \vv\|_{H_0^1(\Omega)}^2: = \int_{\Omega} \e(\vv) \colon \e(\vv)\, \dd x,
\\
&
\boY:=\{\vv \in H^2(\Omega;\R^d)\,:\ \vv ={0} \ \hbox{ on
}\partial\Omega \,\}.
\end{aligned}
$$
For $\sigma, \, p \geq 1 $ we will use the notation
\begin{equation}
\label{not-hk} W_+^{\sigma,p}(\Omega):= \left\{\zeta \in
W^{\sigma,p}(\Omega)\, : \ \zeta(x) \geq 0 \quad \foraa\, x \in
\Omega \right\} \quad \text{ and analogously for }
W_-^{\sigma,p}(\Omega).
\end{equation}
 We standardly denote by
\[
\text{$
A: H^1(\Omega) \to H^1(\Omega)'$ the operator $\pairing{}{H^1(\Omega)}{Au}{v}:= \int_\Omega \nabla u \cdot \nabla v \dd x$}
\]
and, for any $w \in H^1(\Omega)$, by
$m(w):= \pairing{}{H^1(\Omega)}{w}{1}$
its mean value.

 Given    a (separable) Banach space $X$,  we will denote by
  ${\rm BV}([0,T];X)$ (by $\mathrm{C}^0_{\mathrm{weak}}([0,T];X)$, respectively),
 the space
of functions from $[0,T]$ with values in $ X$ that are defined at every  $t \in [0,T]$
and  have  bounded
variation on  $[0,T]$  (and are \emph{weakly} continuous   on  $[0,T]$, resp.)

Finally, throughout the paper we shall denote by the symbols $c,\,c',\,
C,\,C'$  various positive constants depending only on known
quantities {and by $v_t$ (respectively $v_{tt}$) or (whenever it turns out to be more convenient) $\partial_t v$ (respectively $\partial_{tt}v$) the first (respectively) second partial derivatives  with respect to time of  a function $v$.}
\end{notation}

\noindent
\textbf{Preliminaries of mathematical elasticity.}
In what follows, we shall  assume the material to be homogeneous and
isotropic, so that the elasticity matrix $\mathrm{R}_e$ in equation
\eqref{eqII} may be represented by
$$
\mathrm{R}_e\varepsilon({\bf u})=\lambda_1\hbox{tr}(\varepsilon({\bf
u})){\bf 1}+2\lambda_2\varepsilon({\bf u}),
$$
where $\lambda_1,\lambda_2>0$ are the so-called Lam\'e constants and ${\bf
1}$ is the identity matrix.
In order to state the variational formulation  of the
initial-boundary value problem for \eqref{eq0}--\eqref{eqII}, we need
to introduce the bilinear forms related to the $\chi$-dependent
elliptic operators appearing in~\eqref{eqI}. Hence,  given a
\emph{non-negative} function $\eta \in L^\infty (\Omega)$,  let us
consider the continuous bilinear symmetric forms
$\bilh{\eta}{\cdot}{\cdot}, \, \bilj{\eta}{\cdot}{\cdot}: \, \boZ \times \boZ \to \RR$
defined for all $ \ub, \vb \in \boZ$ by
\begin{equation}
\label{bilinear-forms}
\begin{aligned}
& \bilh {\eta}{\uu}{\vv}:= \pairing{}{H^1(\Omega;\R^d)}{-\dive(\eta
\mathrm{R}_e \tensore)}{\vv} =\lambda_1\int_\Omega\eta\,\dive({\bf
u})\dive({\bf v})+2\lambda_2\sum_{i,j=1}^d
\int_\Omega\eta\,\varepsilon_{ij}({\bf u}) \varepsilon_{ij}({\bf
v}),
\\
& \bilj {\eta}{\uu}{\vv}:= \pairing{}{H^1(\Omega;\R^d)}{-\dive(\eta
\mathrm{R}_v \tensore)}{\vv}= \sum_{i,j=1}^d \int_\Omega\eta\,
\ell_{ij}\,\varepsilon_{ij}(\ub)\varepsilon_{ij}(\vb),
\end{aligned}
\end{equation}
where $(\ell_{ij}) \in \R^{d\times d}$ is the viscosity matrix
$\mathrm{R}_v$. Now, by Korn's inequality (see
eg~\cite[Thm.~6.3-3]{ciarlet}), the forms
$\bilh{\eta}{\cdot}{\cdot}$ and $\bilj{\eta}{\cdot}{\cdot}$ are
$\boZ$-elliptic and continuous. Namely,  there exist constants $C_1,
\, C_2>0 $, only depending on $\lambda_1$ and $\lambda_2$, such that
 such that for all~$\ub,\,\vb \in  \boZ$
\begin{align}
& \label{korn} \bilh{\eta}{\uu}{\uu}
 \geq \inf_{x \in \Omega}(\eta(x))\,C_1\Vert{\bf
u}\Vert^2_{\V}, \qquad  \bilj{\eta}{\uu}{\uu}\geq  \inf_{x \in
\Omega}(\eta(x))\,C_1\Vert{\bf u}\Vert^2_{\V},
\\
& \label{a:conti-form} |\bilh{\eta}{\uu}{\vv}| +
|\bilj{\eta}{\uu}{\vv} |
 \leq C_2 \| \eta \|_{L^\infty (\Omega)} \| \mathbf{u} \|_{\V} \| \mathbf{v}
 \|_{\V}.
\end{align}
 We  shall denote by $\ophname(\eta\, \cdot):\boZ\to
H^{-1}(\Omega;\R^d)$ and $\opjname(\eta\, \cdot):\,\boZ\to
H^{-1}(\Omega;\R^d)$ the linear operators associated with
$\bilh{\eta}{\cdot}{\cdot}$
 and $\bilj{\eta}{\cdot}{\cdot}$, respectively,  namely
\begin{equation}
\label{operator-notation-quoted}
\pairing{}{H^1(\Omega;\R^d)}{\oph{\eta}{\vv}}{\ww}:=
\bilh{\eta}{\vv}{\ww},
 \quad
\pairing{}{H^1(\Omega;\R^d)}{\opj{\eta}{\vv}}{\ww}:=
\bilj{\eta}{\vv}{\ww} \qquad \text{for all }{\bf v},\,{\bf w}\in
\boZ.
\end{equation}
It can be checked via an approximation argument
  that the following
regularity results hold:
\begin{subequations}
\begin{align}
& \label{reg-pavel-a} \text{if $\eta \in L^\infty(\Omega)$ and
${\uu}\in \boZ$, \ then  \, $\oph{\eta}{\uu}, \, \opj{\eta}{\uu} \in
H^{-1} (\Omega;\R^d)$,} \\
& \label{reg-pavel-b}
\text{if $\eta \in
W^{1,d} (\Omega)$ and ${\uu}\in \boY$, \ then  \, $\oph{\eta}{\uu}, \,
\opj{\eta}{\uu} \in \Ha$.}
\end{align}
\end{subequations}
\begin{remark}[The anisotropic inhomogeneous case]
\upshape \upshape In fact, the calculations we will develop
extend to the case of an anisotropic and inhomogeneous
material, for which the elasticity and viscosity matrices
$\mathrm{R}_e$ and  $\mathrm{R}_v$ are of the form
 $\mathrm{R}_e=(g_{ijkh})$ and  $\mathrm{R}_v=(\ell_{ijkh})$,
with functions
\begin{equation}
\label{funz_g-l}
 g_{ijkh},  \ \ell_{ijkh} \in \mathrm{C}^{1}(\Omega)\,,
\quad i,j,k,h=1,2,3,
\end{equation}
 satisfying the classical symmetry and
ellipticity conditions (with the usual summation convention)
\begin{equation}
\label{ellipticity}
\begin{array}{ll}
&  \!\!\!\!\!\!\!\!  \!\!\!\!\!\!\!\! g_{ijkh}=g_{jikh}=g_{khij}, \quad  \quad
\ell_{ijkh}=\ell_{jikh}=\ell_{khij}\,,
  \quad   i,j,k,h=1,2,3
\\
&  \!\!\!\!\!\!\!\!  \!\!\!\!\!\!\!\! \exists \, C_1>0 \,:  \qquad g_{ijkh} \xi_{ij}\xi_{kh}\geq
C_1\xi_{ij}\xi_{ij}, \   \ell_{ijkh} \xi_{ij}\xi_{kh}\geq
C_1\xi_{ij}\xi_{ij}   \quad   \text{for all } \xi_{ij}\colon
\xi_{ij}= \xi_{ji}\,,\quad i,j=1,2,3\,.
\end{array}
\end{equation}
Clearly, \eqref{ellipticity} ensures \eqref{korn}, whereas not only does
\eqref{funz_g-l}  imply \eqref{a:conti-form}, but the
$\mathrm{C}^1$-regularity also allows us to perform the third a
priori estimate of Section \ref{ss:3.2} rigorously.
\end{remark}
In what follows we will use  
 the following elliptic regularity result  (see
e.g.~\cite[Thm.\ 6.3-.6, p.\ 296]{ciarlet}, cf.\ also \cite[p.\ 260]{necas}):
\begin{equation}
\label{cigamma} \exists \, C_3,\, C_4>0 \quad \forall\,  \uu \in
\boY\, : \qquad
 C_{3} \| \uu \|_{H^2(\Omega)}  \leq \|\dive (\eps
(\uu))\|_{L^2(\Omega)} \leq C_{4} \| \uu \|_{H^2(\Omega)}\,.
\end{equation}

 Finally, in the weak formulation of the momentum equation \eqref{eqI},
besides $\opjname$ and $\ophname$
 we will also make use of the
operator
\begin{equation}
\label{op_ciro}
\mathcal{C}_\rho: L^2 (\Omega) \to  H^{-1}(\Omega;\R^d) \quad \text{defined by} \quad
\pairing{}{H^1(\Omega;\R^d)}{\ciro(\teta)}{\vv}:= - \rho \int_\Omega \teta \dive (\vv)\, \mathrm{d}x.
\end{equation}
\paragraph{Useful  inequalities.}
We recall  the celebrated  Gagliardo-Nirenberg
inequality (cf.~\cite[p.~125]{nier}) in a particular case: for
 all $r,\,q\in [1,+\infty],$ and for all $v\in L^q(\Omega)$ such that
$\nabla v \in L^r(\Omega)$, there holds
\begin{equation}\label{gn-ineq}
\|v\|_{L^s(\Omega)}\leq C_{\mathrm{GN}}
\|v\|_{W^{1,r}(\Omega)}^{\theta} \|v\|_{L^q(\Omega)}^{1-\theta},
\text{ with } \frac{1}{s}=\theta
\left(\frac{1}{r}-\frac{1}{d}\right)+(1-\theta)\frac{1}{q}, \ \  0
\leq \theta \leq 1, \end{equation}
 the positive constant $C_{\mathrm{GN}}$ depending only on
$d,\,r,\,q,\,\theta$. Combining the compact embedding
\begin{equation}
\label{dstar}
 \boY \Subset W^{1,d^\star{-}\eta}(\Omega;\R^d),
\quad \text{with } d^{\star}=
\begin{cases} \infty & \text{if }d=2,
\\
6 & \text{if }d=3,
\end{cases}
 \quad \text{for all $\eta >0$},
\end{equation}
(where for $d=2$ we mean that $\boY \Subset W^{1,q}(\Omega;\R^d)$ for all $1 \leq q <\infty$),
   with \cite[Thm. 16.4, p. 102]{LM}, we have
\begin{equation}
\label{interp} \forall\, \varrho>0 \ \ \exists\, C_\varrho>0 \ \
\forall\, \uu \in \boY\,: \ \
\|\e(\uu)\|_{L^{d^\star{-}\eta}(\Omega)}\leq \varrho
\|\uu\|_{H^2(\Omega)}+C_\varrho\|\uu\|_{L^2(\Omega)}.
\end{equation}
We will also make use of the compact Sobolev embedding
\begin{align}
& \label{hynek}  \Vp\Subset \mathrm{C}^0(\overline\Omega)
 \qquad \text{for $p>d$, with $d \geq 2$}.
 \end{align}
 We conclude with the following Poincar\'{e}-type inequality
 (cf.\ \cite[Lemma 2.2]{gmrs}), with  $m(w)$ the mean value of $w$:
 \begin{equation}
 \label{poincare-type}
 \forall\, q>0 \quad \exists\, C_q >0 \quad \forall\, w \in H^1(\Omega)\, : \qquad
 \| |w|^{q} w \|_{H^1(\Omega)} \leq C_q (\| \nabla (|w|^{q} w )\|_{L^2(\Omega)} + |m(w)|^{q+1})\,.
\end{equation}
\subsection{Assumptions and weak formulations}
\label{ss:assumptions}
We enlist below our basic assumptions on the functions $\heat$,  $\condu$, $W$, $\mathbf{d}$ in system \eqref{eq0}--\eqref{eqII}.

\noindent
\textbf{Hypothesis (I).} We suppose that
\begin{align}
&
\label{hyp-heat}
\begin{aligned}
&
 \text{the function }\heat:[0,+\infty)\to[0,+\infty)\ \text{ is
continuous, and}
\\
&
 \exists\,\sigma_1\ge\sigma
>\frac{2d}{d{+}2},\ c_1\ge c_0>0\ \
\forall\,\teta\in[0,+\infty)\,:\quad c_0(1{+}\teta)^{\sigma-1}\le
\heat(\teta)\le c_1(1{+}\teta)^{\sigma_1-1}\,.
\end{aligned}
\end{align}

\noindent
\textbf{Hypothesis (II).} We assume that
\begin{align}
\label{hyp-K}
\begin{aligned}
&
\text{the function }   \condu:[0,+\infty)\to(0,+\infty)  \  \text{ is
 continuous and}
\\
&
\exists c_2,\,c_3>0  \ \ \forall\teta\in[0,+\infty)\,:\quad
 c_2 \heat({\teta})
\leq \condu(\teta) \leq c_3 (\heat({\teta})+1)\,.
\end{aligned}
\end{align}

\noindent
\textbf{Hypothesis (III).}
We require
\begin{align}
& \label{data-a} a \in \mathrm{C}^1(\R), \ b \in \mathrm{C}^2(\R)  \
\text{ are such that } a(x), \ b(x) \geq 0 \ \text{for all } x \in
[0,1].
\end{align}

\noindent
\textbf{Hypothesis (IV).}
We suppose that the potential $W$ in~\eqref{eqII} is given by $ W=
\widehat{\beta} + \widehat{\gamma},$ where
\begin{align}
 \label{databeta}  & \overline{\text{\rm
dom}(\widehat{\beta})}=[0,1]\,, \quad
 \widehat\beta:  \text{\rm dom}(\widehat{\beta}) \to \R \,
  \text{ is proper, l.s.c.,  convex;}
  \\
  \label{data-gamma}
  &
   \widehat{\gamma} \in
 {\rm C}^{2}(\R).
\end{align}
Hereafter, we shall denote by $\beta=\partial\widehat\beta$ the
subdifferential of $\widehat\beta$, and set $ \gamma:=
\widehat\gamma'$.

\noindent
\textbf{Hypothesis (V).}
We require that
there exists
\begin{equation}
\label{dataphi}
\begin{aligned}
&\text{a Carath\'eodory integrand }   \phi\,:\,\Omega\times \RR^d\to
[0,+\infty)
   \hbox{ such that }   \forae \  x \in \Omega
\\
   & \text{the map }\phi(x,\cdot):  \R^d\to [0,+\infty) \ \text{ is convex, with $\phi(x,0)=0$,  and in
$ \mathrm{C}^1(\RR^d)$}\,,
\end{aligned}
\end{equation}
and, setting $\mathbf{d}:=
\nabla_{\zzeta}\phi\,:\,\Omega\times\RR^d\to\RR^d$,
 the following coercivity and growth conditions hold true:
\begin{equation}\label{datad2}
\exists\, p>d, \ \ c_4,\,c_5,\,c_6>0  \ \ \forae\, x \in \Omega \ \
\forall\, \zzeta \in \R^d\, : \ \
\left\{
\begin{array}{ll} &
\phi(x,{\zzeta})\geq
c_4|{\zzeta}|^p-c_5,
\\
 &  |{\bf d}(x,{\zzeta})|\leq c_6(1+|{
\zzeta}|^{p-1})\,.
\end{array}
\right.
\end{equation}

\noindent
\textbf{A generalization of the $p$-Laplace operator.}
We now  consider the realization in $L^2(\Omega)$ of $\phi$, i.e.
\begin{equation}\label{defiPhi}
\Phi\,:\,L^2(\Omega)\to[0,+\infty],\quad \Phi(\chi):=\begin{cases}
\io\phi(x,\nabla\chi(x)) \dd x & \text{if
$\phi(\cdot,\nabla\chi(\cdot)) \in L^1 (\Omega)$,}
\\
+\infty & \text{otherwise.}
\end{cases}
\end{equation}
Relying on \cite[Thm.~2.5, p.~22]{giaquinta}, it is possible to
prove that $\Phi$ is convex and lower semicontinuous on
$L^2(\Omega)$, with domain $D(\Phi):=\Vp$ (due to \eqref{datad2});
its subdifferential $\partial \Phi: L^2(\Omega) \rightrightarrows
L^2(\Omega)$ is a maximal monotone operator.  In \eqref{eqII} we will  take
 the
 elliptic operator
 \begin{equation}
\label{def-opchi} \opchi:=\partial \Phi:  L^2(\Omega)
\rightrightarrows L^2(\Omega).
\end{equation}
Clearly, $\opchi$ is a generalization of the $p$-Laplace operator.
Note that
\[
  \mathrm{dom}(\opchi):=\left\{v\in
W^{1,p}(\Omega)\,:\,\sup_{w\in D(\Phi)\setminus \{0\}} \left|\io
{\bf d}(x, \nabla v(x))\cdot\nabla w(x)\,
dx\right|/\|w\|_{L^2(\Omega)}<+\infty\right\},
\]
and (cf. \cite[Ex.~2.4]{nsv})
\begin{equation}\label{defAp}
(\opchi v, w)_{L^2 (\Omega)}=\io{\bf d}(x,\nabla v(x))\cdot\nabla
w(x) \dd x\quad\hbox{for all }w\in D(\Phi)\,.
\end{equation}

In order to obtain further regularity and uniqueness results
for system \eqref{eq0}--\eqref{eqII},
we will have to assume that either of the following
additional hypotheses holds true.

\noindent
\textbf{Hypothesis (VI).} We require that  the function  $\phi$ fulfills the
$p$-coercivity condition
\begin{equation}\label{pcoercive}
 \exists\,c_7>0 \qquad \forall\, x \in
\Omega \ \forall\, \zzeta,\eeta\in \RR^d \,: \qquad ({\bf
d}(x,\zzeta)-{\bf d}(x,\eeta), \zzeta-\eeta)\geq c_7
|\zzeta-\eeta|^p,
\end{equation}
and  it is Lipschitz with respect to $x$, viz.
\begin{equation}
\label{Lipschitz-x}
 \exists\, L>0 \quad \forall\, x\,,y \in \Omega \
\forall\, \zzeta \in\R^d\,:  \qquad |\phi(x,\zzeta)-\phi(y,\zzeta) |
\leq L |x-y| (1+|\zzeta|^p).
\end{equation}
\begin{remark}[A regularity result]
\label{rmk:reg-sav}
\upshape It was proved in \cite[Thm.\ 2, Rmk.\ 3.5]{savare98} that, if in
addition  to Hypothesis (V) the function  $\phi$ fulfills Hypothesis (VI),
then
\begin{equation}
\label{sav-reg}
\begin{aligned} &
 \mathrm{dom}(\opchi) \subset W^{1+\sigma,p}(\Omega) \text{ for all
 } 0 <\sigma<\frac1p, \text{ and }
 \\ & \forall\, 0 <\sigma<\frac1p \ \  \exists\,C_\sigma>0 \ \forall\, v \in  W^{1+\sigma,p}(\Omega)\, : \
 \|v\|_{W^{1+\sigma,p}(\Omega)} \leq C_\sigma \| \opchi(v)\|_{L^2(\Omega)}.
 \end{aligned}
\end{equation}
\end{remark}

\noindent
\textbf{Hypothesis (VII).} We assume that
$\phi$  complies with \eqref{Lipschitz-x}
 and with the following \emph{convexity} requirement
 \begin{equation}
 \label{phi-conv}
 \exists\, c_8>0 \ \exists\, \kappa>0  \qquad \forall\, x \in
\Omega \ \forall\, \zzeta,\eeta\in \RR^d, \ \zzeta \neq 0 \,: \qquad
\mathrm{D}_\zzeta^2 \phi(x,\zzeta) \eeta \eeta \geq c_8
(\kappa+|\zzeta|^2) |\eeta|^2.
 \end{equation}
\begin{remark}
\upshape  Assumptions
\eqref{Lipschitz-x} and \eqref{phi-conv} guarantee the validity of
the following inequality (cf.\ \cite{dorothee-thesis} for
a proof)
\begin{equation}
\label{essential-inequality} \exists\,c_9>0 \qquad \forall\, x \in
\Omega \ \forall\, \zzeta,\eeta\in \RR^d \,: \qquad ({\bf
d}(x,\zzeta)-{\bf d}(x,\eeta), \zzeta-\eeta)\geq c_9 (\kappa +
|\zzeta|+ |\eeta|)^{p-2} |\zzeta-\eeta|^2,
\end{equation}
which will play a crucial role in the proof of Thm.\ \ref{teor2}.
\end{remark}
\begin{example}
\upshape \label{ex:operators}
 The two  $p$-Laplacian operators
 \begin{align}
 &
 \label{p-laplacian1}
A_p(\chi):=-\dive(|\nabla\chi|^{p-1}\nabla\chi),  &  p>d,
\\
 &
 \label{p-laplacian2}
\mathsf{A}_p(\chi):=-\dive((1+|\nabla\chi|^2)^{p/2}),  &  p>d,
\end{align}
 are clearly of the form \eqref{defAp}, and comply with
 \eqref{datad2} and
 \eqref{pcoercive}--\eqref{Lipschitz-x} (cf.
 \cite[Ch.~I, 4--(iii)]{dibene}).

 Observe that \eqref{phi-conv} is fulfilled by the
 $p$-Laplacian operator $\mathsf{A}_p$
 \eqref{p-laplacian2}, whereas for the  \emph{degenerate} operator $A_p$ \eqref{p-laplacian1},
 inequality
\eqref{phi-conv}  holds with $\kappa =0$.
 \end{example}
\paragraph{A \emph{nonlocal} alternative to the $p$-Laplacian operator.}
As done in \cite{Knees-Rossi-Zanini}, we could replace the
$p$-Laplacian-type operator
 $\opchi$ \eqref{defAp} in
\eqref{eq2d} with a  \emph{linear} operator, with domain compactly
embedded in $\mathrm{C}^0 (\overline {\Omega})$. More precisely, as
in \cite{Knees-Rossi-Zanini} we could choose
\begin{equation}
\label{slobo-choice} \opchi:= A_s : H^s(\Omega) \to H^s(\Omega)^*
\quad \text{ with $s >\frac{d}{2}$}.
\end{equation}
In \eqref{slobo-choice},  $H^s(\Omega)$ denotes the  Sobolev-Slo\-bo\-de\-ckij space
$W^{s,2}(\Omega)$, endo\-wed with the in\-ner pro\-duct
\[
(z_1,z_2)_{H^s(\Omega)}:= (z_1, z_2)_{L^2(\Omega)}+ a_s (z_1,z_2),
\]
where
\begin{align}
\label{bilinear-s_a} a_s(z_1,z_2):= \int_\Omega
\int_\Omega\frac{\big(\nabla z_1(x) - \nabla
   z_1(y)\big)\cdot \big(\nabla z_2(x) - \nabla
   z_2(y)\big)}{|x-y|^{d + 2 (s - 1)}}\dd x\dd y.
\end{align}
Indeed, since $d\in \{2,3\}$,  we may suppose that $s\in (1,2)$.
Then, we denote by   $A_s : H^s(\Omega) \to H^s(\Omega)^*$
 the associated operator, viz.
\begin{equation}
\label{As} \pairing{}{H^s(\Omega)}{A_s \chi}{w} := a_s(\chi,w) \quad
\text{for every $\chi,\,w \in H^s(\Omega).$}
\end{equation}
Observe that, for $s>d/2$ we have $H^s(\Omega) \Subset \mathrm{C}^0
(\overline {\Omega})$.

\paragraph{Enthalpy transformation.}
We now reformulate  PDE system \eqref{eq0}--\eqref{eqII} in terms of
the \emph{enthalpy} $\ental$, related to the absolute  temperature
$\teta$~via
\begin{equation}
\label{defin-enthalpy} \ental= h(\teta) \quad \text{with } h(r):=
\int_0^{r} \heat(s) \dd s.
\end{equation}
 It follows from \eqref{hyp-heat} that the function $h$ is
strictly increasing on $[0,+\infty)$. Thus, we are entitled to define
\begin{align}\label{K-T}
\Theta(\ental):= \begin{cases} h^{-1}(\ental) &\text{if }\ental\ge0,
\\
0              &\text{if }\ental<0,
\end{cases}
 \qquad\quad K(\ental):=
\frac{\condu(\Theta(\ental))}{\heat(\Theta(\ental))}.
\end{align}

In terms of the enthalpy $\ental$, the PDE system
\eqref{eq0}--\eqref{eqII} rewrites as
\begin{align}
& \ental_t +\chi_t \Theta(\ental) +\rho\Theta(w)\dive(\ub_t)-\dive(K(\ental) \nabla\ental) =
g \quad\hbox{in }\Omega\times (0,T),\label{eq0-ental}
\\
&
\ub_{tt}-\dive(a(\chi)\mathrm{R}_v\tensoret+b(\chi)\mathrm{R}_e\tensore-\rho\Theta(\w)\mathbf{1})={\bf
f}\quad\hbox{in }\Omega\times (0,T),\label{eqI-ental}
\\
&\chi_t +\mu \partial I_{(-\infty,0]}(\chi_t)-\dive({\bf
d}(x,\nabla\chi))+W'(\chi) \ni - b'(\chi)\frac{|\tensore|^2}2 +
\Theta(\ental) \quad\hbox{in }\Omega \times (0,T),\label{eqII-ental}
\end{align}
 supplemented with
the initial and boundary conditions (where $\mathbf{n}$ denotes the outward unit normal to $\partial\Omega$)
\begin{align}
\label{i-c}
&w(0)=w_0, \quad \ub(0)=\ub_0,\quad \ub_t(0)=\vb_0,\quad \chi(0)=\chi_0\qquad     \text{in }\Omega, \\
\label{b-c}
&\partial_\mathbf{n} w=0, \qquad \uu =0, \qquad \partial_\mathbf{n} \chi=0 \qquad    \text{on }\partial\Omega \times
(0,T).
\end{align}
\begin{remark}
\upshape The enthalpy transformation \eqref{defin-enthalpy} was
proposed in \cite{roubicek-SIAM}, and further developed in
\cite{rossi-roubi}, in order to deal with PDE systems where a
 quasilinear internal energy
balance analogous to  \eqref{eq0} is coupled with
\emph{rate-independent} processes. The advantage of  this change of
variables, is that the \emph{nonlinear} term $\heat(\teta) {\teta}_t $
in \eqref{eq0}
 is replaced  by the \emph{linear} contribution
${\w}_t$ in \eqref{eq0-ental}. We will exploit this fact, when
proving the existence of solutions to (an approximation of) system
\eqref{eq0-ental}--\eqref{b-c}  by means of a
time-discretization scheme.
\end{remark}
 For later use, let us observe that Hyp.\ (I) implies
\begin{align}
& \label{conseq-1} \exists\, d_0,\, d_1>0\, \ \
\forall\, \ental \in [0,+\infty)\, : \quad d_1
(\ental^{1/\sigma_1} -1) \leq \Theta(\ental) \leq d_0
(\ental^{1/\sigma} +1)\,,\\
\no
&\hbox{the map } w\mapsto \Theta(w) \quad \hbox{is Lipschitz continuous.}
\end{align}
 A straightforward consequence of the first of \eqref{conseq-1} is that for every $s \in (1,\infty)$
\begin{equation}\label{used-later}\exists\, C_s >0 \ \  \forall\,w\in
L^1(\Omega)\,: \ \  \|\Theta(\w) \|_{L^s (\Omega)} \leq C_s (
\| \w\|_{L^{s/\sigma}(\Omega)}^{1/\sigma} +1).
\end{equation}
Moreover,
  Hypotheses  (I) and (II) entail
\begin{align}
& \label{conseq-2} \exists\, \bar c>0 \ \ \forall\, \ental \in \R\,: \
c_2\leq K(\ental) \leq \bar c.
\end{align}

Finally, in order to deal with the case $\rho\neq0$, we will adopt the following further assumption,
which we directly state in terms of
 the function $K$
 instead of $\mathsf{K}$ and $\mathsf{c}$, in \emph{replacement}
of Hypothesis (II).
\noindent
\textbf{Hypothesis (VIII).} We require that  the function $K$ defined in \eqref{K-T} (where $\heat$
 fulfills Hyp.\ (I) and $\condu:[0, +\infty)\to (0, +\infty)$   is a continuous function)
satisfies
\begin{equation}\label{hyp-K-rho}
\exists c_{10}>0 \  \exists q  \geq  \frac{d+2}{2d}  \quad \forall w\in [0,+\infty) \,:\, \quad
K(w)=c_{10} \left(w^{2q}+1\right).
\end{equation}
Indeed, we could slightly weaken \eqref{hyp-K-rho} by prescribing that
$K$ is bounded from below and above by two functions behaving like $w^{2q}$,
and we have restricted to \eqref{hyp-K-rho} for simplicity only. Let us stress that,
if \eqref{hyp-K-rho} holds, $K$ is no longer bounded from above.

\paragraph{Problem and Cauchy data.}
We suppose that  bulk force $\mathbf{f}$ and the heat source $g$
fulfill
\begin{align}
\label{bulk-force} & \mathbf{f}\in L^2(0,T;\Ha),
\\
 \label{heat-source} &  g \in L^1(0,T;L^1(\Omega)) \cap L^2 (0,T; H^1(\Omega)'),
\end{align}
and that the initial data comply with
\begin{align}
& \label{datoteta} \teta_0 \in L^{\sigma_1}(\Omega) \quad
\text{whence} \quad \ental_0:= h(\teta_0) \in L^1(\Omega)\,,
\\
 \label{datou}
&\ub_0\in \boY,\quad \vb_0\in \boZ\,,
\\
 & \label{datochi}
 \chi_0\in  \mathrm{dom}(\opchi),\quad \widehat{\beta}(\chi_0)\in
L^1(\Omega).
\end{align}

\paragraph{Variational formulation of the non-degenerating system.}
We  now  consider the \emph{non-degenerate}
version of system \eqref{eq0-ental}--\eqref{b-c}:  as already mentioned in the introduction,
 to rule out the elliptic  degeneracy of the momentum equation \eqref{eqI-ental}, it is indeed sufficient
to truncate away from zero
only
the coefficient $a(\chi)$, 
 cf.\ \eqref{eq1d} below.
In
specifying   the variational formulation of the
initial-boundary value problem for the \emph{non-degenerate} system,
due to the
$0$-homogeneity of the operator $\partial I_{(-\infty, 0]}$ \eqref{1-homog} we will just distinguish the
two cases  $\mu=0$ and $\mu=1$. We mention in advance that
 the $L^r(0, T;W^{1,r}(\Omega))$-regularity for
$\w$ derives from \textsc{Boccardo\&Gallou\"{e}t}-type estimates \cite{boccardo-gallouet1} on the enthalpy equation, combined with
the Gagliardo-Nirenberg inequality \eqref{gn-ineq}. We refer to the forthcoming Sec.\ \ref{ss:3.2} and to
\cite{roubicek-SIAM} for all details.
\begin{problem}
\label{prob}
 Given $\delta>0$,  $\mu \in \{0,1\}$, find 
 functions
\begin{align}
\label{reg-ental}  & w \in L^r(0,T;W^{1,r}(\Omega)) \cap L^\infty
(0,T;L^1(\Omega)) \cap \mathrm{BV}([0,T];W^{1,r'}(\Omega)^*)\quad
\text{for every } 1 \leq r <\frac{d+2}{d+1},
\\
& \label{reg-u} \uu \in    H^1(0,T;H_0^{2}(\Omega;\R^d)) \cap
W^{1,\infty} (0,T;\boZ)  \cap H^2 (0,T;L^2(\Omega;\R^d)),
\\
& \label{reg-chi} \chi \in L^\infty (0,T;W^{1,p} (\Omega)) \cap H^1
(0,T;L^2 (\Omega)),
\end{align}
fulfilling
 the initial conditions
\begin{align}
 \label{iniu}  & \uu(0,x) = \uu_0(x), \ \ \uu_t (0,x) = \vb_0(x) & \forae\, x \in
 \Omega,
 \\
 \label{inichi}  & \chi(0,x) = \chi_0(x) & \forae\, x \in
 \Omega,
\end{align}
 the equations
\begin{align}
&
\begin{aligned}
\label{eq0d}   & \io \varphi(t) \,\ental (t)(\mathrm{d}x) -\int_0^t\io \ental \varphi_t \dd x \dd s   +\itt \io
\chi_t \Theta(\ental)\varphi \dd x \dd s + \rho\itt \io
\hbox{\rm div}(\ub_t) \Theta(\ental)\varphi \dd x \dd s \\
&\qquad \qquad +\itt \io K(\ental)
\nabla \ental\nabla\varphi \dd x \dd s   =\itt \io g \varphi
 +\io \ental_0\varphi(0)\dd x\\
 &\quad \text{for all }
\varphi\in \mathscr{F}:=\mathrm{C}^{0}([0,T]; W^{1,r'}(\Omega)) \cap
W^{1, r'}(0,T; L^{r'}(\Omega))  \text{ and for all } t\in (0,T],
\end{aligned}
\\
 \label{eq1d}
&\ub_{tt}+\opj{(a(\chi)+\delta)}{\ub_t}+\oph{b(\chi)}{\uu} +  \ciro( \Theta(w)) =\mathbf{f}
\quad  \text{in $H^{-1}(\Omega;\R^d)$ \quad  a.e.\  in } (0,T),
\end{align}
 and the subdifferential inclusion
\begin{align}
&\label{eq2d}
\begin{aligned}
 \! \!\!\!\!\!\! \chi_t + \mu \partial I_{(-\infty, 0]}(\chi_t) + \opchi(\chi)+\beta(\chi)+ \gamma(\chi)
  \ni-b'(\chi)\frac{\varepsilon(\ub)
\mathrm{R}_e\varepsilon(\ub)}{2}+\Theta(\ental)  \text{ in }
W^{1, p}(\Omega)^*  \text{ a.e.\ in $(0,T)$}.
\end{aligned}
\end{align}
\end{problem}
\begin{remark}
\label{w-Radon} \upshape Since $\w\in\mathrm{BV} ([0,T];W^{1,r'}(\Omega)^*)$,
 for all $t \in [0,T]$ one has
 $\w(t) \in W^{1,r'}(\Omega)^*$. Combining this with the fact that   $\w\in
L^\infty(0,T;L^1(\Omega))$, we have that $\w(t)$ is a Radon measure
on $\Omega$ for all $t \in [0,T]$, which justifies the notation in
the first integral term on the left-hand side of \eqref{eq0d}. {Moreover, let us note that
the $\BV$-regularity w.r.t.\ time of the absolute temperature, which is mainly  due to the presence
of  quadratic nonlinearities $\chi_t\Theta(w)$ and
$\dive(\ub_t)\Theta(w)$ in \eqref{eq0d}, is quite natural for this kind of problems (cf., e.g. \cite{boccardo-gallouet1}, \cite{fpr09}, and \cite{roubicek-SIAM}).}
\end{remark}
\begin{remark}
\label{rem-other-b.c.} \upshape {The proof of our results could be carried out with
suitable modifications  in the case of Neumann boundary conditions
on $\uu$, as well. We would also be able to handle the case of
Neumann conditions on a portion $\Gamma_0$ of $\partial \Omega$  and
Dirichlet conditions on $\Gamma_1:=\partial \Omega \setminus
\Gamma_0$ ($|\Gamma_0|, \, |\Gamma_1|>0$), provided that the
closures of the sets $\Gamma_0$ and $\Gamma_1$ do not intersect.
 Indeed, without the latter  geometric condition,
  the elliptic regularity results ensuring the (crucial)
 $\boY$-regularity of $\uu$ may fail to hold, see~\cite[Chap.~VI,~Sec.~6.3]{ciarlet}.}
\end{remark}
 \noindent In what follows, we will refer to system
 \eqref{eq0d}--\eqref{eq2d} with $\mu=0$ (with $\mu=1$, respectively), as
 the (non-degenerating)
\emph{reversible full system} (\emph{irreversible full system},
resp). In both cases $\mu=0$ and $\mu=1$, we will call
\emph{isothermal} the  (non-degenerating)  system
\eqref{eq1d}--\eqref{eq2d}, where $\Theta(\w)$ in \eqref{eq2d} is
replaced  by a given temperature
profile~$\Theta^*$.

\subsection{Global existence and uniqueness results for the reversible system}
\label{ss:glob-rev}
Our first main result states the existence of a solution
$(\w,\uu,\chi)$ to
 the reversible full system under  Hypotheses (I)--(V); under the further
  Hypothesis (VI)  we are able to obtain some enhanced regularity for $\chi$.
Its proof will be developed in
Section \ref{sec:proof2} by passing to
the limit in the time-discretization scheme set up in Sec.\
\ref{s:time-discrete}.
\begin{maintheorem}[Global existence for the full system, $\mu=0$, $\rho=0$]
\label{teor1} Let  $\mu=0$, $\rho=0$, and assume   Hypotheses (I)--(V)
and  condi\-tions \eqref{bulk-force}--\eqref{datochi} on the data $\mathbf{f}, \, g,
\teta_0,\, \uu_0, \, \vv_0, \, \chi_0,$. Then,
\begin{compactenum}
\item[1.]
 Problem
\ref{prob} admits a solution $(\ental,\uu,\chi)$, such that there
exists
\begin{align}
& \label{xi-qualification} \xi \in L^2 (0,T;L^2(\Omega)) \text{ with
} \xi(x,t) \in \beta(\chi(x,t)) \ \forae\,(x,t)\in \Omega\times
(0,T), \text{ fulfilling }
\\ &
\label{eq-with-xi}
 \chi_t + \opchi(\chi)+\xi+ \gamma(\chi)
=-b'(\chi)\frac{\varepsilon(\ub)
\mathrm{R}_e\varepsilon(\ub)}{2}+\Theta(\ental) \qquad \text{a.e.~in
}\ \Omega \times (0,T).
\end{align}
Furthermore, $(\ental,\uu,\chi)$ satisfies the \emph{total
energy equality}
\begin{equation}
\label{total-energy-eq}
\begin{aligned}
 &
\int_\Omega w(t)(\mathrm{d} x) +\frac12 \int_\Omega |\uu_t (t)|^2\,
\mathrm{d} x +\int_s^t \int_\Omega |\chi_t|^2 \, \mathrm{d} x\,
\mathrm{d} r \\  & \quad + \int_s^t
\bilj{(a(\chi)+\delta)}{\uu_t}{\uu_t} \, \mathrm{d} r +\frac12
\bilh{b(\chi(t))}{\uu(t)}{\uu(t)} + \Phi(\chi(t)) +\int_\Omega
W(\chi(t)) \, \mathrm{d} x
\\ &
 = \int_\Omega w(s) (\mathrm{d} x) +\frac12 \int_\Omega |\uu_t
(s)|^2\, \mathrm{d} x +\frac12 \bilh{b(\chi(s))}{\uu(s)}{\uu(s)} +
\Phi(\chi(s)) +\int_\Omega W(\chi(s)) \, \mathrm{d} x
\\ & \qquad \qquad \qquad
+\int_s^t \int_\Omega \mathbf{f} \, \cdot \, \uu_t \, \mathrm{d} x
\mathrm{d} r + \int_s^t \int_\Omega g \, \mathrm{d} x\mathrm{d} r  \qquad
\text{for all } 0 \leq s \leq t \leq T.
\end{aligned}
\end{equation}
\item[2.]
 If, in addition, $\phi$ complies with
 Hypothesis (VI),  then there holds
\begin{align}
& \label{furth-reg-chi} \chi \in L^2(0,T;W^{1+\sigma,p}(\Omega))
\quad \text{for all } 0<\sigma<\frac1p.
\end{align}
 \item[3.] Suppose  that
\begin{equation}
\label{g-pos} g(x,t) \geq 0 \quad \forae\, (x,t) \in \Omega \times
(0,T).
\end{equation}
Then, $w\geq 0$ a.e. in $\Omega\times (0,T)$, hence
\begin{equation}
\label{teta-non-neg} \teta(x,t):= \Theta(w(x,t)) \geq 0 \quad
\forae\, (x,t) \in \Omega \times (0,T).
\end{equation}
\end{compactenum}
\end{maintheorem}
\noindent We are going to prove the energy equality \eqref{total-energy-eq}
 by testing \eqref{eq0d} by $\varphi \equiv 1$, \eqref{eq1d} by
 $\uu_t$, \eqref{eq-with-xi} by $\chi_t$, adding the resulting
 relations, integrating  in time, and developing the calculations at
 the end of the proof of Thm.\ \ref{teor1} in  Sec.\
 \ref{ss:4.1}.

 We now turn to the case when the thermal expansion coefficient
$\rho \neq 0$. As previously mentioned, to prove existence of solutions
 we need to replace Hyp.\ (II) with Hyp.\ (VIII), which has a key role in deriving the
 enhanced regularity \eqref{reg-w-rho} for $\w$, cf.\ Remark \ref{rmk:afterThm2} later on.
\begin{maintheorem}[Global existence for the full system, $\mu=0$,  $\rho\neq 0$]
\label{teor1bis} Let  $\mu=0$, $\rho\neq 0$, and assume   Hypotheses (I) and  (III)--(V)
and  condi\-tions \eqref{bulk-force}--\eqref{datochi} on the data $\mathbf{f}, \, g,
\teta_0,\, \uu_0, \, \vv_0, \, \chi_0,$. Suppose moreover that  Hypothesis (VIII)
 is satisfied (in place of  Hypothesis (II)), and that
 \begin{equation}
 \label{w_0-hyp-forte}
w_0\in L^2(\Omega).
\end{equation}
Then,
\begin{compactenum}
\item[1.]
 Problem
\ref{prob} admits a solution $(\ental,\uu,\chi)$
fulfilling
\eqref{xi-qualification}-\eqref{eq-with-xi}, such that
 $w$ has the further regularity
\begin{equation}\label{reg-w-rho}
\begin{aligned}
w\in L^2(0,T; H^1(\Omega)) & \cap L^{2(q+1)} (0,T;L^{6(q+1)}(\Omega)) \cap
 L^\infty(0,T; L^2(\Omega))\\ & \cap  \mathrm{BV}([0,T];
W^{2,\expo}(\Omega)') \quad \text{with } \expo=\frac{6q+6}{4q+5},
\end{aligned}
\end{equation}
 and  the weak formulation
 of the enthalpy equation
 holds in the form
 \begin{equation}
 \begin{aligned}
\label{eq0d-new}
 \io \varphi(t) \,\ental (t)(\mathrm{d}x) -\int_0^t\io \ental \varphi_t \dd x \dd s   &  +\itt \io
\chi_t \Theta(\ental)\varphi \dd x \dd s + \rho\itt \io
\hbox{\rm div}(\ub_t) \Theta(\ental)\varphi \dd x \dd s
\\
& \quad
 +\itt \int_\Omega \widehat{K}(\ental)
A\varphi \dd x \dd s   =\itt \io g \varphi
 +\io \ental_0\varphi(0)\dd x
 \end{aligned}
 \end{equation}
for all test functions
$\varphi\in \mathscr{F}':=\mathrm{C}^0([0,T]; W^{2,\expo}(\Omega))\cap H^1(0,T; L^{6/5}(\Omega)) $,
 where $\widehat{K}(w)= \frac 1{2q+1} w^{2q+1}$ is  a primitive of $K$.
  Moreover,
 $(\ental,\uu,\chi)$ complies with the \emph{total
energy equality} \eqref{total-energy-eq}.
\item[2.]
 If, in addition, $\phi$ complies with
 Hypothesis (VI),  then  the further regularity result \eqref{furth-reg-chi} holds true.
\item[3.]
 If  in addition $g$ complies with \eqref{g-pos}, then \eqref{teta-non-neg}
holds.
\end{compactenum}
\end{maintheorem}
\begin{remark}[Outlook to the enhanced regularity \eqref{reg-w-rho}]
\upshape
\label{rmk:afterThm2}
Let us justify the additional regularity \eqref{reg-w-rho}
for $\w$, by developing
on a purely \emph{formal} level, enhanced estimates on the enthalpy equation \eqref{eq0-ental},
based on the stronger Hypothesis (VIII). Indeed, we (formally) choose $\varphi=\w$ as a test
function for \eqref{eq0d}: re-integrating by parts in time and exploiting \eqref{hyp-K-rho}
we obtain for any $t \in (0,T)$:
\begin{equation}
\label{est-9.1-rmk}
\begin{aligned}
&
\frac12\io |\w(t)|^2\dd x +c_{10}\int_0^t\io\left(|\w|^{2q}+1\right)|\nabla \w|^2 \dd x \dd s
\\ & \leq
\frac12 \int_{\Omega}|\w(0)|^2 \dd x + \int_0^t \io |g||\w|\dd x  + \int_0^t \io \left(|\chi_t|+|\rho||\dive(\uu_t)|\right)|\Theta(\w)||\w|\dd x\dd s\,.
\end{aligned}
\end{equation}
Now, we observe that $\int_\Omega |\w|^{2q} |\nabla \w|^2 \dd x=  1/ (q+1)^2 \int_\Omega
 |\nabla (|\w|^q \w)|^2 \dd x $ and that,
 due to the Poincar\'{e} inequality \eqref{poincare-type} and
 to the fact that $\w \in L^\infty (0,T;L^1(\Omega))$, there holds
 \[
\| |\w|^q w \|_{H^1(\Omega)}^2 \leq C (\|\nabla (|\w|^q \w)\|_{L^2(\Omega)}^2 +1 )\,.
 \]
 Therefore,
 taking into account the continuous embedding $H^1(\Omega) \subset L^6(\Omega)$,
   for the l.h.s.\ of \eqref{est-9.1-rmk} we have the lower bound
 \begin{equation}
 \label{est-9.2-rmk}
 \int_0^t \io\left(|\w|^{2q}+1\right)|\nabla \w|^2 \dd x\dd s  \geq
c \int_0^t  \left(\|\nabla
\w\|_{L^2(\Omega)}^2+\|w\|^{2(q+1)}_{L^{6(q+1)}(\Omega)}\right) \dd s -C.
 \end{equation}
 Clearly, relying on \eqref{heat-source} we can absorb the second term on the r.h.s.\ of
 \eqref{est-9.1-rmk} into its left-hand side.
 On the other hand, using the fact that
 $\ell:=\chi_t + \dive (\uu_t) \in L^2(0,T;L^2(\Omega))$ and taking into account the growth
 \eqref{conseq-1} of $\Theta$, we can estimate the last summand on the r.h.s.  by
 \begin{equation}
 \label{est-9.3.-rmk}
\begin{aligned}
d_0 \int_0^t\int_\Omega |\ell|(w^{1+1/\sigma} {+}1) \dd x \dd s  & \leq
C \int_0^t\int_\Omega |\ell|(w^{q+1} {+}1) \dd x \dd s
\\ &
\leq \varrho \int_0^t \|w\|^{2(q+1)}_{L^{2(q+1)}(\Omega)} \dd s  + C_\varrho\left(
\int_0^t  \| \ell \|_{L^2(\Omega)}^2+1 \right),
\end{aligned}
 \end{equation}
 where the first inequality follows from the fact that $1+1/\sigma < (3d+2)/(2d) \leq q+1$
 thanks to
 \eqref{hyp-heat} and \eqref{hyp-K-rho}, and $\varrho>0$ is chosen sufficiently small, in such a way as to absorb
 $\int_0^t \|w\|^{2(q+1)}_{L^{2(q+1)}(\Omega)} \dd s$ into the r.h.s.\
 of \eqref{est-9.2-rmk}.
 Plugging \eqref{est-9.2-rmk} and \eqref{est-9.3.-rmk} into \eqref{est-9.1-rmk}
we immediately deduce an estimate for $\w$ in the space $w\in L^2(0,T; H^1(\Omega))\cap L^{2(q+1)} (0,T;L^{6(q+1)}(\Omega)) \cap
 L^\infty(0,T; L^2(\Omega))$. Observe that, as a consequence of $\w \in L^{2(q+1)} (0,T;L^{6(q+1)}(\Omega))$,  we have
$w^{2q+1} \in L^{1+1/\epsilon} (0,T;L^{3+3/\epsilon}(\Omega))  $, with
$\epsilon= 2q+1$. Therefore, 
\[
A (\widehat{K}(w)) \text{ is estimated in } L^{1+1/\epsilon} (0,T;
W^{2,\expo}(\Omega)'),
\]
with $\expo= (6q+6)/(4q+5)$ the conjugate exponent of $3+3/(2q+1)$.
A comparison in \eqref{eq0-ental}  entails an estimate for $w_t \in L^1 (0,T;
W^{2,\expo}(\Omega)')$, and we conclude \eqref{reg-w-rho}.

In fact, the $\BV$-estimate for $\w$ (and accordingly, the regularity required of the test functions)
could be slightly improved by resorting to
refined interpolation arguments: however, to avoid overburdening this exposition we choose
not to detail this point.
\end{remark}
\noindent
To prove Thm.\ \ref{teor1bis}, we will need to
combine the time-discretization procedure for system \eqref{eq0d}--\eqref{eq2d}, with a truncation
of the function $K$ in the elliptic operator of \eqref{eq0d}, cf.\ Problem \ref{prob:rhoneq0}
 later on. Hence, in order to make the
estimates
developed in Rmk.\ \ref{rmk:afterThm2} rigorous, we will have to pass to the limit in two phases,
first with the time-step, and then with the truncation parameter, cf.\ the discussion in
Sec.\ \ref{sec:4.2new}.

For the \emph{isothermal} reversible system,  in both cases
$\rho =0$ and $\rho \neq 0$,
 we obtain a continuous dependence result, in particular
yielding uniqueness of solutions, under the additional convexity property for $\phi$ in  Hypothesis (VII).
Indeed, the latter ensures the monotonicity inequality \eqref{essential-inequality} for $\mathbf{d}$,
which is crucial for the continuous dependence estimate. We also need to restrict to the case
in which
$a$ is constant.
\begin{maintheorem}[Continuous dependence on  the data for the isothermal system, $\mu=0$, $\rho\in \RR$]
\label{teor2} Let $\mu=0$, $\rho\in \RR$. As\-su\-me that Hypotheses
 (III)--(V) and (VII)   are satisfied, and, in addition, that
\begin{equation}
\label{constant-a} \text{the function $ a $   is
constant.}
\end{equation}
 Let $(\mathbf{f}_i,\ub_0^i, \vb_0^i,\chi_0^i)$, $i=1,2$, be two
sets of data complying with
\eqref{bulk-force} and
\eqref{datou}--\eqref{datochi},
 and, accordingly, let
$(\ub_i,\chi_i)$, $i=1,2$, be the associated solutions on some
$[0,T]$ with fixed temperature profiles $\Theta(w_i)=\bar \Theta_i\in L^2(0,T; L^2(\Omega))$.
Set
$
\mathrm{M}:=  \max_{i=1,2} \left\{
\| \ub_i \|_{H^{1}(0,T;\boY)}+1 \right\}\,.
$
Then there exists a positive constant $S_0$, depending on $\mathrm{M}$,
 $\delta$, $T$, and  $|\Omega|$, such that
 \begin{equation}
 \label{cont-depe}
\begin{aligned}
 & \| \uu_1 -  \uu_2 \|_{W^{1,\infty}(0,T; \Ha) \cap H^1 (0,T; H^1(\Omega;\R^d))}
+ \| \chi_1 -\chi_2 \|_{L^\infty(0,T; L^2(\Omega)) \cap L^p(0,T; \Vp)} \\
&\leq S_0 \Big( \| \uu_0^1 -\uu_0^2 \|_{H^1(\Omega;\R^d)} +
\|\vv_0^1 -\vv_0^2 \|_{L^2(\Omega;\R^d)} +
 \| \chi_0^1 -\chi_0^2 \|_{L^2(\Omega)}
 \\
 & \qquad \qquad \qquad +
 \| \mathbf{f}_1 - \mathbf{f}_2 \|_{L^2 (0,T;
H^{-1}(\Omega))}+ \|\bar\Theta_1-\bar\Theta_2\|_{L^2(0,T; L^2(\Omega))}\Big)\,.
\end{aligned}
\end{equation}
In particular, the isothermal  reversible system admits a unique solution
$(\uu,\chi)$.
\end{maintheorem}
The proof is
postponed to Section \ref{CD}.
\begin{remark}\label{aconst-sLapl}
\upshape
Let us notice that, if we consider the $s$-Laplacian \eqref{bilinear-s_a} instead of the $p$-Laplacian
\eqref{defAp}, the continuous dependence result stated in Theorem~\ref{teor2} still holds true also without
assumption \eqref{constant-a}. In this case,
for any two solutions $\chi_1$ and $\chi_2$ of the isothermal reversible system,
\eqref{cont-depe} yields an estimate on
 $ \| \chi_1 -\chi_2 \|_{L^2(0,T; H^s(\Omega))}$.
For further details, we refer to   Remark~\ref{sLaplCD} at the end of Sec.~\ref{CD}.
\end{remark}

\subsection{Global existence  results for the  irreversible system}
\label{glob-irrev} \noindent \textbf{Heuristics for weak solutions.}
As mentioned in the introduction, the major problem in dealing with the subdifferential inclusion
\eqref{eq2d} in the case $\mu=1$ is the simultaneous presence of the
three nonlinear, maximal monotone operators $\partial I_{(-\infty, 0]}$, $\opchi$, and
$\beta$, which need to be properly identified when passing to the
limit in the time-discretization scheme we are going to set up in
Section \ref{s:time-discrete}. We now discuss the attached
difficulties on a formal level, treating $\beta$ and  $\partial I_{(-\infty, 0]}$ as
single-valued.

 It would be  possible to handle
$\beta=\partial\widehat{\beta}$ by exploiting the strong-weak
closedness (in the sense of graphs) of the (induced operator)
$\beta: L^2 (0,T;L^2(\Omega)) \rightrightarrows L^2
(0,T;L^2(\Omega))$.
Nonetheless, a $L^2
(0,T;L^2(\Omega))$-estimate
of  the term $\beta(\chi)$ in
\eqref{eq2d} cannot be obtained without estimating as well
$\opchi(\chi)$ (and hence $\partial I_{(-\infty, 0]}(\chi_t)$ by comparison), in $L^2
(0,T;L^2(\Omega))$. To our knowledge, this can be proved  by
testing \eqref{eq2d} by $\partial_t (\opchi(\chi)+\beta(\chi))$ (cf.\ also \cite{bss}). The
related calculations (which we will develop in Sec.\
\ref{s:time-discrete}, on the time-discrete level, for the
\emph{isothermal} irreversible system) would involve an integration
by parts of the terms on the right-hand side of \eqref{eq2d}. Thus,
they would rely on an estimate in $W^{1,1}(0,T;L^2(\Omega))$ of the
term $-b'(\chi)\frac{\varepsilon(\ub)
\mathrm{R}_e\varepsilon(\ub)}{2}+\Theta(\ental)$. However, presently
this enhanced bound  for $\Theta(w)$ does not seem to be at hand due to
the poor time-regularity of $w$, cf.\ \eqref{reg-ental}.

That is why, for the temperature-dependent irreversible system we
are only able to obtain the existence of solutions  $(w, \ub, \chi)$ to a suitable
\emph{weak formulation} of \eqref{eq2d}, mutuated from \cite{hk1},
where we also restrict to the particular case in which
\begin{equation}
\label{e:particular-widehatbeta} \widehat{\beta}=I_{[0,+\infty)}.
\end{equation}
\begin{remark}
\upshape In the present irreversible context it is sufficient to
choose $\widehat\beta$ as in \eqref{e:particular-widehatbeta} to
enforce the constraint $\chi \in [0,1]$ a.e.\ in $\Omega\times
(0,T)$.
 Indeed,  starting
from an initial datum $\chi_0 \leq 1$ a.e.\ in $\Omega$ we will
obtain by irreversibility that
 $\chi(\cdot,t) \leq \chi_0 \leq 1$
  a.e.\ in $\Omega$, for almost all $t \in (0,T)$.
\end{remark}

 The underlying motivation
for the weak formulation of \eqref{eq2d} we will consider
 is that, due to the $1$-homogeneity
 of $I_{(-\infty, 0]}$, it is not difficult to check
that \eqref{eq2d} is equivalent to the system
\begin{subequations}
\label{ineq-system}
\begin{align}
 \label{ineq-system1}  & \chi_t(x,t) \leq 0 \qquad \foraa\, (x,t)
 \in \Omega \times (0,T),
 \\
 \label{ineq-system2}  &
\begin{aligned}
   & \pairing{}{W^{1,p}(\Omega)}{\chi_t(t) +\opchi(\chi(t))+
 \xi(t)+ \gamma(\chi(t))  + b'(\chi(t))\frac{\varepsilon(\ub(t))
\mathrm{R}_e\varepsilon(\ub(t))}{2} -\Theta(w(t))}{\varphi} \geq 0 \\
&  \qquad \qquad \qquad \qquad \qquad  \qquad \qquad \qquad \qquad
\text{for all } \varphi \in W_-^{1,p}(\Omega) \ \foraa\, t\in (0,T),
\end{aligned}
\\
\label{ineq-system3}
 &
\begin{aligned}
   &
 \pairing{}{W^{1,p}(\Omega)}{\chi_t(t) +\opchi(\chi(t))+
 \xi(t) + \gamma(\chi(t))  + b'(\chi(t))\frac{\varepsilon(\ub(t))
\mathrm{R}_e\varepsilon(\ub(t))}{2} -\Theta(w(t))}{\chi_t(t)} \leq 0\\
&  \qquad \qquad \qquad \qquad \qquad  \qquad \qquad \qquad \qquad \
\foraa\, t\in (0,T),
\end{aligned}
\end{align}
\end{subequations}
with $\xi \in \partial I_{[0,+\infty)}(\chi)$ a.e.\ in $\Omega
\times (0,T)$ (and $\pairing{}{W^{1,p}(\Omega)}{\cdot}{\cdot}$ denoting the duality pairing
 between $W^{1,p}(\Omega)'$ and $W^{1,p}(\Omega)$, cf.\ Notation \ref{not:2.1}).  In  order to see this, it is sufficient to
 subtract
 \eqref{ineq-system3} from \eqref{ineq-system2}, and use the definition of $\partial I_{[0,+\infty)}$.
 However, for reasons analogous to those mentioned in the above
lines, the proof of \eqref{ineq-system3} is at the moment an open
problem. Therefore, following \cite{hk1}, in  the forthcoming
Definition \ref{def-weak-sol} we weakly formulate
\eqref{ineq-system} by means of \eqref{ineq-system1},  (an
integrated version of) \eqref{ineq-system2}, and the \emph{energy
inequality} \eqref{energ-ineq} below, in place of
\eqref{ineq-system3}.
\begin{definition}[Weak solution to the  (non-degenerating) irreversible full system]
\label{def-weak-sol} Let $\mu=1$. We call a triple $(w,\uu,\chi) $
as in \eqref{reg-ental}--\eqref{reg-chi}  a \emph{weak solution} to
Problem \ref{prob} if, besides fulfilling the weak enthalpy and
momentum equations \eqref{eq0d}--\eqref{eq1d}, it satisfies
$\chi_t(x,t) \leq 0$ for almost all $(x,t) \in \Omega \times (0,T)$,
as well as
\begin{equation}
 \label{ineq-system2-integrated}
\begin{aligned}
  \int_\Omega  \Big( \chi_t (t) \varphi     +
\mathbf{d}(x,\nabla\chi(t)) \cdot \nabla \varphi  + \xi(t) \varphi +
\gamma(\chi(t)) \varphi 
 & + b'(\chi(t))\frac{\varepsilon(\ub(t))
\mathrm{R}_e\varepsilon(\ub(t))}{2}\varphi  -\Theta(w(t)) \varphi
\Big) \,
\mathrm{d}x 
  \geq 0 \\ &  \text{for all }  \varphi
\in 
W_-^{1,p}(\Omega), \quad \foraa\, t \in (0,T),
\end{aligned}
\end{equation}
with $\xi \in
\partial I_{[0,+\infty)}(\chi)$ in the following sense:
\begin{equation}
\label{xi-def} \xi \in L^1(0,T;L^1(\Omega)) \qquad\text{and}\qquad
\pairing{}{W^{1,p}(\Omega)}{\xi(t)}{\varphi-\chi(t)} \leq 0 \  \
\forall\, \varphi \in W_+^{1,p}(\Omega), \ \foraa\, t \in (0,T),
\end{equation}
 and the energy
inequality for all $t \in (0,T]$, for $s=0$,  and for almost all $0
< s\leq t$:
\begin{equation}
\label{energ-ineq}
\begin{aligned}
 \int_s^t   \int_{\Omega} |\chi_t|^2 \dd x \dd r  & +
 \Phi(\chi(t)) +  \int_{\Omega} W(\chi(t))\dd x\\ & \leq
 \Phi(\chi(s))+ \int_{\Omega}W(\chi(s))\dd x
  +\int_s^t  \int_\Omega \chi_t \left(- b'(\chi)
  \frac{\varepsilon(\ub)\mathrm{R_e}\varepsilon(\ub)}2
+\Theta(\ental)\right)\dd x \dd r.
\end{aligned}
\end{equation}
\end{definition}
\noindent The following result sheds light on the properties of this solution concept.
First of all, it states the \emph{total energy
inequality} \eqref{total-energy-ineq}:  from \eqref{total-energy-ineq},  we will deduce
in Sec.\ \ref{sec:5}
 suitable
estimates independent of $\delta$, which  will allow us to pass to
the limit in Problem \ref{prob} as $\delta \downarrow 0$ for $\mu=1$.
Furthermore,  the second part
of Proposition \ref{more-regu} (whose proof closely follows the
argument for  \cite[Prop.\ 4.1]{hk1}) shows that, if $\chi$ is
regular enough, then \eqref{ineq-system1} and
\eqref{ineq-system2-integrated}--\eqref{energ-ineq} are equivalent to \eqref{eq2d}.
\begin{proposition}\label{more-regu}
Let $\mu=1$. Then, any weak solution $(w,\uu,\chi)$ in
the sense of Def.\ \ref{def-weak-sol} fulfills
the  \emph{total energy inequality}  for all $t \in (0,T]$,
for $s=0$,  and for almost all $0 < s\leq t$
\begin{equation}
\label{total-energy-ineq}
\begin{aligned}
 &
\int_\Omega w(t)(\mathrm{d} x) +\frac12 \int_\Omega |\uu_t (t)|^2\,
\mathrm{d} x +\int_s^t \int_\Omega |\chi_t|^2 \, \mathrm{d} x\,
\mathrm{d} r \\  & \quad + \int_s^t
\bilj{(a(\chi)+\delta)}{\uu_t}{\uu_t} \, \mathrm{d} r +\frac12
\bilh{b(\chi(t))}{\uu(t)}{\uu(t)} + \Phi(\chi(t)) +\int_\Omega
W(\chi(t)) \, \mathrm{d} x
\\ &
 \leq \int_\Omega w(s)(\mathrm{d} x) +\frac12 \int_\Omega |\uu_t
(s)|^2\, \mathrm{d} x +\frac12 \bilh{b(\chi(s))}{\uu(s)}{\uu(s)} +
\Phi(\chi(s)) +\int_\Omega W(\chi(s)) \, \mathrm{d} x
\\ & \qquad \qquad \qquad
+\int_s^t \int_\Omega \mathbf{f} \, \cdot \, \uu_t \, \mathrm{d} x
\mathrm{d} r + \int_s^t \int_\Omega g \, \mathrm{d} x\mathrm{d} r \,.
\end{aligned}
\end{equation}
Assume now Hypotheses  (III)--(V).
 Let $(w,\uu,\chi)$ be as in
\eqref{reg-ental}--\eqref{reg-chi}, and suppose in addition that
\begin{equation}
\label{reg-chain-rule}
\begin{aligned}
&\opchi(\chi) \in L^2 (0,T;L^2(\Omega)) \text{
and there exists } \xi \in L^2 (0,T; L^2(\Omega)) \text{ with }
\\
& \xi (x,t) \in
\partial I_{[0,+\infty)}(\chi(x,t)) \ \foraa\, (x,t)\in \Omega \times (0,T),
\end{aligned}
\end{equation}
such that $(\w,\uu,\chi,\xi)$ comply with \eqref{ineq-system1} and
\eqref{ineq-system2-integrated}--\eqref{energ-ineq}.   Then,
\begin{equation}
\label{eq-with-zeta}
\begin{gathered}
\exists\, \zeta  \in L^2(0,T;L^2(\Omega)) \text{ with } \zeta(x,t)
\in \partial I_{(-\infty, 0]}(\chi_t (x,t)) \ \foraa\, (x,t)\in \Omega \times (0,T)
\text{ s.t. }
\\
 \chi_t +\zeta+  \opchi(\chi)+\xi+ \gamma(\chi)
=-b'(\chi)\frac{\varepsilon(\ub)
\mathrm{R}_e\varepsilon(\ub)}{2}+\Theta(w) \qquad \text{a.e.~in }\
\Omega \times (0,T).
\end{gathered}
\end{equation}
\end{proposition}
\begin{proof}
In order to prove \eqref{total-energy-ineq} it is sufficient choose $\varphi \equiv
1$ in \eqref{eq0d}, test \eqref{eq1d} by $\uu_t$, integrate in time,
perform the calculations in the proof of Thm.\ \ref{teor1},  and add
the resulting equalities with the energy inequality
\eqref{energ-ineq}. The second part of the statement can be proved considering
the energy functional
\begin{equation}
\label{def-mathscrE} \mathscr{E}: L^2(\Omega) \to (-\infty,+\infty],
\qquad \mathscr{E}(\chi):= \Phi(\chi)+\int_\Omega W(\chi)\dd x.
\end{equation}
It follows from  Hypotheses (IV) and (V), as well as the chain rule
of \cite[Lemma 3.3]{brezis} that, if $\chi$ complies with
\eqref{reg-chi} and \eqref{reg-chain-rule}, then  the map $t\mapsto
 \mathscr{E}(\chi(t))$ is absolutely continuous on $(0,T)$ and
 fulfills
 \begin{equation}
\label{form-derivative} \frac{\dd}{\dd t} \mathscr{E}(\chi(t)) =
\int_\Omega  \left(\opchi(\chi(x,t))  + \xi(x,t)  +
\gamma(\chi(x,t)) \right) \chi_t(x,t) \dd x \qquad \foraa\, t \in (0,T).
\end{equation}
 Therefore, differentiating \eqref{energ-ineq} in time and
 using \eqref{form-derivative} we conclude that
$(\w,\uu,\chi,\xi)$ comply with \eqref{ineq-system3},  where the duality
pairing $\pairing{}{W^{1,p}(\Omega)}{\chi_t}{\chi_t}$
 is replaced by the scalar product in $(\chi_t,\chi_t)_{L^2(\Omega)}$.
 Likewise,
\eqref{ineq-system2-integrated} yields \eqref{ineq-system2}.
Again on account of \eqref{reg-chi} and \eqref{reg-chain-rule}, it
is not difficult to infer from \eqref{ineq-system} that $-
b'(\chi)\frac{\varepsilon(\ub)
\mathrm{R}_e\varepsilon(\ub)}{2}+\Theta(w)- \chi_t
-\opchi(\chi)-\xi-  \gamma(\chi)  \in \partial I_{(-\infty,
0]}(\chi_t)$ a.e.\ in $\Omega \times (0,T)$ as an element of
$L^2(0,T;L^2(\Omega))$, and \eqref{eq-with-zeta} follows.
\end{proof}

\begin{remark}[Energy inequality \eqref{total-energy-ineq} vs.\ Energy identity \eqref{total-energy-eq}]
\upshape As we have already pointed out in the proof of
Proposition~\ref{more-regu}, to prove that \emph{any} weak solution $(w,\uu,\chi)$ in
the sense of Def.\ \ref{def-weak-sol} fulfills in the irreversible case (i.e.\ with $\mu=1$)
the  \emph{total energy inequality}  \eqref{total-energy-ineq}
it is sufficient to choose $\varphi \equiv
1$ in \eqref{eq0d}, test \eqref{eq1d} by $\uu_t$, integrate in time and add
the resulting equalities to
\eqref{energ-ineq}.
Instead, the proof of  the \emph{total energy equality}
\eqref{total-energy-eq}
 relies on the fact
that, for $\mu=0$ we are able to obtain the stronger,
\emph{pointwise} form \eqref{eq-with-xi} of  the subdifferential inclusion \eqref{eq2d}.
\end{remark}

We can now state our existence result in the case $\mu=1$, with $\rho=0$.
Its proof will be developed in Sec.\ \ref{ss:5.1} by passing to the
limit in the time-discretization scheme devised in Sec.\ \ref{s:time-discrete}. We mention in advance that,
in this irreversible setting,
 in addition the basic assumptions of Hypotheses  (I)--(V),
we also have to require the  $p$-coercivity condition \eqref{pcoercive}. It has a crucial role
 in proving \emph{strong} convergence  of the approximate solutions to $\chi$
in $L^p(0,T;W^{1,p}(\Omega))$, which enables us to obtain \eqref{energ-ineq}.
Let us also highlight that, exploiting the additional feature of irreversibility, we
 prove a slightly more refined (in comparison with \eqref{teta-non-neg}) positivity result for the
temperature $\teta$, cf.\ \eqref{strictpos} below.

\begin{maintheorem}[Existence of weak solutions for the full system, $\mu=1$,  $\rho=0$]
\label{teor3} Let $\mu=1$, $\rho=0$, and assume  Hypotheses (I)--(V)  with
$\widehat{\beta}=I_{[0,+\infty)}$ as in
\eqref{e:particular-widehatbeta}, and conditions
\eqref{bulk-force}--\eqref{datochi}  on the data $\mathbf{f},$ $g$,
$\teta_0,$ $\uu_0,$ $\vv_0,$ $\chi_0$. Suppose moreover that
$\phi$ complies with \eqref{pcoercive}.
 Then,
\begin{compactenum}
\item[1.] Problem \ref{prob} admits a weak solution $(w,\uu,\chi)$ (cf. Def.~\ref{def-weak-sol}).
\item[2.]
Suppose in addition that
\begin{equation}\label{ipo:strictpos}
g(x,t)\geq 0 \quad \forae\, (x,t) \in \Omega \times
(0,T)\quad  \text{and } \exists\,   \underline\teta_0 \geq 0  \,  \ \ \forae\, x \in \Omega\, :   \quad \teta_0(x)>\underline\teta_0 \geq  0\,.
\end{equation}
Then 
\begin{equation}\label{strictpos}
\teta(x,t):= \Theta(w(x,t))\geq \underline\teta_0  \geq 0 \quad \forae\,
(x,t) \in \Omega \times (0,T).
\end{equation}
\end{compactenum}
\end{maintheorem}
\begin{remark}
\upshape
If $\opchi$ is given by the $s$-Laplacian operator $A_s$
 \eqref{slobo-choice},
 the Definition~\ref{def-weak-sol} of weak
solution to the irreversible full system is obviously modified:
 in
\eqref{ineq-system2-integrated} the term $\int_\Omega
\mathbf{d}(x,\nabla\chi) \cdot \nabla \varphi$ is replaced by
$a_s(\chi,\varphi)$ with $\varphi \in L^2 (0,T;W_-^{s,2}(\Omega))$,
whereas $\Phi(\chi)$ in \eqref{energ-ineq} now reads $\frac12
a_s(\chi,\chi)$.
Our existence result
Theorem~\ref{teor3} extends to such a  case, cf.\ also the forthcoming
Remark \ref{remk:added-revision}.

In Section \ref{sec:5}, we will focus on the
irreversible full system \eqref{eq0d}--\eqref{eq2d} with
$\opchi=A_s$, and perform an asymptotic analysis of weak solutions
(in the sense of Definition \ref{def-weak-sol}) in the
\emph{degenerate} limit $\delta \down 0$.
\end{remark}

\begin{remark}\label{irrev-rho}
\upshape
Replacing Hyp.\ (II) with (VIII) and
carefully tailoring the  estimates and  techniques for the proof of Thm.\ \ref{teor1bis}
 (cf.\ Rmk.\ \ref{rmk:afterThm2}) to the irreversible case,
we could indeed prove the existence of \emph{weak} (in the sense of Def.\ \ref{def-weak-sol}) solutions
also for $\mu = 1$
and $\rho \neq 0$. However we expect that, in the latter setting, only the weaker positivity result
\eqref{teta-non-neg}  can be proved. Indeed,  the estimate yielding the lower bound \eqref{strictpos} (cf.\
Step $4$ in the proof of Lemma \ref{lemma:ex-discr}),
cannot be performed on the enthalpy equation due to the presence of term $-\rho\Theta(\w)\dive(\uu_t)$.
\end{remark}

We finally turn to the irreversible \emph{isothermal} case, and
improve the existence result of Theorem \ref{teor3}.
\begin{maintheorem}[Global existence for the isothermal system, $\mu=1$]
\label{teor4} Let $\mu=1$. In addition to
 Hypotheses (III)--(V),
assume that
\begin{equation}
\label{bsecondnull} b{''}(x)=0 \qquad \text{for all } x \in [0,1],
\end{equation}
  and suppose that the data
$\mathbf{f},  \uu_0, \, \vv_0, \, \chi_0$ comply with conditions
\eqref{bulk-force} and \eqref{datou}--\eqref{datochi}. Suppose in
addition $\phi$ fulfills \eqref{pcoercive} and
\eqref{Lipschitz-x},
 that
\begin{equation}
\label{additional-chizero} \opchi(\chi_0), \, \beta(\chi_0)\in
L^2(\Omega),
\end{equation}
 and consider a
 fixed
temperature profile
 \begin{equation}
 \label{fixed-profile}
 \Theta^* \in W^{1,1}
(0,T;L^2 (\Omega)).
\end{equation}

Then,
  there exists a quadruple
$(\uu,\chi,\xi,\zeta)$, fulfilling \eqref{reg-u}--\eqref{reg-chi},
$\xi \in \beta(\chi)$ and $\zeta \in \partial I_{(-\infty, 0]}(\chi_t) $ a.e. in
$\Omega \times (0,T)$,
 as
well as
\begin{align}
&  \label{enhanced-regu-chi}
 \chi \in
L^\infty(0,T;W^{1+\sigma,p}(\Omega)) \cap
W^{1,\infty}(0,T;L^2(\Omega)) \quad \text{ for every $0<\sigma
<\frac1p$},
\\
 & \label{enhanced-regu-xi}
\xi \in L^\infty (0,T;L^2(\Omega)),
\\
&  \label{enhanced-regu-omega} \zeta \in L^\infty (0,T;L^2(\Omega)),
\end{align}
satisfying equations \eqref{eq1d} and \eqref{eq-with-zeta}, with
$\Theta(w)$ replaced by $\Theta^*$.
\end{maintheorem}
\begin{remark}\label{uni-irrev}
\upshape Uniqueness of solutions for the \emph{irreversible} system,
even in the isothermal case, is still an open problem. This is
mainly due to the triply nonlinear character of \eqref{eq2d} (cf.\
also \cite{CV} for non-uniqueness examples for a general doubly
nonlinear equation).
\end{remark}
\paragraph{A more general dissipation potential.} As observed in
Remark \ref{more-general-alpha} later on, in Thm.\ \ref{teor4} we
could consider a more general dissipation potential in \eqref{eqII}.
Indeed, in place of subdifferential operator $\partial
I_{(-\infty,0]}$, we could  allow for a general cyclical monotone
operator
\begin{equation}
\label{1-homog}
\begin{gathered}
\text{ $\alpha:= \partial \widehat{\alpha}: \R \rightrightarrows
\R$, with}
\\
\widehat{\alpha}:\R \to \R \text{ convex, lower semicontinuous, with
$\mathrm{dom}(\alpha) \subset (-\infty,0]$.}
\end{gathered}
\end{equation}


\section{Time discretization}
\label{s:time-discrete} \noindent First, in Section \ref{ss:3.1} we will
approximate Problem \ref{prob} via
time discretization. In fact, in the reversible case $\mu=0$ with $\rho \in \R$, we will set up
an \emph{implicit} scheme (cf.\ Problems \ref{probk-rev} and \ref{prob:rhoneq0}), whereas for the irreversible system
$\mu=1$ with $\rho=0$,  we will employ the \emph{semi-implicit} scheme of Problem \ref{probk-irr}. Moreover,
we will tackle separately the discretization of the isothermal irreversible
system in Problem \ref{probk-irr-iso}. We refer to Remarks \ref{rmk:comparison-1} and \ref{rem:irrev-discr}
for a thorough comparison between the various time-discretization procedures,
and more comments.
Second,
in Sec.\ \ref{ss:3.2new} we will prove existence results for Problems \ref{probk-rev}--\ref{probk-irr-iso}.
Third,
 in Sec.\ \ref{ss:3.2} we will perform suitable a priori
estimates.



\begin{notation}
\upshape \label{not-alpha} In what follows, also in view of the
extension \eqref{1-homog} mentioned at the end of Sec.\
\ref{glob-irrev}  (cf.\ Rmk.\ \ref{more-general-alpha}),  we will use $\widehat{\alpha}$ and $\alpha$ as
place-holders for $I_{(-\infty,0]}$ and $\partial I_{(-\infty,0]}$.
\end{notation}
\subsection{Setup of the time discretization}
\label{ss:3.1}
 We consider
an equidistant partition of $[0,T]$, with time-step $\tau>0$ and
nodes $t_\tau^k:=k\tau$, $k=0,\ldots,K_\tau$.
In this framework, we
approximate the data $\mathbf{f}$ and $g$
 by local means, i.e.
setting for all $k=1,\ldots,K_{\tau}$
\begin{equation}
\label{local-means} \ftau{k}:=
\frac{1}{\tau}\int_{t_\tau^{k-1}}^{t_\tau^k} \mathbf{f}(s)\dd s\,,
\qquad  \gtau{k}:= \frac{1}{\tau}\int_{t_\tau^{k-1}}^{t_\tau^k} g(s)
\dd s\,.
\end{equation}

\begin{problem}[Time discretization of the full reversible system, $\mu=0$, $\rho=0$]
\label{probk-rev} \upshape  Given
\begin{align}\label{IC2}
\wtau{0}:=\w_{0}, \qquad \utau{0}:=\uu_{0},\qquad
\utau{-1}:=\uu_{0}-\tau \vv_0, \qquad \chitau{0}:=\chi_{0},
\end{align}
find $\{\wtau{k}, \utau{k}, \chitau{k}, \xitau{k}\}_{k=1}^{K_\tau}
\subset  H^1 (\Omega) \times \boY \times W^{1,p}(\Omega)\times L^2
(\Omega)$, with $\xitau{k} \in \beta(\chitau{k})$ a.e. in $\Omega$,
fulfilling
\begin{align}
& \label{eq-discr-w} \frac{\wtau{k} -\wtau{k-1}}{\tau} +
\frac{\chitau{k} -\chitau{k-1}}{\tau}\Theta(\wtau{k})
+   A_{\wtau{k-1}}(\wtau{k})  = \gtau{k} \quad \text{in $H^1(\Omega)'$,}
\\
& \label{eq-discr-u} \frac{\utau{k} -2\utau{k-1} +
\utau{k-2}}{\tau^2} + \opj{(a(\chitau{k})
+\delta)}{\frac{\utau{k}-\utau{k-1}}{\tau}} +
\oph{b(\chitau{k})}{\utau{k}} 
= \ftau{k} \quad \aein\, \Omega,
\\
& \label{eq-discr-chi}
\begin{aligned}
 \frac{\chitau{k} -\chitau{k-1}}{\tau}
+ \opchi(\chitau{k}) + \xitau{k} +\gamma(\chitau{k})  = -
b'(\chitau{k-1})\frac{\eps(\utau{k-1})\mathrm{R}_e
\eps(\utau{k-1})}2 + \Theta(\wtau{k})
  \quad \aein\, \Omega,
  \end{aligned}
\end{align}
 where in \eqref{eq-discr-w} the operator  $A_{\wtau{k-1}}:
H^1(\Omega)\rightarrow H^1(\Omega)'$  is defined by
  \begin{equation}
  \label{def:opA}
 \pairing{}{H^1(\Omega)}{A_{\wtau{k-1}}(w)}{v} :=
  \int_{\Omega} K(\wtau{k-1})\nabla w \cdot\nabla v \dd x  \quad \text{for all } w, v \in
  H^1(\Omega).
\end{equation}
\end{problem}
 For the full reversible system with $\rho \neq 0$, we work under the stronger Hypothesis (VIII)
and thus prescribe a suitable growth on the function
$K$, which is no longer bounded. Therefore, in order to properly deal with the elliptic operator in
the enthalpy equation on the time-discrete level, we need to truncate
$K$.  We thus introduce the operator
\begin{equation}
\label{def-trunc-ope}
 \pairing{}{H^1(\Omega)}{A_{\wtau{k-1},M}(w)}{v} :=
  \int_{\Omega} K_M(\wtau{k-1})\nabla w \cdot\nabla v \dd x  \quad \text{for all } w, v \in
  H^1(\Omega)\,,
\end{equation}
with
\begin{equation}
\label{def-k-m}
K_M(r):=  \left\{ \begin{array}{ll}
K(-M) & \text{if } r <-M,
\\
K(r)   & \text{if } |r| \leq M,
\\
K(M) & \text{if } r >M.
\end{array}
\right.
\end{equation}
Observe that
\begin{equation}
\label{ellipticity-retained}
K_M(r) \geq c_{10} \qquad \text{  for every $r \in \R$}
\end{equation}
 (with $c_{10}$ from \eqref{hyp-K-rho}).
Accordingly, we will have to truncate the function $\Theta$, replacing it with
\begin{equation}
\label{def-theta-m}
\Theta_M(r):= \left\{ \begin{array}{ll}
\Theta(-M) & \text{if } r <-M,
\\
\Theta(r)   & \text{if } |r| \leq M,
\\
\Theta(M)  & \text{if } r >M.
\end{array}
\right.
\end{equation}
\begin{problem}[Time discretization of the full reversible system, $\mu=0$, $\rho\neq 0$]
\label{prob:rhoneq0}
Starting from
$(\utau{0},$ $\utau{-1},$ $\chitau{0},$ $\wtau{0})$ as in \eqref{IC2},
find
 $\{\wtau{k}, \utau{k}, \chitau{k}, \xitau{k}\}_{k=1}^{K_\tau}
\subset  H^1 (\Omega) \times \boY \times W^{1,p}(\Omega)\times L^2
(\Omega)$
with $\xitau{k} \in \beta(\chitau{k})$ a.e. in $\Omega$,
fulfilling
\begin{align}
& \label{eq-discr-w-TRUNC} \frac{\wtau{k} -\wtau{k-1}}{\tau} +
\frac{\chitau{k} -\chitau{k-1}}{\tau}\Theta_M(\wtau{k})
+  \rho\dive\left(\frac{\utau{k}-\utau{k-1}}{\tau}\right)\Theta_M(\wtau{k})
+  A_{\wtau{k-1},M}(\wtau{k})   = \gtau{k} \quad \text{in $H^1(\Omega)'$,}
\\
& \label{eq-discr-u-TRUNC} \frac{\utau{k} -2\utau{k-1} +
\utau{k-2}}{\tau^2} + \opj{(a(\chitau{k})
+\delta)}{\frac{\utau{k}-\utau{k-1}}{\tau}} +
\oph{b(\chitau{k})}{\utau{k}}  + \ciro(\Theta_M(\wtau{k})) = \ftau{k} \quad \aein\, \Omega,
\\
& \label{eq-discr-chi-TRUNC}
\begin{aligned}
 \frac{\chitau{k} -\chitau{k-1}}{\tau}
+ \opchi(\chitau{k}) + \xitau{k} +\gamma(\chitau{k})  = -
b'(\chitau{k-1})\frac{\eps(\utau{k-1})\mathrm{R}_e
\eps(\utau{k-1})}2 + \Theta_M(\wtau{k})
  \quad \aein\, \Omega\,.
  \end{aligned}
\end{align}
\end{problem}

 We now present the time discretization of the  full irreversible system, postponing to
Remark \ref{rmk:comparison-1} a detailed comparison
between Problem \ref{probk-rev} and the forthcoming Problem \ref{probk-irr}. Let us only mention
in advance that, in the irreversible case
we will restrict to the particular choice
$\beta=\partial I_{[0,+\infty)}$. Furthermore,
in Problem \ref{probk-irr} instead of the time discretization of
\eqref{eq2d}, we will consider the  minimum problem
\eqref{eq-discr-chi-irr},
 such that its Euler equation is \eqref{eq2d} discretized. We resort
  to this approach in view of  the
passage to the limit argument as $\tau \to 0$, mutuated from \cite{hk1}, which we
will develop in the proof of Thm.\ \ref{teor3}.
Finally,
due to technical reasons related to the proof of the
 \emph{Third a priori estimate} in Sec.\ \ref{ss:3.2}, we will also need to
 approximate the initial datum $w_0$ with a sequence
 \begin{equation}
 \label{approx-w0}
 (w_{0\tau})_\tau \subset W^{1,\bar{r}}(\Omega) \text{ such that } \
\sup_{\tau>0} \tau\|\nabla w_{0\tau}\|_{L^{\bar{r}}(\Omega)}^{\bar{r}}\leq C, \quad
w_{0\tau}\to w_0 \text{ in $L^1(\Omega)$ as $\tau \to 0$,}
 \end{equation}
  with  $ \bar{r}=(d+2)/(d+1)$, cf.\ \eqref{reg-ental}.
 We construct $(w_{0\tau})_\tau$ in such a way that, if
 $\teta_0$ complies with \eqref{ipo:strictpos}, then for every $\tau>0$
 \begin{equation}
 \label{appr-w-0-pos}
w_{0\tau}(x) \geq  \underline{\w}_0:= h(\underline{\teta}_0)\geq 0 \qquad \foraa\, x \in \Omega.
 \end{equation}
\begin{problem}[Time discretization of the irreversible \emph{full} system, $\mu=1$, $\rho=0$]
\label{probk-irr} \upshape  Starting from the data
$(\utau{0},$ $\utau{-1},$ $\chitau{0},$ $\wtau{0})$ as in \eqref{IC2},  with
$\wtau{0}=w_{0\tau}$ as in \eqref{approx-w0},
 find
$\{\wtau{k}, \utau{k}, \chitau{k},  \zetau{k}\}_{k=1}^{K_\tau} \in
H^1 (\Omega) \times \boY \times W^{1,p}(\Omega)\times L^2 (\Omega)$,
 such that for all $k=1,\ldots,K_\tau$ there holds
 \begin{equation}
\label{e:irrevers} \chitau{k} \leq \chitau{k-1} \qquad \text{ a.e.
in $\Omega$ and } \qquad  \zetau{k} \in \alpha ((\chitau{k}
-\chitau{k-1})/\tau) \qquad \text{ a.e. in $\Omega$,}
\end{equation}
 and $(\wtau{k}, \utau{k}, \chitau{k})$ fulfill
\begin{align}
& \label{eq-discr-w-irr} \frac{\wtau{k} -\wtau{k-1}}{\tau} +
\frac{\chitau{k} -\chitau{k-1}}{\tau}\Theta(\wtau{k-1})
+A_{\wtau{k-1}}(\wtau{k}) = \gtau{k} \quad \text{in $H^1(\Omega)'$,}
\\
& \label{eq-discr-u-irr} \frac{\utau{k} -2\utau{k-1} +
\utau{k-2}}{\tau^2} + \opj{(a(\chitau{k})
+\delta)}{\frac{\utau{k}-\utau{k-1}}{\tau}} +
\oph{b(\chitau{k})}{\utau{k}}= \ftau{k} \quad \aein\, \Omega,
\end{align}
 and
\begin{align}
  \label{eq-discr-chi-irr}
\begin{aligned}
\chitau k \in \mathrm{Argmin}_{\chi \in W^{1,p} (\Omega)}  & \left\{
 \int_{\Omega}\left(\frac{\tau}2\left| \frac{ \chi - \chitau{k-1}}\tau
\right|^2 + \widehat{\alpha} \left(\frac{ \chi - \chitau{k-1}}\tau
\right)\right) \dd x +\Phi(\chi) + \int_{\Omega}(\widehat{\beta}(\chi) +
\widehat{\gamma}(\chi)) \dd x \right.
\\   & \left.   \qquad
+\int_{\Omega} \left(  b'(\chitau{k-1})
\frac{\eps(\utau{k-1})\mathrm{R}_e \eps(\utau{k-1})}2 -
\Theta(\wtau{k-1})\right)\chi \,\dd x \right\}.
  \end{aligned}
\end{align}
\end{problem}
\begin{remark}
\label{rmk:comparison-1}
\upshape
 The main difference
between systems \eqref{eq-discr-w}--\eqref{eq-discr-chi} and \eqref{eq-discr-w-irr}--\eqref{eq-discr-chi-irr}
consists in the discretization of the coupling term
$\Theta(\w)$
in the temperature and in the phase parameter equations.
Indeed, in the reversible case
$\Theta(\w)$ on the l.h.s.\ of
\eqref{eq-discr-w} (and accordingly its coupled term on the r.h.s.\ of \eqref{eq-discr-chi},
which will cancel out in the \emph{First a priori estimate} of Sec.\ \ref{ss:3.2}), is kept \emph{implicit}.
Only relying on this
we can prove the positivity of the discrete enthalpy
$\wtau{k}$, which for system \eqref{eq-discr-w}--\eqref{eq-discr-chi}
would not follow from other considerations.
Instead, in the time discretization  \eqref{eq-discr-w-irr}--\eqref{eq-discr-chi-irr}
we can allow for an \emph{explicit} coupling term
$\Theta(\wtau{k-1})$
in \eqref{eq-discr-w-irr} and in \eqref{eq-discr-chi-irr}. Therein, the positivity
of the discrete enthalpy will be proved by means of a suitable test
of the discrete enthalpy equation, relying on the irreversibility \eqref{e:irrevers}.


 Because of its implicit character, in Lemma \ref{lemma:ex-discr-rev}
 existence for
 system \eqref{eq-discr-w}--\eqref{eq-discr-chi} will be proved
 by resorting to fixed-point type existence results for elliptic systems
 featuring pseudo-monotone (cf.\ e.g.\ \cite[Chap.\ II]{roub-NPDE})  operators.

 Instead, in the semi-implicit scheme of Problem \ref{probk-irr},
  equations \eqref{eq-discr-w-irr}, \eqref{eq-discr-u-irr}, and \eqref{eq-discr-chi-irr} are decoupled,  hence  we will
 proceed by tackling them
 separately, solving  time-incremental minimization problems. Such a procedure
 could be useful for the numerical analysis of the problem. That is
why,
 in Sec.\ \ref{ss:3.2new} we will focus on the
 proof of Lemma \ref{lemma:ex-discr}  and develop in detail
the calculations for system \eqref{eq-discr-w-irr}--\eqref{eq-discr-chi-irr},
whereas  we will only outline the argument for the existence of solutions to \eqref{eq-discr-w}--\eqref{eq-discr-chi} in Lemma
\ref{lemma:ex-discr-rev}.
\end{remark}

In the time-discretization of the irreversible \emph{isothermal} system, we approximate the given temperature
profile $\Theta^*$ (cf.\ \eqref{fixed-profile}) by local means as well, i.e.
\begin{equation}
\label{local-means-theta}
\Thetatau{k} :=
\frac{1}{\tau}\int_{t_\tau^{k-1}}^{t_\tau^k} \Theta^*(s) \dd s.
\end{equation}
Contrary to the temperature-dependent irreversible case, we  may again address a
 general maximal monotone $\beta: \R \rightrightarrows \R$.
 However, in order to
perform  enhanced estimates on the discrete equation for $\chi$
(cf.\ the \emph{Seventh} and
\emph{Eighth a priori estimates} of Sec.\ \ref{ss:3.2}),
 we will need to replace
 $\beta$
  with
 its Yosida regularization $\beta_\tau: \R\to \R$, namely
the nondecreasing, Lipschitz continuous derivative of the  convex $
\mathrm{C}^1 $  function $\widehat{\beta}_{\tau}(x):= \min_{y \in
\R} \{ |y-x|^2/{2\tau} + \widehat{\beta}(y)\}$, cf.\ e.g.\
\cite{barbu, brezis}. In Problem \ref{probk-irr-iso}
below,  we set the regularization parameter equal to the
time-step, in view of passing to the limit simultaneously in the
time discretization and in the Yosida regularization as $\tau \to
0$.
Furthermore, we will have to work with
a suitable truncation of the
coefficient  $a(\chi)$  in \eqref{eq1d}, cf.\ Remark \ref{rem:irrev-discr} below for further
comments.
\begin{problem}[Time discretization of the irreversible \emph{isothermal} system]
\label{probk-irr-iso}
\upshape
 Starting from the triple of da\-ta $(\utau{0},\utau{-1},\chitau{0})$ defined as in \eqref{IC2}
and
considering the discrete approximations
$(\Thetatau{k})_{k=1}^{K_\tau}$ of the given temperature profile
$\Theta^*$, find $\{\utau{k}, \chitau{k}, \zetau{k}\}_{k=1}^{K_\tau}
\in \boY \times W^{1,p}(\Omega)\times L^2 (\Omega)$,
 such that for all $k=1,\ldots,K_\tau$ there holds
\begin{align}
&
 \label{eq-discr-u-irr-iso} \frac{\utau{k} -2\utau{k-1
 } +
\utau{k-2}}{\tau^2} + \opj{((a(\chitau{k}))^+
+\delta)}{\frac{\utau{k}-\utau{k-1}}{\tau}} +
\oph{b((\chitau{k})^+)}{\utau{k}}{} = \ftau{k} \quad \aein\, \Omega,
\\
&  \label{eq-discr-chi-irr-iso}
\frac{\chitau{k} -\chitau{k-1}}{\tau}  + \zetau{k}
+ \opchi(\chitau{k}) + \beta_\tau(\chitau{k}) +\gamma(\chitau{k})=-
b'((\chitau{k-1})^+) \frac{\eps(\utau{k-1})\mathrm{R}_e
\eps(\utau{k-1})}2 + \Thetatau{k-1}
  \quad \aein\, \Omega,
  \\
  &
  \label{eq-zeta-discr}
  \zetau{k} \in \alpha \left(\frac{\chitau{k} -\chitau{k-1}}{\tau}
\right)  \quad \aein\, \Omega.
  \end{align}
\end{problem}
\begin{remark}
\label{rem:irrev-discr} \upshape  In  Problem
\ref{probk-irr-iso}  we need to approximate $\beta$ by a Lipschitz
continuous function $\beta_\tau$ because, only with such a
regularization can we  test equation  \eqref{eq-discr-chi-irr-iso}  by the
discrete difference $\tau^{-1} (( \opchi(\chitau{k}) +
\beta_\tau(\chitau{k}) )- (\opchi(\chitau{k-1}) +
\beta_\tau(\chitau{k-1})) )$ (cf.\ the following Seventh a priori
estimate). Hence, we need to take the positive part of $a$ in
 \eqref{eq-discr-u-irr-iso} because, replacing $\beta$ by   its Lipschitz
regularization $\beta_\tau$,  we are no longer able to enforce the
constraint that $\chitau{k} \in [0,1]$ a.e.\ in~$\Omega$. Therefore,
at the discrete level we loose all positivity information on the
coefficient $a(\chi)$. The lack of the constraint  $\chitau{k} \in
[0,1]$ also motivates the truncations $b((\chitau{k})^+)$ in
 \eqref{eq-discr-u-irr-iso}  and $b'((\chitau{k})^+)$  in \eqref{eq-discr-chi-irr-iso},
   mainly due to technical
reasons (cf.\ the \emph{First a priori estimate}).
\end{remark}
\subsection{Existence for the time-discrete problems}
\label{ss:3.2new}
First, we prove the existence of solutions to the semi-implicit
schemes \eqref{eq-discr-w-irr}--\eqref{eq-discr-chi-irr}
and \eqref{eq-discr-u-irr-iso}--\eqref{eq-zeta-discr}.
\begin{lemma}[Existence for the time-discrete Problems~\ref{probk-irr} and \ref{probk-irr-iso}, $\mu=1$, $\rho=0$]
\label{lemma:ex-discr}
Let $\mu=1$ and $\rho=0$.
 Assume   Hypotheses (I)--(V), and
\eqref{bulk-force}--\eqref{datochi} on the data $\mathbf{f},\,g,\,
\teta_0,\, \uu_0,\,\vv_0,\, \chi_0$.

Then, Problems \ref{probk-irr} and \ref{probk-irr-iso} admit at
least one solution  $\{(\wtau{k}, \utau{k},
\chitau{k}, \zetau{k})\}_{k=1}^{K_\tau}$   and  $\{(\utau{k}, \chitau{k},
\zetau{k})\}_{k=1}^{K_\tau}$, resp.

Furthermore, if \eqref{ipo:strictpos} holds, then
  any so\-lu\-tion   $\{(\wtau{k}, \utau{k}, \chitau{k},  \zetau{k})\}_{k=1}^{K_\tau}$   of Problem
  \ref{probk-irr}
  fulfills
  \begin{equation}
\label{strict-pos-wk}
 \wtau{k}(x) \geq \underline{\w}_0= h(\underline{\teta}_0)\geq0  \quad \foraa\, x \in \Omega.
 \end{equation}
\end{lemma}
\begin{proof}
We treat Problems \ref{probk-irr}, and
\ref{probk-irr-iso} in a unified way, and proceed by induction on
$k$. Thus,   starting from a quadruple $(\utau{k-2},\wtau{k-1},
\utau{k-1}, \chitau{k-1}) \in \boY \times H^1 (\Omega) \times \boY
\times W^{1,p}(\Omega)$, we show that there exist functions
$(\wtau{k}, \utau{k},
\chitau{k}, \zetau{k})$ and  $(\utau{k}, \chitau{k},\zetau{k})$, resp.,
solving \eqref{eq-discr-w-irr}--\eqref{eq-discr-chi-irr}
for Problem \ref{probk-irr}, and
\eqref{eq-discr-u-irr-iso}--\eqref{eq-zeta-discr} for Problem
\ref{probk-irr-iso}, resp.

\paragraph{Step $1$: discrete equation for $\chi$.}
In the irreversible isothermal case (i.e.\ for Problem
\ref{probk-irr-iso}), in order to solve \eqref{eq-discr-chi-irr-iso}
 we start from the approximate equation
\begin{equation}
\label{chi-disc-approx}
\begin{aligned}
\frac{\chitaue{k} -\chitau{k-1}}{\tau} & +
\alpha_\eps \left(\frac{\chitaue{k} -\chitau{k-1}}{\tau}\right) +
\opchi(\chitaue{k}) + \beta_\tau(\chitaue{k}) +\gamma(\chitaue{k}) \\ &  =
- b'(\chitau{k-1})\frac{\eps(\utau{k-1})\mathrm{R}_e
\eps(\utau{k-1})}2 +  \Thetatau{k-1} \quad \aein\, \Omega,
\end{aligned}
\end{equation}
where $\eps>0$ and $\alpha_\eps $ is the Yosida
regularization of the operator $\alpha$. Clearly,
\eqref{chi-disc-approx} is the Euler equation for the minimum
problem
\[
\min_{\chi \in W^{1,p} (\Omega)} \left\{\tau \int_{\Omega}\left(\left|
\frac{ \chi - \chitau{k-1}}\tau \right|^2 + \widehat{\alpha}_\eps
\left(\frac{ \chi - \chitau{k-1}}\tau \right) \right)\dd x
 +\Phi(\chi)
+ \int_{\Omega}(\widehat{\beta}_\tau (\chi) +
\widehat{\gamma}(\chi)) \dd x+\int_{\Omega} h_{\tau}^{k-1} \chi \dd
x\right\},
\]
where the function
\begin{equation}
\label{not-h}
 h_{\tau}^{k-1}:=
b'(\chitau{k-1})\eps(\utau{k-1})\mathrm{R}_e \eps(\utau{k-1})/2 -
 \Thetatau{k-1} \text{ is in $L^2 (\Omega)$.}
\end{equation}
The latter admits a solution $\chitaue{k}$ by the direct method of
the calculus of variations (also taking into account  the fact that
$\widehat{\beta}_\tau$ is bounded from below because
$\widehat{\beta}$ is). We now want to pass to the limit in
\eqref{chi-disc-approx} as $\eps \down 0$. Note that, a comparison
in \eqref{chi-disc-approx}  and the fact that $\alpha_\eps$ is
Lipschitz continuous yield that $\opchi(\chitaue{k}) \in
L^2(\Omega)$.  Then, following \cite[Sec.\ 3]{lss02} (to which we refer for
all details), we multiply \eqref{chi-disc-approx} firstly by
$\chitaue{k} - \chitau{k-1}$, and secondly by $\opchi (\chitaue{k})
- \opchi(\chitau{k-1})$. To perform the latter estimate,  we  rely
on the Lipschitz continuity of $\beta_{\tau}$ and $\gamma$, as well
as on the monotonicity of $\alpha_\eps$, yielding
\[
\int_{\Omega} \alpha_\eps \left(\frac{\chitaue{k}
-\chitau{k-1}}{\tau}\right)\left( \opchi (\chitaue{k}) -
\opchi(\chitau{k-1}) \right)\dd x \geq 0.
\]
It follows from  these tests  that there exists a constant $C>0$,
depending on $\tau>0$ but not on $\eps>0$, such that
$$ \sup_{\eps>0}(\| \chitaue{k}
\|_{W^{1,p}(\Omega)}+\|\opchi( \chitaue{k}) \|_{L^{2}(\Omega)} )
\leq C .$$
By comparison, $\sup_{\eps>0}\|\alpha_\eps ((\chitaue{k}
-\chitau{k-1})/\tau)\|_{L^{2}(\Omega)} \leq C$.
 Also in view of  the regularity result \eqref{sav-reg}, there exist $(\chitau{k},\zetau{k}) \in
 W^{1+\sigma,p}(\Omega) \times L^2(\Omega)$ for all $0<\sigma<1/p$ such that,  up to  a
subsequence, $(\chitaue{k})_\eps$  strongly converges in
$W^{1,p}(\Omega)$ to $\chitau{k}$ as $\eps \to 0$, and $(\alpha_\eps
((\chitaue{k} -\chitau{k-1})/\tau))_\eps$ weakly converges in
$L^2(\Omega)$ to $\zetau{k}$  as $\eps \to 0$. Therefore,
\[
\limsup_{\eps \to 0}\int_{\Omega} \alpha_\eps
\left(\frac{\chitaue{k} -\chitau{k-1}}{\tau}\right)
\left(\frac{\chitaue{k} -\chitau{k-1}}{\tau}\right)  \dd x \leq
\int_{\Omega} \zetau{k} \left(\frac{\chitau{k}
-\chitau{k-1}}{\tau}\right)  \dd x,
\]
so that $\zetau{k} \in \alpha((\chitau{k}-\chitau{k-1})/\tau)$ thanks to
\cite[p.~42]{barbu}. Thus, passing to the limit as $\eps \to 0$ in \eqref{chi-disc-approx} for
$\tau>0$ fixed,   we conclude that the
functions $(\chitau{k},\zetau{k})$ fulfill \eqref{eq-discr-chi-irr-iso}.

Clearly, the direct method of the calculus of variations also yields
the existence of a solution to the minimum problem
\eqref{eq-discr-chi-irr}.

\paragraph{Step $2$: discrete equation for $\uu$.}
 Next, we solve \eqref{eq-discr-u-irr}, which can be rewritten  $\aein\, \Omega$
in the form
\begin{equation}
\label{ellipt-u}
\left(\mathrm{Id} +\tau \opj{(a(\chitau{k})+\delta)}{\,\cdot} +
\tau^2 \oph{b(\chitau{k})}{\,\cdot} \right) (\utau{k}) = \tau^2
\ftau{k} +\tau \opj{(a(\chitau{k})+\delta)}{\utau{k-1}} +
2\utau{k-1} - \utau{k-2}.
\end{equation}
Combining the fact that  $\chitau{k} \in [0,1] $ a.e.\ in $\Omega$
with \eqref{data-a} on $a$ and $b$ and \eqref{korn}--\eqref{a:conti-form}, we
conclude that (the bilinear form associated with) the operator on
the left-hand side of the above equation is  continuous and coercive.
Hence, by Lax-Milgram's theorem, equation  \eqref{ellipt-u} admits
a (unique) solution $\utau{k}\in \boZ$.
Since the right-hand side of \eqref{ellipt-u} is in $L^2(\Omega;\R^d)$,
 relying on the regularity results of, e.g., \cite{ciarlet},
we conclude that  in fact $\utau{k}\in
\boY$. The analysis of  \eqref{eq-discr-u-irr-iso}  follows the  very same
lines.

\paragraph{Step $3$: discrete equation for $\w$.}
Finally, let us consider the functional $\mathcal{G}_{\tau}^{k-1}:
 H^{1} (\Omega)  \to \R$ 
 \[
\begin{aligned}
\mathcal{G}_{\tau}^{k-1}(\w):= & \frac1{2\tau}\int_\Omega |\w -
\wtau{k-1}|^2\dd x +  \int_{\Omega}\Theta(\wtau{k-1}) \left(
\frac{\chitau{k} - \chitau{k-1}}{\tau}\right) \dd x
\\ & \quad
+ \frac12
\int_{\Omega} K(\wtau{k-1})|\nabla w|^2 \dd x
-\pairing{}{H^1(\Omega)}{\gtau{k}}{\w}\,.
\end{aligned}
\]
Now, $\mathcal{G}_{\tau}^{k-1}$ is lower semicontinuous w.r.t.\ the
topology of $L^2(\Omega)$. Furthermore,    in view of
\eqref{conseq-2} and of the Young inequality we have for a fixed
$\varrho>0$ 
\begin{equation}
\label{coer-est}
\begin{aligned}
\mathcal{G}_{\tau}^{k-1}(\w) \geq
&\frac1{4\tau}\|\w\|_{L^2(\Omega)}^2  + \frac{c_2}2 \| \nabla w
\|_{L^2(\Omega)}^2
-\varrho\|\w\|_{H^1(\Omega)}^2 \\ &
-C_\varrho(\|\wtau{k-1}\|_{L^2(\Omega)}^2 + \|\gtau{k} \|_{H^1(\Omega)'}^2
+
\| (\chitau{k} - \chitau{k-1})/\tau\|_{L^2(\Omega)}^2).
\end{aligned}
\end{equation}
Choosing $\varrho$ sufficiently small, we
thus obtain
that there exist two positive constants $c$ and $C$ such that
\begin{align}\no
 \mathcal{G}_{\tau}^{k-1}(\w) \geq & c \| \w \|_{H^1(\Omega)}^2 - C
(1+ \|\wtau{k-1}\|_{L^2(\Omega)}^2 + \|\gtau{k} \|_{H^1(\Omega)'}^2
+ \| (\chitau{k} - \chitau{k-1})/\tau\|_{L^2(\Omega)}^2\quad
\text{for all } \w \in H^1(\Omega)\,.
\end{align}
This shows that the sublevels of $\mathcal{G}_{\tau}^{k-1}$ are
bounded in $H^1(\Omega)$. Hence, again by the direct method in the
calculus of variations, we conclude   that there exists $\wtau{k}
\in \argmin_{\w \in H^1(\Omega)} \mathcal{G}_{\tau}^{k-1}(\w)$, and
$\wtau{k}$ satisfies the associated Euler equation, namely
\eqref{eq-discr-w-irr}.

\paragraph{Step $4$: positivity.}
Let us assume in addition  that
\eqref{ipo:strictpos} holds, and prove \eqref{strict-pos-wk} by induction on $k$.

Preliminarily, we prove by induction on $k$ that
\begin{equation}
\label{pos-wk}
\wtau{k}(x) \geq 0 \qquad \foraa\, x \in \Omega \text{ and  for all } k \in \N.
\end{equation}
Clearly \eqref{pos-wk} holds for $k=0$ thanks to \eqref{appr-w-0-pos}. It remains to show that,
if $\wtau{k-1} \geq 0$
a.e. in $\Omega$, then $\wtau{k} \geq 0$ a.e. in $\Omega$.   Indeed,
let us test  \eqref{eq-discr-w-irr}  by $-(\wtau{k})^-$. Taking into
account  the definition  \eqref{K-T} of $\Theta$, we have that
$\int_{\Omega}\Theta(\wtau{k-1}) (\chitau{k}-\chitau{k-1})
(-(\wtau{k})^-) \dd x \geq 0 $. Combining this with the inequality
\begin{equation}
\label{quoted-later-on}
\frac1{\tau}\int_{\Omega} (\wtau{k}-\wtau{k-1}) (-(\wtau{k})^-) \dd x
\geq \frac1{2\tau}\int_{\Omega}
(|(\wtau{k})^-|^2-|(\wtau{k-1})^-|^2) \dd x,
\end{equation}
and noting that
$(\wtau{k-1})^- =0$  a.e. in $\Omega$, also in view of
\eqref{conseq-2} we obtain
\[
\frac1{2\tau}\int_{\Omega}|(\wtau{k})^-|^2 \dd x + c_2 \int_{\Omega}
|\nabla (\wtau{k})^-|^2 \dd x \leq - \int_{\Omega}
\gtau{k}(\wtau{k})^- \dd x \leq  0,
\]
yielding $(\wtau{k})^-=0$ a.e. in $\Omega$, whence \eqref{pos-wk}.

Now, to prove \eqref{strict-pos-wk}, we observe that  \eqref{strict-pos-wk} holds for $k=0$ due to
\eqref{appr-w-0-pos}.  Suppose now that $\wtau{k-1} \geq
\underline{\w}_0$ a.e. in $\Omega$: in order to prove that $\wtau{k}
\geq \underline{\w}_0$ a.e. in $\Omega$, we test \eqref{eq-discr-w-irr}
by $-(\wtau{k}-\underline{w}_0)^-$. With analogous  calculations as
above we obtain
\[
\frac1{2\tau}\int_{\Omega}|(\wtau{k}-\underline{\w}_0)^-|^2 \dd x +
c_2 \int_{\Omega} |\nabla (\wtau{k}-\underline{\w}_0)^-|^2 \dd x
\leq - \int_{\Omega} \gtau{k}(\wtau{k}-\underline{\w}_0)^- \dd x
\]
\[+
\int_{\Omega}\Theta(\wtau{k}) (\chitau{k-1}-\chitau{k})
(-(\wtau{k}-\underline{\w}_0)^-) \dd x\leq  0,
\]
where the last inequality is due to the fact that $\gtau{k} \geq 0$
a.e. in $\Omega$, and that $\Theta(\wtau{k})
(\chitau{k-1}-\chitau{k}) \geq 0$ a.e.\ in $\Omega$ by the
previously proved \eqref{pos-wk} and the irreversibility constraint.
Thus, we conclude \eqref{strict-pos-wk}.
\end{proof}
The existence result for Problems \ref{probk-rev} {and \ref{prob:rhoneq0}} reads:
\begin{lemma}[Existence for the time-discrete {Problems~\ref{probk-rev}, \ref{prob:rhoneq0}, $\mu=0$}]
\label{lemma:ex-discr-rev}
 Let $\mu=0$.
 Assume Hypotheses (I) and (III)--(V), and \eqref{bulk-force}--\eqref{datochi} on the data $\mathbf{f},\,g,\,
\teta_0,\, \uu_0,\,\vv_0,\, \chi_0$. Furthermore,
\begin{compactenum}
\item[1.] if $\rho =0$, assume Hypothesis (II);
\item[2.]  if $\rho \neq 0$, assume Hypothesis (VIII) and in addition that $w_0 \in L^2(\Omega)$.
\end{compactenum}
Then, Problem \ref{probk-rev} admits at
least one solution $\{(\wtau{k}, \utau{k}, \chitau{k},
\xitau{k})\}_{k=1}^{K_\tau}$.

Moreover, if $g\geq 0$ a.e.\ in $\Omega\times (0,T)$, and $w_0(x) \geq 0$ for a.a.\ $x \in \Omega$,
then any so\-lu\-tion $\{(\wtau{k}, \utau{k}, \chitau{k})\}_{k=1}^{K_\tau}$ of Problem \ref{probk-rev}
fulfills
\begin{equation}
\label{pos-wk-rev}
\wtau{k}(x) \geq 0 \quad \foraa\, x \in \Omega.
\end{equation}
\end{lemma}
\begin{proof}

\noindent
\textbf{Step $1$: existence of solutions.}  Our argument relies on
existence results for elliptic systems from the theory of pseudo-monotone operators which
can be found, e.g., in \cite[Chap.\ II]{roub-NPDE}. Indeed, we observe that
system \eqref{eq-discr-w}--\eqref{eq-discr-chi} can be recast as
\begin{equation}
\label{sistemone}
\begin{aligned}
&  \wtau{k} + (\chitau{k} -\chitau{k-1})\Theta(\wtau{k}) 
+\tau A_{\wtau{k-1}}(\wtau{k})  = \wtau{k-1} +\tau \gtau{k} \quad \text{in $H^1(\Omega)'$,}
\\
&
\utau{k} + \tau\opj{(a(\chitau{k})
+\delta)}{(\utau{k}-\utau{k-1})} + \tau^2
\oph{b(\chitau{k})}{\utau{k}}  
=  2\utau{k-1} -
\utau{k-2} +\tau^2\ftau{k} \quad \aein\, \Omega,
\\
&
\chitau{k}
+\tau \opchi(\chitau{k}) + \tau \beta(\chitau k)+ \tau \gamma(\chitau{k})  - \tau  \Theta(\wtau{k})  \ni  \chitau{k-1} -
\tau
b'(\chitau{k-1})\frac{\eps(\utau{k-1})\mathrm{R}_e
\eps(\utau{k-1})}2
  \quad \aein\, \Omega.
\end{aligned}
\end{equation}
Denoting by $\mathcal{R}_{k-1}$ the operator acting on the unknown $(\wtau k, \utau k, \chitau k)$
and by $H_{k-1}$ the
vector of the terms on the r.h.s.\ of the  above equations,  we can reformulate system
\eqref{sistemone} in the abstract form
\begin{equation}
\label{abstract-sistemone}
\mathcal{R}_{k-1}(\wtau k, \utau k, \chitau k)= H_{k-1}.
\end{equation}
In fact,  mimicking for example the calculations
in \cite[Lemma 7.4]{rossi-roubi},
it can be checked that  $\mathcal{R}_{k-1}$ is a
  pseudo-monotone operator  (according to \cite[Chap.\ II, Def.\ 2.1]{roub-NPDE})  on
$H^1(\Omega)\times H_0^1(\Omega;\R^d) \times W^{1,p}(\Omega)$,  coercive on that space. Therefore,
the Leray-Lions type existence result of \cite[Chap.\ II, Thm.\ 2.6]{roub-NPDE} applies, yielding the existence of a solution $(\wtau k, \utau k, \chitau k)$ to \eqref{abstract-sistemone}.

\noindent
\textbf{Step $2$: non-negativity of $\wtau{k}$.}
Let us assume in addition  that $g\geq 0$ a.e.\
in $\Omega \times (0,T)$ and
$ w_0\geq 0$ a.e.\ in $\Omega$.
Then $\gtau{k} \geq 0$ a.e. in $\Omega$. To prove \eqref{pos-wk-rev}, we
proceed by induction on $k$ and show that, if $\wtau{k-1} \geq 0$
a.e. in $\Omega$, then $\wtau{k} \geq 0$ a.e. in $\Omega$. Indeed,
let us test \eqref{eq-discr-w} by $-(\wtau{k})^-$. Taking into
account  the definition  \eqref{K-T} of $\Theta$, we have that
\[
\int_{\Omega}\Theta(\wtau{k}) \left((\chitau{k}-\chitau{k-1})+\rho\dive\left(\frac{\utau{k}-\utau{k-1}}{\tau}\right)\right)
(-(\wtau{k})^-) \dd x = 0
\]
 (here we have  kept $\rho \in \R$ also to encompass the case with thermal expansion, cf.\ below).
Combining this with the inequality  \eqref{quoted-later-on}
and noting that
$(\wtau{k-1})^- =0$  a.e. in $\Omega$, also in view of
\eqref{conseq-2} we obtain
\[
\frac1{2\tau}\int_{\Omega}|(\wtau{k})^-|^2 \dd x + c_2 \int_{\Omega}
|\nabla (\wtau{k})^-|^2 \dd x \leq - \int_{\Omega}
\gtau{k}(\wtau{k})^- \dd x \leq  0,
\]
yielding $(\wtau{k})^-=0$ a.e. in $\Omega$, whence \eqref{pos-wk-rev}.
{Under the additional Hypothesis (VIII) (which gives \eqref{ellipticity-retained}),
 an analogous proof of existence of solutions
can be given for Problem \ref{prob:rhoneq0}, hence we omit to give the details.}
\end{proof}

\subsection{A priori estimates}
\label{ss:3.2}
\paragraph{Notation and auxiliary results.}
 Hereafter, for a given Banach space $X$ and a
$K_\tau$-tuple $( b_\tau^k )_{k=1}^{K_\tau} \subset X$, we shall use
the short-hand notation
\[
\dtau{k}{b}:= \frac{b_\tau^k-b_\tau^{k-1}}{\tau}, \qquad
\duetau{k}{b}:= \dtau{k}{\dtau{k}{b}}= \frac{b_\tau^k -2
b_\tau^{k-1} + b_\tau^{k-2}}{\tau^2}.
\]
We recall the well-known \emph{discrete by-part integration} formula
\begin{equation}
\label{discr-by-part} \sum_{k=1}^{K_\tau} \tau \dtau{k}{b} \vtau{k}=
\btau{K_\tau} \vtau{K_\tau} -\btau{0}\vtau{1}
-\sum_{k=2}^{K_\tau}\tau \btau{k-1}\dtau{k}{v} \quad \text{for all
$\{b_\tau^k \}_{k=1}^{K_\tau},\, \{ v_\tau^k \}_{k=1}^{K_\tau} \subset X$.}
\end{equation}

We consider
 the left-continuous and  right-continuous piecewise constant, and the piecewise linear interpolants
 of the values  $\{ b_\tau^k
\}_{k=1}^{K_\tau}$, namely the functions
\[
\left.
\begin{array}{llll}
& \pwc  b{\tau}: (0,T) \to X  & \text{defined by}  & \pwc
b{\tau}(t): = b_\tau^k,
\\
& \upwc  b{\tau}: (0,T) \to X  & \text{defined by}  & \upwc
b{\tau}(t) := b_\tau^{k-1},
\\
 &
\pwl  b{\tau}: (0,T) \to X  & \text{defined by} &
 \pwl b{\tau}(t):
=\frac{t-t_\tau^{k-1}}{\tau} b_\tau^k +
\frac{t_\tau^k-t}{\tau}b_\tau^{k-1}
\end{array}
\right\}
 \qquad \text{for $t \in
(t_\tau^{k-1}, t_\tau^k]$.}
\]
We also introduce the piecewise linear interpolant    of the values
$\{ (b_\tau^k - b_{\tau}^{k-1})/\tau\}_{k=1}^{K_\tau}$ (namely, the
values taken by  the -piecewise constant- function $\pwl
{b}{\tau}'$), viz.
\[
\pwwll  b{\tau}: (0,T) \to X \qquad \pwwll b{\tau}(t)
:=\frac{t-t_\tau^{k-1}}{\tau} \frac{b_\tau^k - b_{\tau}^{k-1}}{\tau}
+ \frac{t_\tau^k-t}{\tau} \frac{b_\tau^{k-1} - b_{\tau}^{k-2}}{\tau}
\qquad \text{for $t \in (t_\tau^{k-1}, t_\tau^k]$.}
\]
Note that $ {\pwwll  {b}{\tau}}'(t) =\duetau{k}{b}$ for $t \in
(t_\tau^{k-1}, t_\tau^k]$.

In view of \eqref{bulk-force}, \eqref{heat-source}, and \eqref{fixed-profile},
it is easy to check that the piecewise constant interpolants $ (\pwc
{\mathbf{f}}{\tau} )_{k=1}^{K_\tau}$, $(\pwc  g{\tau}
)_{k=1}^{K_\tau}$, $(\pwc {\Theta^*}{\tau} )_{k=1}^{K_\tau}$, and
$(\pwl {\Theta^*}{\tau} )_{k=1}^{K_\tau}$
 of the values $\ftau{k}$,
$\gtau{k}$ \eqref{local-means}, and $\Thetatau{k} $ \eqref{local-means-theta}, fulfill as
$\tau \down 0$
\begin{subequations}
\label{converg-interp}
\begin{align}
 \label{converg-interp-f}  & \pwc {\mathbf{f}}{\tau}  \to \mathbf{f}
 & \text{ in $L^2(0,T;L^2(\Omega;\R^d))$,}
 \\
\label{converg-interp-g}  & \pwc g{\tau}  \to g
 & \text{ in $L^1(0,T;L^1(\Omega))\cap L^2(0,T;H^1(\Omega)')$,}
 \\
 \label{converg-interp-Theta}
 &
 \pwc {\Theta^*}{\tau} \to \Theta^* & \text{ in $L^p(0,T;L^2(\Omega))$ for all $1\leq p <\infty$,}
 \\
 &
 \label{converg-interp-Theta-d}
  \| \partial_t \pwl {{\Theta^*}}{\tau} \|_{L^1(0,T;L^2(\Omega))} \leq 2
 \|\partial_t \Theta^*
 \|_{L^1(0,T;L^2(\Omega))} & \text{ for all $\tau>0$.}
\end{align}
\end{subequations}

 Finally, we shall denote by $\pwc{\mathsf{t}}{\tau}$ and by
$\upwc{\mathsf{t}}{\tau}$ the left-continuous and right-continuous
piecewise constant interpolants associated with the partition, i.e.
 $\pwc{\mathsf{t}}{\tau}(t) := t_\tau^k$ if $t_\tau^{k-1}<t \leq t_\tau^k $
and $\upwc{\mathsf{t}}{\tau}(t):= t_\tau^{k-1}$ if $t_\tau^{k-1}
\leq t < t_\tau^k $. Clearly, for every $t \in [0,T]$ we have
$\pwc{\mathsf{t}}{\tau}(t) \downarrow t$ and
$\upwc{\mathsf{t}}{\tau}(t) \uparrow t$ as $\tau\to 0$.

 Propositions \ref{prop:aprio} and \ref{prop:aprio-2} collect
in the cases $\rho =0$ and $\rho \neq 0$
several a priori estimates
on the approximate solutions, obtained by interpolation of the discrete solutions to
Problems \ref{probk-rev}, \ref{probk-irr},  \ref{probk-irr-iso},
and Problem \ref{prob:rhoneq0}, respectively.
\begin{proposition}[$\mu \in \{0,1\}, \, \rho=0$]
\label{prop:aprio}
 Let $\rho=0$.  Assume Hypotheses (I)--(V) and  \eqref{bulk-force}--\eqref{datochi} on the data $\mathbf{f},\,g,\,
\teta_0,\, \uu_0,\,\vv_0,\, \chi_0$. Then,
\begin{enumerate}
\item
in the case $\mu \in \{0,1\}$
 there exist a constant $S>0$ such that for the interpolants of the solutions to
 Problem \ref{probk-rev} and to  Problem \ref{probk-irr} there holds:
\begin{align}
& \label{aprio1}
 \sup_{\tau>0}\|\pwl \uu{\tau}\|_{H^1(0,T;\boY) \cap
W^{1,\infty}(0,T;\boZ)}
 \leq S,
\\
& \label{aprio1bis}
 \sup_{\tau>0}\|\pwc \uu{\tau}\|_{L^\infty(0,T;\boY)}
 \leq S,
\\
& \label{aprio2} \sup_{\tau>0}\|\pwwll \uu{\tau}
\|_{H^1(0,T;L^{2}(\Omega;\R^d))} \leq S,
\\
& \label{aprio3}\sup_{\tau>0} \|\pwc
\chi{\tau}\|_{L^\infty(0,T;W^{1,p}(\Omega))} \leq S,
\\
& \label{aprio4} \sup_{\tau>0}\|\pwl
\chi{\tau}\|_{L^\infty(0,T;W^{1,p}(\Omega)) \cap H^1
(0,T;L^2(\Omega))} \leq S,
\\
& \label{aprio6} \sup_{\tau>0}\|\pwc
\w{\tau}\|_{L^{\infty}(0,T;L^{1}(\Omega))} \leq S,
\\
& \label{aprio7}\sup_{\tau>0} \|\pwc
\w{\tau}\|_{L^{r}(0,T;W^{1,r}(\Omega))} \leq S \quad \text{for every
$1\leq r <\frac{d+2}{d+1}$,}
\\
& \label{aprio5} \sup_{\tau>0}\|\pwc
\w{\tau}\|_{\mathrm{BV}([0,T];W^{1,r'}(\Omega)^*)} \leq S,
\\
& \label{aprio8}\sup_{\tau>0} \|\Theta(\pwc
\w{\tau})\|_{L^{2+\epsilon}(0,T;L^{2+\epsilon}(\Omega))} \leq S
\quad \text{for any $0<\epsilon<\frac{\sigma (d+2)}{d}-2$.}
\end{align}
\item if $\mu=0$ in addition there exists $S'>0$ such that
\begin{align}
& \label{aprio9}  \sup_{\tau>0} \left(\|\opchi(\pwc \chi{\tau})
\|_{L^2 (0,T;L^2(\Omega))} + \| \pwc \xi{\tau}\|_{L^2
(0,T;L^2(\Omega))} \right) \leq S'.
\end{align}
Moreover, if $\phi$ also fulfills
Hypothesis (VI), then  
\begin{align}
& \label{aprio9bis}  \sup_{\tau>0} \|\pwc \chi{\tau} \|_{L^2
(0,T;W^{1+\sigma,p}(\Omega))}\leq S' \quad \text{for every $0<
\sigma <\frac1p$.}
\end{align}
\item in the isothermal case
with $\mu=1$, if   $b{''} \equiv 0$  (cf.\ \eqref{bsecondnull})   and $\phi$ also fulfills
Hypothesis (VI),
 estimates
\eqref{aprio1}--\eqref{aprio4} hold. Moreover,
 there exists
$S{''}>0$ such that  for (the interpolants of) the solutions to Problem
 \ref{probk-irr-iso}  
 \begin{align}
& \label{aprio11}
 \sup_{\tau>0} \left(\|\opchi(\pwc \chi{\tau})
\|_{L^\infty (0,T;L^2(\Omega))} + \| \beta_{\tau}(\pwc
\chi{\tau})\|_{L^\infty (0,T;L^2(\Omega))} \right) \leq S{''},
\\
& \label{aprio11bis}  \sup_{\tau>0} \|\pwc \chi{\tau} \|_{L^\infty
(0,T;W^{1+\sigma,p}(\Omega))}\leq S{''} \quad \text{for every $0<
\sigma <\frac1p$,}
\\
& \label{aprio12}  \sup_{\tau>0} \|\pwl \chi{\tau}
\|_{W^{1,\infty}(0,T;L^2(\Omega))}\leq S{''},
\\
& \label{aprio12bis}  \sup_{\tau>0} \|\pwc \zeta{\tau} \|_{L^\infty
(0,T;L^2(\Omega))}\leq S{''}.
 \end{align}
 \end{enumerate}
The constants in \eqref{aprio7}, \eqref{aprio8}, and \eqref{aprio9bis}, \eqref{aprio11bis} also depend on the parameters $r$, $\epsilon$, and $\sigma$, respectively.
 \end{proposition}
 \begin{proposition}[$\mu=0$, \, $\rho\neq 0$]
 \label{prop:aprio-2}  Let $\mu=0$ and $\rho \neq 0$.
Assume Hypotheses (I), (III)--(V), and Hypothesis (VIII); suppose that
the data $\mathbf{f},\,g,\,
\teta_0,\, \uu_0,\,\vv_0,\, \chi_0$ comply with
  \eqref{bulk-force}--\eqref{datochi}, and in addition that $\w_0 \in L^2(\Omega)$.
Then,
for the interpolants of the solutions to
 Problem  \ref{prob:rhoneq0}
  estimates \eqref{aprio3}--\eqref{aprio6}
hold with a constant   \emph{independent} of $M$,
whereas estimates \eqref{aprio1}--\eqref{aprio2}, \eqref{aprio9} and (under the additional
Hypothesis (VI)) \eqref{aprio9bis} hold for a constant  \emph{depending} on $M$.
 Moreover,
 there exists a constant $S^{'''}= S^{'''}(M)>0$
such that
 \begin{align}
& \label{aprio6-rho} \sup_{\tau>0}\|\pwc
\w{\tau}\|_{L^{2}(0,T;H^1(\Omega)) \cap L^{\infty}(0,T;L^{2}(\Omega))} \leq S^{'''},
\\
& \label{aprio5-rho}   \sup_{\tau>0}\|\pwl
\w{\tau}\|_{H^1(0,T;H^1(\Omega)')}  \leq S^{'''}.
\end{align}
\end{proposition}
\noindent
We will
treat the proofs
of Propositions \ref{prop:aprio} and \ref{prop:aprio-2} in a unified way,
developing a series of a priori  estimates.
\paragraph{Proof of Proposition \ref{prop:aprio}.}
Most of the calculations below
will be detailed on the
discretization scheme \eqref{eq-discr-w}--\eqref{eq-discr-chi} for the full \emph{reversible}
system, and whenever necessary we will  outline the differences
in comparison with the discrete systems of Problems \ref{probk-irr} and \ref{probk-irr-iso}.
Furthermore, for each estimate we will specify the values of the parameters $\mu$ and $\rho$
for which it is valid and,
to make the computations
more readable, we will  illustrate them  first on the
time-continuous level, i.e.\ referring to system \eqref{eq0d}--\eqref{eq2d}.

\noindent
\textbf{First a priori estimate for $\mu \in \{0,1\}$, $\rho \in \R$:} \emph{we test \eqref{eq1d} by $\uu_t$
 \eqref{eq0d} by $1$, \eqref{eq2d} by $\chi_t$,
add them and integrate in time. This is the so-called {\sl energy estimate}.}
 We test \eqref{eq-discr-u} by
$\utau{k}-\utau{k-1}$. Note that
\begin{equation}
\label{est-1.1} \tau\int_{\Omega} \duetau{k}{\uu} \cdot
\dtau{k}{\uu}\dd x \geq \frac1{2} \|\dtau{k}{\uu}\|_{L^2(\Omega)}^2
-\frac1{2} \|\dtau{k-1}{\uu}\|_{L^2(\Omega)}^2
\end{equation}
for all $k=1,\ldots,K_\tau$. Since $\chitau{k} \in [0,1]$ a.e.\ in
$\Omega$, thanks to \eqref{data-a}
 we have that $a(\chitau{k}) \geq 0 $  a.e.\ in
$\Omega$, thus by \eqref{korn} we have
\begin{equation}
\label{est-1.2} \pairing{}{H^1(\Omega)}{\opj{(a(\chitau{k})
+\delta)}{\dtau{k}{\uu}}}{\utau{k}-\utau{k-1}} \geq  C_1 \delta \tau
\| \dtau{k}{\uu}\|_{H^1(\Omega)}^2.
\end{equation}
On the other hand, using that $\|b(\chitau{k})\|_{L^\infty(\Omega)}
\leq \|b\|_{L^\infty (0,1)}$ and taking into account
\eqref{a:conti-form}, we find
\begin{equation}
\label{est-1.3}
\begin{aligned}
\!\! \!\!\!\!\!\!
|\pairing{}{H^1(\Omega)}{\oph{b(\chitau{k})
}{\utau{k}}}{\utau{k}-\utau{k-1}}| & \leq C_2 \tau \|b\|_{L^\infty
(0,1)} \|\utau{k} \|_{H^1(\Omega)} \|\dtau{k}{\uu} \|_{H^1(\Omega)}
\\ & \leq \frac12 C_1 \delta \tau \| \dtau{k}{\uu}\|_{H^1(\Omega)}^2 +
C\tau \| \utau{0}\|_{H^1(\Omega)}^2  + C_\delta\tau \|\utau{k}- \utau{0}
\|_{H^1(\Omega)}^2.
\end{aligned}
\end{equation}
We estimate the latter term by observing that
\begin{equation}
\label{est-1.4} \|\utau{k}- \utau{0} \|_{H^1(\Omega)}^2 =
\|\sum_{j=1}^k(\utau{j}- \utau{j-1}) \|_{H^1(\Omega)}^2 \leq k
\tau^2 \sum_{j=1}^k \| \dtau{j}{\uu}\|_{H^1(\Omega)}^2\leq T\tau
\sum_{j=1}^k \| \dtau{j}{\uu}\|_{H^1(\Omega)}^2\,.
\end{equation}
Altogether, collecting \eqref{est-1.1}--\eqref{est-1.4}  and summing
over the index $k=1,\ldots,K_\tau$, we conclude
\begin{align}\label{est-u-rho}
&\frac1{2} \|\dtau{K_\tau}{\uu}\|_{L^2(\Omega)}^2 +  \frac12 C_1
\delta \sum_{k=1}^{K_\tau} \tau \| \dtau{k}{\uu}\|_{H^1(\Omega)}^2
 +\rho\sum_{k=1}^{K_\tau}\tau \int_\Omega \Theta(\wtau{k}) \dive(\dtau{k}{\uu})\dd x\\
\no
&\leq  \frac1{2} \|\dtau{0}{\uu}\|_{L^2(\Omega)}^2  + C
\sum_{k=1}^{K_\tau} \tau \left(\sum_{j=1}^k \tau \|
\dtau{j}{\uu}\|_{H^1(\Omega)}^2\right).
\end{align}

We multiply \eqref{eq-discr-chi} by
$\chitau{k} - \chitau{k-1}$. With standard convexity inequalities,
we  obtain
\begin{equation}
\label{to-be-quoted-later}
\begin{aligned}
 \tau &\| \dtau{k}{\chi} \|_{L^2(\Omega)}^2
  + \Phi(\chitau{k}) + \int_{\Omega}  \widehat{\beta}(\chitau{k}) \dd
x + \int_{\Omega} \gamma(\chitau{k})(\chitau{k}-\chitau{k-1}) \dd
x\\ & \leq \Phi(\chitau{k-1}) + \int_{\Omega}
\widehat{\beta}(\chitau{k-1}) \dd x +\tau
\int_{\Omega}\dtau{k}{\chi} \left(\Theta(\wtau{k}) - \frac12
b'(\chitau{k-1}) \eps(\utau{k}) \mathrm{R}_e \eps(\utau{k}) \right)
\dd x.
\end{aligned}
\end{equation}
We then test \eqref{eq-discr-w} by $\tau$ and add the resulting
relation to   \eqref{est-u-rho} and  \eqref{to-be-quoted-later},
 summing over the index
$k=1,\ldots,K_\tau$. The terms
$\tau\int_{\Omega}\dtau{k}{\chi} \Theta(\wtau{k})\dd x $ and
$\rho \tau \int_\Omega \Theta(\wtau{k}) \dive(\dtau{k}{\uu})\dd x$ cancel out. Furthermore,
we note that $|b'(\chitau{k-1})| \leq C $ a.e.\ in $\Omega$ since $b \in
 \mathrm{C}^1(\R)$ and $0 \leq \chitau{k-1} \leq 1$ a.e. in
 $\Omega$,
and  exploit the Lipschitz continuity of the function $\gamma$, which
enables us to estimate the last term on the left-hand side of
\eqref{to-be-quoted-later}. Ultimately,
 we obtain
\begin{equation}
\label{est-2.1}
\begin{aligned}
& \int_\Omega \wtau{K_\tau} \dd x+  \sum_{k=1}^{K_\tau} \tau   \|
\dtau{k}{\chi} \|_{L^2(\Omega)}^2 + \Phi(\chitau{K_{\tau}}) +
\int_{\Omega} \widehat{\beta}(\chitau{K_{\tau}}) \dd x+ \frac1{2} \|\dtau{K_\tau}{\uu}\|_{L^2(\Omega)}^2 \\
\no
&\quad+  \frac12 C_1
\delta \sum_{k=1}^{K_\tau} \tau \| \dtau{k}{\uu}\|_{H^1(\Omega)}^2\\ &  \leq
\int_\Omega w_0 \dd x+ \Phi(\chitau{0}) + \int_{\Omega} \widehat{\beta}(\chitau{0}) \dd x +
\sum_{k=1}^{K_\tau} \tau \| \gtau{k} \|_{H^1(\Omega)^*}\\
& \quad+
 \sum_{k=1}^{K_\tau} C \tau ( \|\chitau{k}
\|_{L^2(\Omega)} + \|| \eps(\utau{k})|^2 \|_{L^2(\Omega)} +1 ) \|
\dtau{k}{\chi}\|_{L^2(\Omega)}
\\ &
\leq C + \frac14 \sum_{k=1}^{K_\tau}\tau  \| \dtau{k}{\chi}
\|_{L^2(\Omega)}^2 + C \sum_{k=1}^{K_\tau}\tau \| \eps(\utau{k})
\|_{L^4(\Omega)}^4 +  C \sum_{k=1}^{K_\tau} \tau \left(\sum_{j=1}^k
\tau \| \dtau{j}{\chi}\|_{L^2(\Omega)}^2\right)\\
\no
&\quad+ C
\sum_{k=1}^{K_\tau} \tau \left(\sum_{j=1}^k \tau \|
\dtau{j}{\uu}\|_{H^1(\Omega)}^2\right),
\end{aligned}
\end{equation}
where the last inequality follows  from assumptions
\eqref{heat-source}--\eqref{datochi} on the data, from  the Young
inequality, and from
 estimating $\tau\|\chitau{k} \|_{L^2(\Omega)}^2 \leq  2\tau\|\chitau{0} \|_{L^2(\Omega)}^2+
 2 \tau\|\chitau{k}-\chitau{0} \|_{L^2(\Omega)}^2
$ and dealing with the latter term like in \eqref{est-1.4}.
Therefore,  applying a discrete version of
the Gronwall lemma  (cf., e.g., \cite[Prop. 2.2.1]{jerome}), we
conclude estimates   \eqref{aprio3}--\eqref{aprio6}, as well as
 estimate
\begin{equation}
\label{aprio0} \sup_{\tau>0} \|\pwl \uu{\tau}\|_{H^1(0,T;\boZ) \cap
W^{1,\infty}(0,T;\Ha)} \leq S,
\end{equation}
 which in turn implies
\begin{equation}
\label{aprio-1} \sup_{\tau>0} \|\pwc \uu{\tau}\|_{L^\infty(0,T;\boZ)}
\leq S.
\end{equation}

We can perform  this energy estimate  on Problem \ref{probk-irr} as well:
calculations \eqref{est-1.1}--\eqref{est-u-rho} can be trivially adapted to \eqref{eq-discr-u-irr},
 whereas
\eqref{to-be-quoted-later} derives from choosing in the minimum
problem \eqref{eq-discr-chi} $\chitau {k-1}$ as a competitor.
We again conclude \eqref{aprio3}--\eqref{aprio6} as well as
\eqref{aprio0}--\eqref{aprio-1}.

 In the case of  Problem \ref{probk-irr-iso},
 \eqref{eq-discr-chi-irr-iso} also features the term $\zetau{k}$,
whence the additional term $ \int_{\Omega}\zetau{k}
\dtau{k}{\chi} \dd x$ on the left-hand side of \eqref{to-be-quoted-later}.
 Since $0 \in \alpha(0)$,  by monotonicity the latter term in nonnegative. Taking into account
 this,  replacing $\widehat{\beta}$ with $\widehat{\beta}_\tau$
 in \eqref{to-be-quoted-later},
and observing that the coefficient of $\eps(\utau{k}) \mathrm{R}_e
\eps(\utau{k}) $ on the right-hand side of \eqref{eq-discr-chi-irr}
is bounded,
  we may repeat the same calculations
 as in the above lines.
The coercivity estimate \eqref{est-1.2} goes through because
$a(\chitau{k})$, which is no longer guaranteed to be positive, is
replaced by $a(\chitau{k})^+$. Furthermore, since $\chitau{k} \leq
\chitau{k-1}\leq \chi_0 \leq 1$ a.e. in $\Omega$ (due to the
irreversibility constraint), we have that $(\chitau{k})^+ \in
[0,1]$ a.e.\ in $\Omega$, thus we may again obtain \eqref{est-1.3}.

\noindent
\textbf{Second a priori estimate for $\mu \in \{0,1\}$, $\rho=0$:} \emph{
following \cite{bss} (see also \cite{rocca-rossi1}), we
test~\eqref{eq1d} by $- \dive (\eps(\uu_t))$ and integrate
in time.}
 We test
\eqref{eq-discr-u} by $ -{\dive} (\eps(\utau{k}-\utau{k-1})) $. This
gives rise to the following terms on the left-hand side:
\begin{align}
 \label{est-3.0}
 &
\begin{aligned}
 & -\tau\int_{\Omega} \duetau{k}{\uu} \cdot
\dive(\eps(\dtau{k}{\uu})) \dd x \geq \frac12 \int_{\Omega}
|\eps(\dtau{k}{\uu})|^2 \dd x - \frac12 \int_{\Omega}
|\eps(\dtau{k-1}{\uu})|^2 \dd x,
\end{aligned}
\\
&
 \label{est-3.-1}
\begin{aligned}
  -\tau \int_{\Omega} \opj{(a(\chitau{k}) + \delta)}{\dtau{k}{\uu}}
\cdot \dive(\eps(\dtau{k}{\uu})) \dd x &  =  \tau\int_{\Omega}
(\delta+a(\chitau{k})) \dive(\eps(\dtau{k}{\uu}))  \cdot
\dive(\eps(\dtau{k}{\uu})) \dd x \\
 & +\tau \int_{\Omega}
\eps(\dtau{k}{\uu})\nabla a({\chitau{k}}) \cdot
\dive(\eps(\dtau{k}{\uu})) \dd x \\ & \doteq I_0+I_1  \geq \delta
C_3^2 \tau \|\dtau{k}{\uu} \|_{H^2(\Omega)}^2 + I_1,
\end{aligned}
\end{align}
the latter inequality due to \eqref{cigamma}. Moreover, always on the  l.h.s.\ we have
\begin{align}
 & \label{est-3.-2}
\begin{aligned}
 -  \tau \int_{\Omega} \oph{b(\chitau{k})}{\utau{k}} \cdot
\dive(\eps(\dtau{k}{\uu})) \dd x  & =  \lambda_1  \tau  \int_{\Omega}
b(\chitau{k}) \Delta(\dtau{k}{\uu}) \cdot \nabla (\dive(\utau{k}))
\dd x \\ & + 2 \lambda_2  \tau \int_{\Omega} b(\chitau{k}) {\dive}
(\eps(\dtau{k}{\uu}))  \cdot  {\dive} (\eps(\utau{k})) \dd x\\ &  +
\lambda_1  \tau\int_{\Omega} \dive(\utau{k}) \nabla b(\chitau{k}) \cdot
\Delta(\dtau{k}{\uu})\dd x \\ &  + 2 \lambda_2    \tau\int_{\Omega} \eps(\utau{k})
\nabla  b(\chitau{k})\cdot {\dive} (\eps(\dtau{k}{\uu})) \dd x\\ &  \doteq
I_2 +I_3 +I_4 +I_5,
\end{aligned}
\end{align}
(where $\Delta$ stands for the vectorial Laplace operator). On the
right-hand side, we have
\begin{equation}
\label{est-3.1}
 -\tau\int_{\Omega} \ftau{k}\cdot
\dive(\eps(\dtau{k}{\uu}))\dd x  \leq C_\delta \tau \|
\ftau{k}\|_{L^2(\Omega)}^2 + \frac{\delta}{8} C_3^2\tau \| \dtau{k}{\uu}
\|_{H^2(\Omega)}^2,
\end{equation}
where the latter inequality follows from \eqref{cigamma}.  We now
move the integral terms $I_1,\ldots,I_5$ to the right-hand side. Let
us fix $0<\varsigma\leq 3$ such that $p \geq d+\varsigma$  (where $p$ is
the exponent  in \eqref{datad2}). Then,
\begin{align}
 \label{est-3.2}
 \begin{aligned}
| I_1|  & \leq \tau \|{\dive} (\eps(\dtau{k}{\uu})) \|_{L^2(\Omega)}
\|\eps(\dtau{k}{\uu})\|_{L^{d^{\star}{-}\varsigma}(\Omega)}
 \|\nabla a(\chitau{k})\|_{L^{d+\varsigma}(\Omega)}
 \\ &  \leq
  \frac{\delta}4 C_{3}^2 \tau \| \dtau{k}{\uu} \|_{H^2(\Omega)}^2
 +\delta\tau \| \eps (\dtau{k}{\uu})\|_{L^{d^{\star}{-}\varsigma} (\Omega)}^2
 \| \nabla a(\chitau{k}) \|_{L^{d+\varsigma}(\Omega)}^2
 \\ &
 \begin{aligned}
 \leq
\frac{\delta}4 C_{3}^2 \tau \| \dtau{k}{\uu} \|_{H^2(\Omega)}^2  & +
\varrho^2 C\tau \| \dtau{k}{\uu} \|_{H^2(\Omega)}^2 \|
a'\|_{L^\infty (-\mathsf{m},\mathsf{m})}^2 \| \nabla \chitau{k} \|_{L^{p}(\Omega)}^2
\\ & + C_{\varrho,\delta}^2 C' \| \dtau{k}{\uu} \|_{L^2(\Omega)}^2 \|
a'\|_{L^\infty (-\mathsf{m},\mathsf{m})}^2  \| \nabla \chitau{k} \|_{L^{p}(\Omega)}^2
\end{aligned}
\\
& \leq \frac{\delta}2 C_{3}^2 \tau \| \dtau{k}{\uu}
\|_{H^2(\Omega)}^2 + C S^4,
\end{aligned}
\end{align}
where the first and second inequalities respectively follow from the
H\"older and  Young inequalities, with $d^\star$ as in
\eqref{dstar}, the third one from \eqref{interp}, and the last one
taking into account estimates  \eqref{aprio0}  for
$\sup_{k=1,\ldots,K_\tau}\| \dtau{k}{\uu} \|_{L^2(\Omega)}$,
\eqref{aprio3} for $\sup_{k=1,\ldots,K_\tau}\| \chitau{k}
\|_{W^{1,p}(\Omega)}$, which in particular yields
that  $|\chitau{k}| \leq
\mathsf{m}$ a.e.\ in $\Omega \times (0,T)$ for some $\mathsf{m}>0$, and from choosing $\varrho \leq
C^{-1/2} (\| a'\|_{L^\infty (-\mathsf{m},\mathsf{m})} S)^{-1}$. Furthermore, taking
into account that $b(\chitau k) \in L^\infty(\Omega)$,
 one easily
checks that
\begin{align}
 \label{est-3.3}
\begin{aligned}
|I_2 + I_3|  & \leq C \tau \| \dtau{k}{\uu}\|_{H^2(\Omega)} \|
\utau{k}\|_{H^2(\Omega)} \\ & \leq \frac{\delta}8C_{3}^2 \tau \|
\dtau{k}{\uu} \|_{H^2(\Omega)}^2 + C \|\utau{0}\|_{H^2(\Omega)}^2 +
C \tau \sum_{j=1}^k \tau \| \dtau{j}{\uu}\|_{H^2(\Omega)}^2,
\end{aligned}
\end{align}
where the second inequality follows from the Young inequality, from
$ \tau\| \utau{k}\|_{H^2(\Omega)}^2 \leq 2\tau \|
\utau{0}\|_{H^2(\Omega)}^2 + 2\tau\|
\utau{k}-\utau{0}\|_{H^2(\Omega)}^2$, and from estimating the latter
term as in \eqref{est-1.4}. Analogously, again  using
that $\sup_{k=1,\ldots,K_\tau}\| \chitau{k} \|_{W^{1,p}(\Omega)}\leq
S$ and that $|\chitau {k}| \leq \mathsf{m}$ a.e.\ in $\Omega$, we have
\begin{equation}
 \label{est-3.4}
 \begin{aligned}
|I_4+I_5| &  \leq \tau \| \dtau{k}{\uu}\|_{H^2(\Omega)} \|\nabla
b(\chitau{k})\|_{L^3(\Omega)} (\| \dive(\utau{k}) \|_{L^6 (\Omega)}
+ \| \eps(\utau{k})\|_{L^6 (\Omega)}) \\ & \leq C \tau
\|b'\|_{L^\infty (-\mathsf{m},\mathsf{m})}\| \dtau{k}{\uu}\|_{H^2(\Omega)} \|\chitau{k}
\|_{W^{1,p}(\Omega)} \| \utau{k}\|_{H^2(\Omega)}\\ &  \leq
\frac{\delta}8 C_{3}^2 \tau \| \dtau{k}{\uu} \|_{H^2(\Omega)}^2 + C
S^2 \|\utau{0}\|_{H^2(\Omega)}^2  + CS^2\tau \sum_{j=1}^k \tau \|
\dtau{j}{\uu}\|_{H^2(\Omega)}^2.
\end{aligned}
\end{equation}
Collecting \eqref{est-3.0}--\eqref{est-3.4} and summing over the
index $k=1,\ldots,K_\tau$, we obtain
\[
\begin{aligned}
\frac12 \|\eps(\dtau{K_\tau}{\uu})\|_{L^2 (\Omega)}^2 &  +
  \frac{C_3^2\delta}8 \sum_{k=1}^{K_\tau} \tau
\|\dtau{k}{\uu}\|_{H^2(\Omega)}^2  \\ &  \leq   C   +  \frac12
\|\eps(\dtau{0}{\uu})\|_{L^2 (\Omega)}^2   + C \sum_{k=1}^{K_\tau}
\tau \| \ftau{k}\|_{L^2(\Omega)}^2 + C \sum_{k=1}^{K_\tau} \tau
\sum_{j=1}^k \tau \| \dtau{j}{\uu}\|_{H^2(\Omega)}^2.
\end{aligned}
\]
Applying the discrete Gronwall Lemma once again, we conclude
estimate \eqref{aprio1}, whence \eqref{aprio1bis}.

It is immediate to check that calculations
\eqref{est-3.0}--\eqref{est-3.4} can also be performed on  the
discrete momentum equation \eqref{eq-discr-u-irr} in Problem
\ref{probk-irr-iso}.
\begin{remark}
\label{remk:added-revision}
\upshape
The calculations for the \emph{Second a priori estimate}
carry over to the case the operator $\opchi$ is replaced by the $s$-Laplacian $A_s$, provided that $s>\frac d 2$. Indeed, this
ensures the continuous embedding
$W^{s,2}(\Omega)\subset W^{1,\bar{p}}(\Omega)$ for some $\bar{p}>d$, which is crucial in the above calculations, cf.\
\eqref{est-3.2}.
\end{remark}
\noindent
\textbf{Third a priori estimate for $\mu \in \{0,1\}$, $\rho=0$:} \emph{\textsc{Boccardo\&Gallou\"et}-type estimate on \eqref{eq0d}.}
As in the proof of  \cite[Prop.~4.2]{roubicek-SIAM}, we test
equation \eqref{eq-discr-w} by $\Pi(\wtau{k})$, where
\begin{equation}
\label{speciaL-test-func}
\Pi : [0,+\infty) \to [0,1]  \text{ is defined by } \Pi(\w)=
1-\frac1{(1+\w)^{\varsigma}} \quad \text{ for some }\varsigma>0.
\end{equation}
Note that $\Pi(\wtau{k})$ is well-defined, since $\wtau k \geq 0$
a.e.\ in $\Omega$, and it belongs to $ H^1(\Omega)\cap L^\infty(\Omega)$, as $\Pi$ is
Lipschitz continuous. Such a test function has been first proposed
in \cite{feireisl-malek}, as a simplification of the technique by
\textsc{Boccardo\&Gallou\"et} \cite{boccardo-gallouet1}. We shall
denote by $\widehat{\Pi}$ the primitive of $\Pi$ such that
$\widehat{\Pi}(0)=0$ (hence $\widehat{\Pi}(\w) \geq 0$ for $\w\geq
0$).
 Summing
over $k=1,\ldots,K_\tau$, we obtain
\[
\begin{aligned}
c_2\varsigma \sum_{k=1}^{K_\tau} \tau \int_{\Omega} \frac{|\nabla
\wtau{k}|^2}{(1+\wtau{k})^{\varsigma+1}} \dd x &  \leq
\int_{\Omega}\widehat{\Pi}(\wtau{K_{\tau}}) \dd x +
\sum_{k=1}^{K_\tau}\int_{\Omega} K(\wtau{k-1}) \nabla \wtau{k} \cdot
\nabla \Pi(\wtau{k}) \dd x \\ & \leq
\int_{\Omega}\widehat{\Pi}(\wtau{0}) \dd x + \sum_{k=1}^{K_\tau}
\tau(\|\gtau{k}\|_{L^1(\Omega)} + \| \dtau{k}{\chi}
\Theta(\wtau{k})\|_{L^1(\Omega)}) \|\Pi(\wtau{k})\|_{L^\infty
(\Omega)},
\end{aligned}
\]
where the first inequality follows from \eqref{conseq-2}$_1$, the fact
that  $\nabla \Pi(\wtau{k}) = \varsigma (\nabla \wtau{k})/
(1+\wtau{k})^{\varsigma+1}$, and the second one from  the convex
analysis inequality $\int_{\Omega}\Pi(\wtau{k})(\wtau{k}-\wtau{k-1})
\dd x \geq
\int_{\Omega}(\widehat{\Pi}(\wtau{k})-\widehat{\Pi}(\wtau{k-1})) \dd
x $ and from the fact that, due to assumption \eqref{datoteta}, we have
$$\int_{\Omega}\widehat{\Pi}(\wtau{0}) \dd x \leq C\left(\|\wtau{0}\|_{L^1(\Omega)}+1\right)\leq C.$$
Taking into account that $0 \leq \Pi(\wtau{k}(x)) \leq 1 $ for
almost all $x \in \Omega$ and all $k=0, \ldots,K_\tau$, and relying
on \eqref{datoteta} and on \eqref{converg-interp-g}, we conclude
that
\begin{equation}
\label{est-5.1} \sum_{k=1}^{K_\tau} \tau \int_{\Omega} \frac{|\nabla
\wtau{k}|^2}{(1+\wtau{k})^{\varsigma+1}} \dd x \leq C
\sum_{k=1}^{K_\tau} \tau \| \dtau{k}{\chi}\|_{L^2(\Omega)} \|
\Theta(\wtau{k})\|_{L^2(\Omega)} + C'.
\end{equation}
 Now, we argue in the very
same way as in \cite[Proof of Prop.~4.2]{roubicek-SIAM}. Combining
the H\"older and Gagliardo-Nirenberg inequalities
(cf.~\eqref{gn-ineq}) with the previously proved estimate
\eqref{aprio6} and with \eqref{est-5.1}, we see that (cf.
\cite[Formula~(4.35)]{roubicek-SIAM})
\begin{equation}
\label{crucial-Rou} \forall \,1\leq r<\frac{d+2}{d+1} \,  \exists\,
C_r,\,C_r'>0 \, \forall\, \tau>0 \, : \, \sum_{k=1}^{K_\tau}\tau
\|\nabla \wtau{k} \|_{L^r(\Omega)}^r  \leq C_r\sum_{k=1}^{K_\tau}
\tau \| \dtau{k}{\chi} \|_{L^2(\Omega)}
\|\Theta(\wtau{k})\|_{L^2(\Omega)} + C_r',
\end{equation}
where the restriction on the index $r$ in fact derives from the
application of the Gagliardo-Nirenberg inequality \eqref{gn-ineq}. Next,  for a
sufficiently small $\epsilon>0$ such that
$\sigma$  from \eqref{hyp-heat} fulfills $\sigma>(2+\epsilon)d/(d+2)$, there holds
\begin{equation}
\label{gagliarda}
\begin{aligned}
 \|\Theta(\wtau{k})\|_{L^{2+\epsilon}(\Omega)}^{2+\epsilon} \leq C (\|
\wtau{k}\|_{L^{(2+\epsilon)/\sigma}(\Omega)}^{(2+\epsilon)/\sigma}
+1) & \leq C \|\wtau{k}\|_{L^{1}(\Omega)}^{(2+\epsilon)(1-\theta)/\sigma} \|\wtau{k}\|_{W^{1,r}(\Omega)}^{(2+\epsilon)\theta/\sigma} +C'\\ &
\leq C S^{(2+\epsilon)(1-\theta)/\sigma} (S + \|\nabla\wtau{k}
\|_{L^r(\Omega)})^{(2+\epsilon)\theta/\sigma} +C'\\ &  \leq \varrho
\|\nabla \wtau{k} \|_{L^r(\Omega)}^r + C_\varrho \qquad \text{if
$\frac{d(d+2)}{d^2+d+2}<r<\frac{d+2}{d+1}$}
\end{aligned}
\end{equation}
(where we have omitted to indicate the dependence of the constants on $\epsilon$ and $\sigma$).
 The first inequality follows from
\eqref{used-later},
the second one  from the Gagliardo-Nirenberg inequality
\eqref{gn-ineq} with $s=2/\sigma$ and $q=1$: in fact the
constraints
\begin{equation}
\label{epsi}   \frac{\sigma}{2+\epsilon}
>\frac d{d+2}, \ \ \frac{d(d+2)}{d^2+d+2}<r<\frac{d+2}{d+1} \ \
\text{ imply } \  \ \exists\, \theta \in (0,1)\,: \
\frac{\sigma}{2+\epsilon} = \theta\left( \frac1r-\frac1d\right) +
1-\theta,
\end{equation}
in accord with formula \eqref{gn-ineq}. Finally, the
 last inequality in \eqref{gagliarda} is due  to  the
Young inequality, with $C_\varrho$ depending on the constant
$\varrho>0$ to be suitably specified,
under the additional condition that $r<(d+2)/(d+1)$ fulfills
\[
\frac{(2+\epsilon)\theta}{\sigma}<r.
\]
 Combining \eqref{gagliarda}
with \eqref{crucial-Rou}, we immediately obtain
\begin{equation}
\label{stp-intermedio-w}
\sum_{k=1}^{K_\tau}\tau \|\nabla \wtau{k} \|_{L^r(\Omega)}^r  \leq
\frac{C_r}2\sum_{k=1}^{K_\tau} \tau \| \dtau{k}{\chi}
\|_{L^2(\Omega)}^2 + C\varrho \sum_{k=1}^{K_\tau}\tau \|\nabla
\wtau{k} \|_{L^r(\Omega)}^r+ C'.
\end{equation}
Hence, we choose $\varrho>0$ in such a way as to  absorb the second
term on the right-hand side into the left-hand side. Therefore,
 on account of \eqref{aprio4}
$\sup_{\tau} \sum_{k=1}^{K_\tau}\tau \|\nabla \wtau{k}
\|_{L^r(\Omega)} \leq C$, which yields \eqref{aprio7} via
\eqref{aprio6} and the Poincar\'e inequality. Finally, estimate
\eqref{aprio8} ensues from \eqref{aprio7} and \eqref{gagliarda}.

 Observe that, when performing this estimate on the
 semi-implicit equation \eqref{eq-discr-w-irr},
we will obtain on the r.h.s.\ of \eqref{stp-intermedio-w}
the term $\sum_{k=1}^{K_\tau}\tau \|\nabla
\wtau{k-1} \|_{L^r(\Omega)}^r \leq \tau \|\nabla
\w_{0\tau} \|_{L^r(\Omega)}^r +\sum_{k=1}^{K_\tau} \tau \|\nabla
\wtau{k} \|_{L^r(\Omega)}^r
$, and we can estimate $\tau \|\nabla
\w_{0\tau} \|_{L^r(\Omega)}^r$ thanks to \eqref{approx-w0}.

\noindent
\textbf{Fourth a priori estimate for $\mu \in \{0,1\}$, $\rho=0$:} \emph{comparison in \eqref{eq1d}.}
It follows from estimates \eqref{aprio1}, \eqref{aprio1bis},
\eqref{aprio4},
 and from the  regularity result
\eqref{reg-pavel-b},  that
\[
\sup_{\tau} \|\opj{(a(\pwc \chi\tau) +\delta)}{{\partial_t\pwl
{\uu}{\tau}}} \|_{L^2 (0,T;\Ha)}, \ \sup_{\tau} \|\oph{b(\pwc
\chi\tau)}{\pwc {\uu}{\tau}} \|_{L^\infty (0,T;\Ha)} \leq C.
\]
Thus,    for $\rho=0$  estimate \eqref{aprio2} follows from a comparison in
\eqref{eq-discr-u}.

 The same argument carries over to
\eqref{eq-discr-u-irr}  and to \eqref{eq-discr-u-irr-iso}.

\noindent
\textbf{Fifth a priori estimate for $\mu \in \{0,1\}$, $\rho=0$:} \emph{comparison in \eqref{eq0d}.}
In view of estimates  and of \eqref{converg-interp-g},
 a comparison argument  in
\eqref{eq-discr-w} yields estimate \eqref{aprio5}.  The same for \eqref{eq-discr-w-irr}.

\noindent
\textbf{Sixth a priori estimate for $\mu=0$, $\rho \in \R$:}
\emph{we test \eqref{eq2d} by $\opchi(\chi)+\beta(\chi)$ and
integrate in time.} We test \eqref{eq-discr-chi} by $\tau
\left(\opchi(\chitau{k}) + \xitau{k}\right)$. Arguing as for
\eqref{to-be-quoted-later} via convexity inequalities and referring
to notation \eqref{not-h} for the symbol $h_{\tau}^{k-1}$, we get
\[
\begin{aligned}
\Phi(\chitau{k}) &  + \int_{\Omega}\widehat{\beta}(\chitau{k}) \dd x
+ \tau \| \opchi(\chitau{k}) + \xitau{k}\|_{L^2(\Omega)}^2\\ & \leq
\Phi(\chitau{k-1}) + \int_{\Omega}\widehat{\beta}(\chitau{k-1}) \dd
x + \tau \| \gamma(\chitau{k}) +h_{\tau}^{k-1} \|_{L^2(\Omega)} \|
\opchi(\chitau{k}) + \xitau{k}\|_{L^2(\Omega)} \\ & \leq
\Phi(\chitau{k-1}) + \int_{\Omega}\widehat{\beta}(\chitau{k-1}) \dd
x +  \frac1{2}\tau \| \opchi(\chitau{k}) +
\xitau{k}\|_{L^2(\Omega)}^2 + C \tau ( \|h_{\tau}^{k-1}
\|_{L^2(\Omega)}^2+1)
\end{aligned}
\]
where the last inequality follows from $0 \leq \chitau{k} \leq 1$
a.e.\ in $\Omega$, and the fact that $\gamma$ is Lipschitz
continuous on $[0,1]$. Summing up the above inequality for
$k=1,\ldots,K_\tau$ and taking into account a priori estimates
\eqref{aprio1bis} and \eqref{aprio3}, we conclude
\begin{equation}
\label{intermediate-a-prio} \sup_{\tau>0} \| \opchi(\pwc \chi{\tau})
+ \pwc \xi{\tau}\|_{L^2(0,T;L^2(\Omega))} \leq C.
\end{equation}
From this  bound,  exploiting the monotonicity of $\beta$ and
applying \cite[Prop.~2.17]{brezis}, we deduce \eqref{aprio9}. In
view of \eqref{sav-reg}, from the estimate for $\opchi(\pwc
\chi{\tau})$ we deduce \eqref{aprio9bis}.

\noindent
\textbf{Seventh a priori estimate for $\mu=1$, $b{''} \equiv 0$, and  in the isothermal case:}
\emph{we test \eqref{eq2d} by $\partial_t
(\opchi(\chi)+\beta(\chi))$.} Since  $b{''} \equiv 0$, we have that
$b'(\chitau{k-1}) \equiv b$ on $\Omega$.
 We test
\eqref{eq-discr-chi-irr-iso} 
by
$\tau\dtau{k}{\opchi(\chi)
+\beta_\tau(\chi)}=(\opchi(\chitau{k})+\beta_\tau(\chitau{k})
-(\opchi(\chitau{k-1})+\beta_\tau(\chitau{k-1}) $. We observe that
\begin{equation}
\label{est-8.1} I_6:=\int_{\Omega}
(\chitau{k}-\chitau{k-1})(\opchi(\chitau{k})-\opchi(\chitau{k-1}))
\dd x =  \int_{\Omega} (\nabla \chitau{k} -\nabla \chitau{k-1})\cdot
(\mathbf{d}(x, \nabla \chitau{k})  - \mathbf{d}(x, \nabla
\chitau{k-1})) \dd x \geq  0,
\end{equation}
and, if \eqref{pcoercive} holds, we have in addition
\begin{equation}
\label{est-8.2} I_6 \geq c_7\int_{\Omega} |\nabla(\chitau{k}
-\chitau{k-1})|^p \dd x = c_7 \tau \int_{\Omega}\tau^{p-1}
|\nabla\dtau{k}{\chi}|^p \dd x
\end{equation}
Moreover,
 by monotonicity we have
\begin{equation}
\label{est-8.3}
\quad \int_{\Omega} (\chitau{k}-\chitau{k-1})(\beta_\tau(\chitau{k})
-\beta_\tau(\chitau{k-1}) ) \dd x \geq 0, \quad
 \tau\int_{\Omega}
\zetau{k} (\beta_\tau(\chitau{k}) -\beta_\tau(\chitau{k-1}) ) \dd
x\geq 0.
\end{equation}
Furthermore, always by monotonicity, we get
\[
\tau\int_{\Omega}
\zetau{k}(\opchi(\chitau{k})-\opchi(\chitau{k-1})) \dd x \geq 0.
\]
Here, in order to perform a rigorous argument we should approximate
the graph $\alpha$ with a Lipschitz continuous function $\alpha_\e$
as we have done in Lemma \ref{lemma:ex-discr}. However we prefer not
to do it now in order not to overburden the calculations. Combining
\eqref{est-8.1} and \eqref{est-8.3} with the inequalities
\begin{align}\no
& \int_{\Omega}
(\opchi(\chitau{k})+\beta_{\tau}(\chitau{k}))(\tau\dtau{k}{\opchi(\chi)
+\beta_\tau(\chi)})
 \dd x \\
 \no
 &\geq
\frac12\int_{\Omega}|\opchi(\chitau{k})+\beta_{\tau}(\chitau{k})|^2
\dd x -
\frac12\int_{\Omega}|\opchi(\chitau{k-1})+\beta_{\tau}(\chitau{k-1})|^2
\dd x,
\end{align}
and summing over the index $k=1, \ldots, K_\tau$,  we get  (cf. \eqref{not-h} for the notation $h_\tau^{k-1}$,
with $\Theta(\wtau{k-1})$ replaced by $\Thetatau{k-1}$)
\begin{align}
\no
\frac12\int_{\Omega}|\opchi(\chitau{K_\tau})+\beta_{\tau}(\chitau{K_\tau})|^2
\dd x \leq &
\frac12\int_{\Omega}|\opchi(\chitau{0})+\beta_{\tau}(\chitau{0})|^2
\dd x \\
\label{est-8.4}
&+ \ddd{\sum_{k=1}^{K_\tau}\tau  \int_{\Omega}(h_\tau^{k-1}
-\gamma(\chitau{k}))\dtau{k}{\opchi(\chi) +\beta_\tau(\chi)}\dd
x\,.}{$\doteq I_7$}{}
\end{align}
Clearly, the first term on the right-hand side of \eqref{est-8.4} is
bounded thanks to \eqref{additional-chizero}.
 Applying the discrete integration by part formula
\eqref{discr-by-part}, we find
\begin{equation}
\label{est-8.5}
\begin{aligned}
I_7 =(\opchi(\chitau{K_\tau}) +\beta_\tau(\chitau{K_\tau})) &
(h_\tau^{K_\tau-1}+\gamma(\chitau{K_\tau})) -(\opchi(\chitau{0})
+\beta_\tau(\chitau{0}))(h_\tau^{0} +\gamma(\chitau{1})) \\ &
-\sum_{k=2}^{K_\tau}\tau (\opchi(\chitau{k-1})
+\beta_\tau(\chitau{k-1})) (\dtau{k-1}{h} + \dtau{k}{\gamma(\chi)})\,.
\end{aligned}
\end{equation}
Now, by the Lipschitz continuity of $\gamma$ on $[0,1]$,
 we have $ \|
\dtau{k}{\gamma(\chi)}\|_{L^2(\Omega)} \leq C \|\dtau{k}{\chi}
\|_{L^2(\Omega)}$. Furthermore, we find
\[
\begin{aligned}
\| \dtau{k-1}{h}\|_{L^2(\Omega)}  & \leq \|\dtau{k-1}{\Theta^*}
\|_{L^2(\Omega)}   +  C \tau |b|\||\eps( \utau{k-1})|^2-|\eps(
\utau{k-2})|^2 \|_{L^2(\Omega)}
\\ &
 \leq
\|\dtau{k-1}{\Theta^*} \|_{L^2(\Omega)}  + C'\| \eps( \utau{k-1}) +
\eps( \utau{k-2})\|_{L^4(\Omega)}^2 \|\dtau{k-1}{\uu}
\|_{L^4(\Omega)}^2 \doteq j_{\tau}^{k-1},
\end{aligned}
\]
where the second inequality  also follows from the fact that
 $b'$ is constant.
 Collecting
\eqref{est-8.4}--\eqref{est-8.5} and the above inequalities, we thus
infer
\[
\frac12\int_{\Omega}|\opchi(\chitau{K_\tau})+\beta_{\tau}(\chitau{K_\tau})|^2
\dd x \leq C+\sum_{k=1}^{K_\tau} \tau(\|\dtau{k+1}{\chi}
\|_{L^2(\Omega)}+j_{\tau}^{k}) \|
 \opchi(\chitau{k})
+\beta_\tau(\chitau{k}) \|_{L^2(\Omega)},
\]
where we set $\dtau{K_\tau+1}{\chi}=0$. Then, estimate
\eqref{aprio11} ensues via the discrete Gronwall Lemma, taking into
account that
\begin{equation}
\label{L1data} \sum_{k=1}^{K_\tau} \tau (\|\dtau{k}{\chi}
\|_{L^2(\Omega)}+j_{\tau}^{k}) \leq C
\end{equation}
in view of \eqref{converg-interp-Theta-d}, \eqref{aprio1}, and
\eqref{aprio4}.
 Ultimately, \eqref{aprio11bis} follows from \eqref{aprio11} and
 the regularity result~\eqref{sav-reg}.

 \noindent
\textbf{Eighth a priori estimate for $\mu=1$ and in the isothermal case:}
\emph{comparison in \eqref{eq2d}.} From a comparison argument in
\eqref{eq-discr-chi-irr-iso}, we conclude that
$\dtau{k}{\chi}+\zetau{k}$ is estimated in
$L^\infty(0,T;L^2(\Omega))$. Then, \eqref{aprio12} and
\eqref{aprio12bis} follow from the fact that $\int_{\Omega}
\dtau{k}{\chi}\zetau{k}\dd x \geq 0$.
\fin
\paragraph{Proof of Proposition \ref{prop:aprio-2}.}
The a priori bounds \eqref{aprio3}--\eqref{aprio6} follow from the calculations developed
for the
\emph{First a priori estimate}, which also yields \eqref{aprio0} and \eqref{aprio-1}  for a
constant independent of $M>0$.
The \
\textsc{Boccardo\&Gallou\"{e}t}-type \emph{Third estimate} is replaced by the following

\noindent
\textbf{Ninth a priori estimate for $\mu=0$, $\rho\neq 0$:}  \emph{test~\eqref{eq0d} by w.}
We  test \eqref{eq-discr-w-TRUNC} by $\tau\wtau{k}$.
 Summing
over $k=1,\ldots,K_\tau$ and recalling \eqref{def-trunc-ope}--\eqref{def-theta-m}   we obtain
\begin{equation}
\label{est-9.1}
\begin{aligned}
&
\frac12\io |\wtau{K_\tau}|^2\dd x +c_{10}\sum_{k=1}^{K_\tau} \tau\io|\nabla \wtau{k}|^2 \dd x
\\ &  \leq
\int_{\Omega}|\wtau{0}|^2 \dd x + \sum_{k=1}^{K_\tau}\left(
\tau\io\gtau{k}\wtau{k}\dd x  + \tau \io \left(|\dtau{k}\chi|+ |\rho |  |\dive(\dtau{k}\uu)|\right)|\Theta_M(\wtau{k})||\wtau{k}|\dd x\right)
\\
&
\leq \int_{\Omega}|\w_0|^2 \dd x + \nu \sum_{k=1}^{K_\tau} \tau \| \wtau{k} \|_{H^1(\Omega)}^2
+
 C_\nu \sum_{k=1}^{K_\tau} \tau \left(\| \gtau{k} \|_{H^1(\Omega)'}^2 +\|\dtau{k}\chi\|_{L^2(\Omega)}^2
+\rho^2\|\dive(\dtau{k}\uu)\|^2_{L^2(\Omega)}\right),
\end{aligned}
\end{equation}
for a suitably small constant $\nu>0$,  where we have used that
the terms
$\Theta_M (\wtau{k})$ are uniformly bounded (by a constant depending on $M>0$).
Hence, estimate  \eqref{aprio6-rho} follows from
using that $\w_0 \in L^2(\Omega)$, taking into account the previously proved bounds
\eqref{aprio4} and \eqref{aprio0}, and applying the discrete Gronwall Lemma.
Estimate \eqref{aprio5-rho} then ensues from a comparison in \eqref{eq-discr-w-TRUNC},
in view of the previously proved estimates.

Finally, relying on \eqref{aprio6-rho}, we are able to perform the analogue of the \emph{Second a priori estimate}
 on the momentum equation in the case $\rho \neq 0$ as well, as the following calculations show.

\noindent
\textbf{Tenth a priori estimate for $\mu=0$, $\rho\neq 0$:}
\emph{test~\eqref{eq1d} by $- \dive (\eps(\uu_t))$ and integrate
in time.}
  We test
\eqref{eq-discr-u-TRUNC} by $ -{\dive} (\eps(\utau{k}-\utau{k-1})) $.
 Every term can be dealt with  like in the \emph{Second estimate},
in addition
we  need to estimate the term
\[
\left|\tau \rho \int_\Omega    \nabla(\Theta_M(\wtau{k})) \dive(\eps(\dtau{k}\uu) )\dd x\right|
\leq C_\nu\rho\tau \|\nabla(\Theta_M(\wtau{k}))  \|_{L^2(\Omega)}^2+\nu\tau\|\dtau{k}\uu\|_{H^2(\Omega)}^2\,.
\]
 Choosing $\nu$ sufficiently small
 in such a  way as to absorb $\|\dtau{k}\uu\|_{H^2(\Omega)}^2$ into
\eqref{est-3.-1},
and estimating $\|\nabla(\Theta_M(\wtau{k}))  \|_{L^2(\Omega)}^2$  via \eqref{aprio6-rho}
(observe that $\Theta_M$ is Lipschitz continuous),   we re-obtain \eqref{aprio1}--\eqref{aprio1bis},
for  a constant \emph{depending} on $M$.
 Moreover, estimate  \eqref{aprio2} ensues from a comparison in \eqref{eq-discr-u-TRUNC}.

Finally, estimates \eqref{aprio9} and \eqref{aprio9bis} can be obtained by repeating on equation
\eqref{eq-discr-chi-TRUNC} the very same calculations developed for the \emph{Sixth estimate}: again, we get
bounds depending on the truncation parameter $M$.
\fin
\begin{remark}
\upshape
\label{rmk:est-indep-delta}
A close perusal of the proof of Proposition \ref{prop:aprio}, and in particular
 of the calculations performed in the Second and Fourth a priori estimates,
  reveals that in fact estimates \eqref{aprio3}--\eqref{aprio5} hold for constants \emph{independent} of
  $\delta>0$  in both cases $\mu=0$ and $\mu=1$.  This will play a key role in the proof of Theorem \ref{teor5}.
\end{remark}

 We conclude this section by mentioning in advance,
for the reader's convenience,
that the relevant estimates
\begin{compactenum}
\item for the proof of Thm.\ \ref{teor1} are the \emph{First, Second, Third, Fourth, Fifth}, and
 \emph{Sixth a priori estimate};
 \item for the proof of Thm.\ \ref{teor1bis} are the \emph{First,  Fourth, Sixth,  Ninth}, and
 \emph{Tenth a priori estimate};
 \item for the proof of Thm.\ \ref{teor3} are the \emph{First, Second, Third, Fourth}, and
 \emph{Fifth a priori estimate};
 \item  for the proof of Thm.\ \ref{teor4} are the \emph{First, Second,  Fourth, Seventh}, and the
 \emph{Eighth a priori estimate}.
 \end{compactenum}
\section{Proofs of Theorems \ref{teor1},  \ref{teor1bis},  and  \ref{teor2}}
\label{sec:proof2}
\subsection{Proof of Theorem \ref{teor1}}
\label{ss:4.1}
 Preliminarily,
we rewrite
equations \eqref{eq-discr-w}--\eqref{eq-discr-chi} in terms of the
interpolants
$\pwc\w\tau,$  $\upwc\w\tau,$   $\pwc\uu\tau,$ $\upwc\uu\tau,$
$\pwl\uu\tau,$ $\pwwll\uu\tau,$ $\pwc\chi\tau,$
$\upwc\chi\tau,$  $\pwl\chi\tau,$
  and $\pwc \xi\tau$,
 namely
\begin{align}
&
\begin{aligned}
\label{eq-w-interp}
& -\int_0^{\pwc{\mathsf{t}}{\tau}(t)}\io
\pwc\w\tau \varphi_t \dd x \dd s   +
\int_0^{\pwc{\mathsf{t}}{\tau}(t)}\io
\partial_t
\pwl {\chi}{\tau} \Theta(\pwc \w{\tau}) \varphi \dd x \dd s  +\int_0^{\pwc{\mathsf{t}}{\tau}(t)}\io
 K(  \upwc \w{\tau} ) \nabla\pwc
\w{\tau} \nabla \varphi \dd x \dd s
\\ &  \quad  =
\int_0^{\pwc{\mathsf{t}}{\tau}(t)}\io \pwc g{\tau}\varphi \dd x \dd
s   -\io \pwc\w\tau(t) \varphi(t) \dd x  +\io w_0 \varphi(0) \dd x \qquad \text{ for all }\varphi \in \mathscr{F}, \ t \in [0,T], 
\end{aligned}
\\
& \label{eq-u-interp} \partial_t\pwwll {\uu}{\tau}(t)+
\opj{(a(\pwc\chi\tau (t)) +\delta)}{\partial_t \pwl \uu\tau(t)} +
\oph{b(\pwc\chi\tau (t))}{\pwc \uu{\tau}(t)}  
 = \pwc
{\mathbf{f}}{\tau}(t) \quad \aein\, \Omega,  \ \foraa\, t \in (0,T),
\\
& \label{eq-chi-interp}
\begin{aligned}
\partial_t \pwl
{\chi}{\tau}(t) + \opchi(\pwc\chi\tau(t)) +
\pwc\xi\tau(t)+\gamma(\pwc\chi\tau(t)) &= -
b'(\upwc\chi\tau(t))\frac{\eps(\upwc\uu\tau(t))\mathrm{R}_e
\eps(\upwc\uu\tau(t))}2 + \Theta(\pwc\w\tau(t))\\
&\qquad \qquad \aein\, \Omega, \ \foraa\, t \in (0,T),
  \end{aligned}
\end{align}
where for later use in \eqref{eq-w-interp} we have already
integrated by parts in time, and $\mathscr{F}$ is as in
\eqref{eq0d}.
 In what follows  we will take the limit of
\eqref{eq-w-interp}--\eqref{eq-chi-interp} as $\tau \down 0$ by
means of compactness arguments, combined with techniques from
maximal monotone operator 
 theory.

\noindent
\emph{Step $1$: compactness.}
 First of all, we observe that
 due to estimates \eqref{aprio1} and \eqref{aprio2}, there holds
 \begin{equation}
\label{e:stability-estimate}
\begin{array}{ll}
 &  \| \pwl \uu\tau - \pwc \uu\tau\|_{L^\infty (0,T;\boY)} \leq \tau^{1/2}  \| \partial_t \pwl
 {\uu}\tau \|_{L^2 (0,T;\boY)} \leq S \tau^{1/2},
 \\
&  \| \pwwll \uu\tau - \partial_t\pwl {\uu}\tau\|_{L^\infty
(0,T;L^2(\Omega;\R^d))}\leq \tau^{1/2}  \| \partial_t \pwwll
 {\uu}\tau \|_{L^2 (0,T;L^2(\Omega;\R^d))} \leq S \tau^{1/2}.
\end{array}
 \end{equation}
Therefore, \eqref{aprio1}--\eqref{aprio2}, joint with
\eqref{e:stability-estimate} and well-known weak and strong
compactness results (cf.\ \cite{simon}), yield that there exist a
vanishing sequence of time-steps $(\tau_k)$ and
 $\uu$ as in
\eqref{reg-u} such that as $k \to \infty$
\begin{equation}
\label{e:convergences-u}
\begin{array}{lll}
& \pwl \uu{\tau_k} \weaksto \uu  &  \text{ in $H^1(0,T;\boY) \cap
W^{1,\infty}(0,T;\boZ)$,}
\\
&
 \pwl \uu{\tau_k},\, \pwc \uu{\tau_k},\, \upwc \uu{\tau_k} \to \uu &
 \text{ in $L^\infty(0,T;H^{2-\epsilon}(\Omega;\R^d)) $ for all $\epsilon \in (0,1]$,}
\\
& \partial_t \pwwll {\uu}{\tau_k} \weakto \partial_{tt}\uu &
 \text{ in $L^2(0,T;L^{2}(\Omega;\R^d)) $,}
 \\
 &\partial_t \pwl{\uu}{\tau_k}\to \partial_t \uu &\text{ in $L^2(0,T;H^1(\Omega;\R^d)) $.}
 \end{array}
 \end{equation}
A stability estimate analogous to the first of
\eqref{e:stability-estimate}, the a priori bounds
\eqref{aprio3}, \eqref{aprio4} and \eqref{aprio9} and
 the previously
mentioned compactness arguments, imply that there exist $\chi
\in L^\infty (0,T;W^{1,p}(\Omega))\cap H^1 (0,T;L^2(\Omega))$ and
$\xi, \, \lambda \in L^2 (0,T;L^2(\Omega))$ such that, along a not
relabeled subsequence there hold as $k \to \infty$
\begin{equation}
\label{e:converg-chi}
\begin{array}{lll}
&  \pwl \chi{\tau_k} \weaksto \chi &  \text{ in
$L^\infty(0,T;W^{1,p}(\Omega)) \cap H^1 (0,T;L^2(\Omega))$,}
\\
& \pwl \chi{\tau_k}, \, \pwc \chi{\tau_k},\, \upwc \chi{\tau_k}  \to
\chi & \text{ in $L^\infty(0,T;W^{1-\epsilon,p}(\Omega)) $ for all
$\epsilon \in (0,1]$,}
\end{array}
\end{equation}
as well as
\begin{align}
\label{e:converg-xi} & \pwc \xi{\tau_k} \weakto \xi \quad \text{in
$L^2(0,T;L^2(\Omega))$,}
\\
 \label{e:converg-opchi} & \opchi (\pwc \chi{\tau_k}) \weakto
\lambda \quad \text{in $L^2(0,T;L^2(\Omega))$.}
\end{align}
Furthermore, if in addition $\phi$ complies with \eqref{pcoercive}
and \eqref{Lipschitz-x}, then, due to \eqref{aprio9bis}  we also
have the enhanced regularity \eqref{furth-reg-chi}, and the  strong
convergence
\begin{equation}
\label{e:converg-chi-better} \pwl \chi{\tau_k}, \, \pwc \chi{\tau_k},
\, \upwc \chi{\tau_k} \to \chi \qquad  \text{ in
$L^s(0,T;W^{1,p}(\Omega)) $ for all $1\leq s<\infty$.}
\end{equation}

As for $(\pwc \w \tau)_\tau$, estimates
\eqref{aprio6}--\eqref{aprio5} and a
 generalization of the Aubin-Lions theorem  to the case of
 time derivatives as measures (see e.g.\ \cite[Chap.\ 7, Cor.~7.9]{roub-NPDE})
yield that there exists $\w$ as in \eqref{reg-ental} such that, up
to the extraction of a further subsequence, as $k \to \infty$ there
hold
\begin{equation}
\label{up-to-extraction}
\begin{array}{lll}
& \pwc \w {\tau_k}, \, \upwc \w {\tau_k}  \weakto \w & \text{ in
$L^r(0,T;W^{1,r}(\Omega)),$}
\\
& \pwc \w {\tau_k}, \, \upwc \w {\tau_k}   \to \w & \text{ in
$L^r(0,T;W^{1-\epsilon,r}(\Omega))\cap L^s(0,T;L^1(\Omega))$ for all
$\epsilon \in (0,1]$ and  $1\leq s<\infty$.}
\end{array}
\end{equation}
  Furthermore,
 by an infinite-dimensional version of  Helly's selection principle (cf.\ e.g.\ \cite{barbu-precupanu86}) we  have
 $\pwc \w {\tau_k} (t)
\weakto \w(t) $ in $W^{1,r'}(\Omega)^*$ for all $t \in [0,T]$.
Taking into account the a priori bound \eqref{aprio6} of $(\pwc \w
{\tau_k} (t))_{\tau_k}$ in $L^1 (\Omega)$, we then conclude  that
\begin{equation}
\label{further-converg-w} \pwc \w {\tau_k}(t) \weakto \w(t) \qquad
\text{in $\mathrm{M}(\Omega)$ for all $t \in [0,T]$.}
\end{equation}
Clearly,  the second of \eqref{up-to-extraction} implies that $\pwc
\w {\tau_k} \to \w $ a.e.\ in $\Omega \times (0,T)$, hence by the
continuity of $\Theta$ we also have $\Theta(\pwc \w {\tau_k}) \to
\Theta(\w)$ a.e.\ in $\Omega \times (0,T)$. Moreover, estimate
\eqref{aprio8} guarantees that $(\Theta(\pwc \w {\tau_k}))_{\tau_k}$
is uniformly integrable in $L^2(0,T;L^2(\Omega))$. Therefore, thanks
e.g.\  to \cite[Thm. III.3.6]{edwards}, we conclude that
\begin{equation}
\label{e:convej13} \Theta(\pwc{\w}{\tau_k})\to\Theta(\w)\qquad
\text{ in $L^{2} (0,T;L^2(\Omega))$}. \medskip
\end{equation}

\noindent
 \emph{Step $2$: passage to the limit in \eqref{eq-w-interp}--\eqref{eq-chi-interp}.}
It follows from \eqref{e:converg-chi} and \eqref{e:convej13} that
\begin{equation}
\label{conv-useful-1}
\partial_t \pwl {\chi}{\tau_k} \Theta(\pwc \w{\tau_k})  \weakto
\chi_t \Theta(\w) \qquad \text{in $L^1(0,T;L^1(\Omega))$.}
\end{equation}
Moreover, \eqref{up-to-extraction} and \eqref{hyp-K} easily yield
that $  (K(\upwc \w{\tau_k}) \nabla\pwc \w{\tau_k})_{\tau_k}$ is
bounded in $L^r (0,T;L^r(\Omega))$. Since
\begin{equation}
\label{conv-useful-2} K( \upwc \w{\tau_k})
\to K(w) \qquad \text{in $L^s(0,T;L^s(\Omega))$ for every $s\in [1, \infty)$}
\end{equation}
taking into account \eqref{up-to-extraction},
 we can pass to the limit as $k\to \infty$ in the third integral in the first line of
\eqref{eq-w-interp}.
Convergences \eqref{up-to-extraction}--\eqref{further-converg-w},
\eqref{conv-useful-1}--\eqref{conv-useful-2}, as well as
\eqref{converg-interp-g} for $(\pwc g{\tau_k})_{\tau_k}$, allow us to take
the limit of \eqref{eq-w-interp} as $\tau_k \downarrow 0$. Hence we
conclude that $(\w,\chi)$ comply with \eqref{eq0d}.

As for the   passage to the limit in \eqref{eq-u-interp}, we observe
that \eqref{e:converg-chi}, the compact embedding \eqref{hynek}, and
\eqref{data-a} imply that $a(\pwc \chi{\tau_k}) \to a(\chi)$ and
$b(\pwc \chi{\tau_k}) \to b(\chi)$ in $L^\infty (0,T;L^\infty
(\Omega))$. Therefore, by the \eqref{e:convergences-u} we
immediately conclude that
$
 \opj{(a(\pwc\chi{\tau_k} (t))
+\delta)}{\partial_t \pwl \uu{\tau_k}(t)} \to \opj{(a(\chi(t))
+\delta)}{\partial_t \uu (t)}$ and $\oph{b(\pwc\chi{\tau_k}
(t))}{\pwc \uu{\tau_k}(t)} \to \oph{b(\chi (t))}{\uu(t)} $
 in $L^2(\Omega)$ for almost all $t \in (0,T)$. Also relying on the
 third of
\eqref{e:convergences-u} and on \eqref{converg-interp-f} for $(\pwc
{\mathbf{f}}{\tau_k})_k$, we find that $(\uu,\chi)$ fulfill
\eqref{eq1d}.

Finally, combining \eqref{e:converg-chi} with
\eqref{e:converg-xi}--\eqref{e:converg-opchi} and taking into
account the strong-weak closedness of the graphs of the operators
$\partial\beta, \ \opchi: L^2(\Omega) \rightrightarrows
L^2(\Omega)$,  we immediately conclude that
\eqref{e:converg-xi}--\eqref{e:converg-opchi}  hold with
\[
 \xi(x,t) \in \beta(\chi(x,t)) \text{ and }
 \lambda(x,t)=\mathcal{B}(\chi(x,t)) \qquad \foraa\, (x,t) \in
 \Omega \times (0,T).
 \]
 We also observe that \eqref{e:convergences-u},
 \eqref{e:converg-chi}, and \eqref{data-a} yield
\begin{equation}
\label{later-on}
 b'(\upwc\chi{\tau_k}(t))\frac{\eps(\upwc\uu{\tau_k}(t))\mathrm{R}_e
\eps(\upwc\uu{\tau_k}(t))}2 \to
b'(\chi(t))\frac{\eps(\uu(t))\mathrm{R}_e \eps(\uu(t))}2 \qquad
\text{in } L^2(0,T;L^2(\Omega)).
\end{equation}
 Therefore, relying on
\eqref{e:converg-chi} and the Lipschitz continuity
\eqref{data-gamma} of $\gamma$ on bounded intervals, we easily pass
to the limit in \eqref{eq-chi-interp} and infer that $(\chi,\xi)$
fulfill
\medskip \eqref{xi-qualification}--\eqref{eq-with-xi}.

\noindent
 \emph{Step $3$: proof of  total energy equality
 \eqref{total-energy-eq}.} We choose $\varphi\equiv 1$ in
 \eqref{eq0d}, test \eqref{eq1d} by $\uu_t$ and integrate on time,
 and test \eqref{eq-with-xi} by $\chi_t$ and integrate on time, then
 add the resulting relations. Some terms cancel out, and
 to conclude \eqref{total-energy-eq}
 we use the
 chain rule formula \eqref{form-derivative} for the functional $\mathscr{E}(\chi):= \Phi(\chi)+\int_\Omega W(\chi) \dd
 x$,
 as well as the fact that
 \[
 \begin{aligned}
\frac{\dd }{\dd t} \left(  \frac12 \bilh {b(\chi)}{\uu}{\uu} \right)
 & =\frac12 \bilh {b'(\chi)\chi_t}{\uu}{\uu} + \bilh
{b(\chi)}{\uu}{\uu_t}
\\
&= \frac12 \int_\Omega b'(\chi)\chi_t\frac{\eps(\uu)\mathrm{R}_e
\eps(\uu)}2 \dd x +
\pairing{}{H^1(\Omega;\R^d)}{\oph{b(\chi)}{\uu}}{\uu_t} \qquad
\foraa\, t \in (0,T).
\end{aligned}
\medskip
 \]
 \fin
\subsection{Proof of Theorem
\ref{teor1bis}}
\label{sec:4.2new}
\noindent
\emph{Outline.}
First of all,
relying on the a priori estimates of Prop.\ \ref{prop:aprio-2},
 we pass to the limit as $\tau \to 0$ in the time-discretization scheme \eqref{eq-discr-w-TRUNC}--\eqref{eq-discr-chi-TRUNC}: we thus obtain the existence of a
 triple $(\w_M,\uu_M,\chi_M)$ solving the truncated version of system \eqref{eq0d}--\eqref{eq2d}.

Secondly, we perform the passage to the limit as the truncation parameter $M $ tends to
$+\infty$. In this step, we need to obtain for the functions $(\w_{M})_M$  a bound in the spaces
specified in
\eqref{reg-w-rho}, \emph{independent} of the parameter $M$.
In this direction, the key estimate 
consists in testing \eqref{eq0d-trunc} below by (a truncation of) $w_M \in H^1(\Omega)$,
which is now an admissible test function for \eqref{eq0d-trunc}:
it is indeed in view of
performing this test, that we need to keep the two passages to the limit
as $\tau \to 0$ and as $M \to \infty$ distinct.
In order to carry out the calculations related to such an estimate, we need to carefully
tailor to the present truncated setting the \emph{formal} computations outlined in Remark
\ref{rmk:afterThm2}. 

\noindent
\emph{Step $1$: passage to the limit as $\tau \to 0$, for $M>0$ fixed, in \eqref{eq-discr-w-TRUNC}--\eqref{eq-discr-chi-TRUNC}.}
 The argument follows the very same
lines as the one in the proof of Thm.\ \ref{teor1}: it is
even easier, due to the truncations
of the functions $K$ and $\Theta$.
Therefore, we omit the details.
Let us just observe that
passing to the limit as $\tau \to 0$ in
\eqref{eq-discr-w-TRUNC} leads to a solution
\begin{equation}
\label{better-wM}
\w_M \in L^2(0,T;H^1(\Omega)) \cap \mathrm{C}^0 ([0,T];L^2(\Omega)) \cap H^1(0,T;H^1(\Omega)')
\end{equation}
of
 the \emph{truncated} enthalpy equation
\begin{equation}
\label{eq0d-trunc}
\begin{aligned}\!\!\!\!
 \pairing{}{H^1(\Omega)}{\ental_t}{\varphi}    + \io
\chi_t \Theta_M(\ental)\varphi \dd x  +   \rho \io
\hbox{\rm div}(\ub_t) \Theta_M(\ental)\varphi \dd x +\io K_M(\ental)
\nabla \ental\nabla\varphi \dd x    =\pairing{}{H^1(\Omega)}{ g}{ \varphi}
\end{aligned}
\end{equation}
for all $\varphi \in H^1(\Omega).$
Indeed, regularity \eqref{better-wM} follows from estimates \eqref{aprio6-rho}--\eqref{aprio5-rho}, and
in
 turn
it allows for the \emph{stronger} formulation \eqref{eq0d-trunc} (in comparison with \eqref{eq0d}) of the
enthalpy equation.
Therefore,
for every $M>0$ the (Cauchy problem for the) truncated version of
system \eqref{eq0d}--\eqref{eq2d}, consisting of \eqref{eq0d-trunc} and of \eqref{eq1d}--\eqref{eq2d} with
$\Theta$ replaced by $\Theta_M$, admits a solution
$(\w_M, \uu_M, \chi_M)$, with the regularity \eqref{better-wM} for $\w_M$ and
\eqref{reg-u}--\eqref{reg-chi} for $(\uu_M,\chi_M)$,
further
fulfilling \eqref{xi-qualification}--\eqref{eq-with-xi} (with $\Theta_M$ in place of $\Theta$).

\noindent
\emph{Step $2$: passage to the limit as $M \to \infty$.}
Let $(\w_M, \uu_M, \chi_M)_{M}$ be the family of solutions constructed in the previous step.
Since estimates \eqref{aprio3}--\eqref{aprio6} and \eqref{aprio0}--\eqref{aprio-1}
hold with a constant   \emph{independent} of $M$, we conclude by   lower semicontinuity
that
\begin{equation}
\label{estimates-useful-M}
\begin{array}{ll}
\exists\, C>0 \quad \forall\, M>0\, : \ \    \| \w_M\|_{L^\infty(0,T;L^1(\Omega))} & \!\!\!\!\!\! +\| \uu_M\|_{H^1(0,T;H_0^1 (\Omega;\R^d)) \cap W^{1,\infty}(0,T;L^2 (\Omega;\R^d))}
 \\ & \!\!\!\!\!\! +\| \chi_M\|_{L^\infty(0,T;W^{1,p} (\Omega)) \cap H^{1}(0,T;L^2 (\Omega))}
\leq C\,.
\end{array}
\end{equation}

Next,
introduce the truncation operator
\[
\mathcal{T}_M(r)=
\left\{
\begin{array}{ll}
-M & \text{if } r < -M,
\\
r & \text{if } |r| \leq  M,
\\
M & \text{if } r >  M,
\end{array}
\right.
\]
and the sets
\begin{equation}
\label{not-set}
\!\!\!\!\!
\left\{
\begin{array}{lllll}
\mathscr{A}_M :=  \{ (x,t) \in \Omega \times (0,T)\, : \ |\w_M(x,t)|\leq M\},
  & \quad \mathscr{A}_M^t :=  \{ x \in \Omega\, : (x,t) \in \mathscr{A}_M   \}
\\
 \mathscr{O}_M :=   \{ (x,t) \in \Omega \times (0,T)\, : \ |\w_M(x,t)|> M\},  & \quad\mathscr{O}_M^t :=  \{ x \in \Omega\, : (x,t) \in \mathscr{O}_M   \}\,.
\end{array}
\right.
\end{equation}
Hence,
 we test \eqref{eq0d-trunc} by $\mathcal{T}_M(\w_M)$ and integrate on $(0,t)$, $t \in (0,T)$:
observing that
\begin{equation}
\label{key-trick}
\left.
\begin{array}{ll}
&
K_M(\w_M)\nabla w_M \cdot \nabla(\mathcal{T}_M(\w_M))= K(\mathcal{T}_M(\w_M)) |\nabla(\mathcal{T}_M(\w_M))|^2
\\
&
\Theta_M(\w_M)=\Theta (\mathcal{T}_M(\w_M))
\end{array}
\right\}\qquad
 \text{a.e.\ in }\Omega \times (0,T),
\end{equation}
 we thus obtain
\begin{equation}
\label{est-9.1-M}
\begin{aligned}
&\frac12\io |\mathcal{T}_M(\w_M(t))|^2\dd x +\int_0^t\io  K(\mathcal{T}_M(\w_M)) |\nabla(\mathcal{T}_M(\w_M))|^2 \dd x \dd s
\\ & \leq
\frac12\int_{\Omega}|\mathcal{T}_M(\w_M(0))|^2 \dd x + \int_0^t \io |g||\mathcal{T}_M(\w_M)|\dd x  + \int_0^t \io |\ell_M||\Theta (\mathcal{T}_M(\w_M))||\mathcal{T}_M(\w_M)|\dd x\dd s\,,
\end{aligned}
\end{equation}
where we have used the place-holder $\ell_M:= \partial_t \chi_M+\rho\dive(\partial_t\uu_M)$.
Now, arguing in the very same way
as in Rmk.\ \ref{rmk:afterThm2}  we observe that
\begin{equation}
\label{intermediate-M}
\begin{aligned}
&
\int_0^t\io  K(\mathcal{T}_M(\w_M)) |\nabla(\mathcal{T}_M(\w_M))|^2 \dd x
 \geq  c \int_0^t  \left(\|\nabla
\mathcal{T}_M(\w_M)\|_{L^2(\Omega)}^2+\|\mathcal{T}_M(\w_M)\|^{2(q+1)}_{L^{6(q+1)}(\Omega)}\right) \dd s -C
\end{aligned}
\end{equation}
 for positive constants
$c$ and $C$
 \emph{independent} of $M$. Let us now focus on the  on the r.h.s.\ of \eqref{est-9.1-M}:
note that $\|\mathcal{T}_M(\w_M(0))\|_{L^2(\Omega)}^2 \leq \| \w_0\|_{L^2(\Omega)}^2$, whereas the second integral
term can be estimated thanks to
\eqref{heat-source}  on $g$.
Taking into account the growth \eqref{conseq-1}
of $\Theta$
 and \eqref{hyp-K-rho},
the
third integral can be estimated by
 \begin{equation}
 \label{est-9.3.-M}
\begin{aligned}
C \int_0^t\int_\Omega |\ell_M|(|\mathcal{T}_M(\w_M)|^{q+1} {+}1) \dd x \dd s
 \leq
 \varrho \int_0^t \int_\Omega |\mathcal{T}_M(\w_M)|^{2(q+1)} \dd s  + C_\varrho,
\end{aligned}
 \end{equation}
where we have used estimates
\eqref{estimates-useful-M}. Choosing $\varrho>0$ sufficiently small, we can absorb the
integral term on the r.h.s.\ of \eqref{est-9.3.-M} into the
r.h.s.\ of \eqref{intermediate-M}. As in Rmk.\ \ref{rmk:afterThm2}, we thus conclude that
\begin{equation}
\label{estimates-useful-M-2}
\exists\, C>0 \  \forall\, M>0\, : \  \|\mathcal{T}_M(\w_M) \|_{L^2(0,T;H^1(\Omega))
\cap
 L^\infty(0,T; L^2(\Omega))\cap  L^{2(q+1)} (0,T;L^{6(q+1)}(\Omega))  } \leq C\,.
\end{equation}

We now use \eqref{estimates-useful-M-2} in order to infer an analogous estimate for the family $(\w_M)_M$.
 To do so, we preliminarily observe that from the bound for $\|\mathcal{T}_M(\w_M)\|_{L^\infty(0,T; L^2(\Omega))}$
we infer
\begin{equation}
\label{clever-trick}
C \geq \int_\Omega |\mathcal{T}_M(\w_M)(t)|^2 \dd x \geq \int_{\mathscr{O}_M^t} M^2 \dd x = M^2 |\mathscr{O}_M^t|.
\end{equation}
Therefore, upon
 testing \eqref{eq0d-trunc} by $\w_M$, integrating in time, and repeating the same calculations as above (also relying on
\eqref{ellipticity-retained}),
 we end up with
\begin{equation}
\label{est-9.4-M}
\begin{aligned}
&
\frac12\io |\w_M(t)|^2\dd x + c_{10} \int_0^t \|\nabla
\w_M\|_{L^2(\Omega)}^2 \dd s
+ c \int_0^t \|\w_M\|_{L^{6(q+1)}(\mathscr{A}_M^s)}^{2(q+1)} \dd s
\\ &  \leq C + \frac12\io |\w_0|^2\dd x +\varrho_1
\int_0^t \| \w_M \|_{H^1(\Omega)}^2 \dd s
+C_{\varrho_1}  \int_0^t \| g \|_{H^1(\Omega)'}^2 \dd s + I_8
\end{aligned}
\end{equation}
for some $\varrho_1>0$ to be specified later,
and  we estimate
\begin{equation}
\label{est-9.5-M}
\begin{aligned}
I_8  &
= \int_0^t \int_{\mathscr{A}_M^s} |\ell_M| |\Theta_M(\w_M)||\w_M| \dd x \dd s
+\int_0^t \int_{\mathscr{O}_M^s} |\ell_M| |\Theta_M(\w_M)||\w_M| \dd x \dd s
\\
&
\leq  C_{\varrho_2} \left(\int_0^t \| \ell_M\|_{L^2(\Omega)}^2+1 \right)+ \varrho_2
 \int_0^t \|\w_M\|_{L^{2(q+1)}(\mathscr{A}_M^s)}^{2(q+1)} \dd s
\\ & \qquad
+ C_{\varrho_3} \int_0^t \| \ell_M\|_{L^2(\Omega)}^2 \| \Theta_M (\w_M)\|_{L^3(\mathscr{O}_M^s)}^2 \dd s
  + \varrho_3 \int_0^t \| \w_M \|_{H^1 (\mathscr{O}_M^s)}^2 \dd s
\end{aligned}
\end{equation}
where  we have argued along the same lines as in Rmk.\ \ref{rmk:afterThm2}.
Now, observe that  for almost all $t \in (0,T)$
\[
\| \Theta_M (\w_M)\|_{L^3(\mathscr{O}_M^t)}^2 \leq   C ( |\Theta(M)|+|\Theta(-M)|)^{2} |\mathscr{O}_M^t|^{2/3}
\leq C ( |M|^{2/\sigma}+1) |\mathscr{O}_M^t|^{2/3} \leq C' \frac{|M|^{2/\sigma}+1}{M^{4/3}},
\]
where the second inequality is due to the growth \eqref{conseq-1} of $\Theta$, and the last one to
\eqref{clever-trick}. Observe that $2/\sigma - 4/3 <0$ for $d=3$ thanks to \eqref{hyp-heat}, therefore
$(|M|^{2/\sigma}+1)/M^{4/3} \to 0$ as $M \to \infty$. For $d=2$, in \eqref{est-9.5-M}
taking into account the Sobolev embedding
$H^1 (\mathscr{O}_M^s) \subset L^s (\mathscr{O}_M^s)$ for every $s \in [1,\infty)$, we can replace
$\| \Theta_M (\w_M)\|_{L^3(\mathscr{O}_M^s)}^2 $ with $\| \Theta_M (\w_M)\|_{L^{2+\epsilon}(\mathscr{O}_M^s)}^2 $
for any $\epsilon>0$, and tune $\epsilon $ in such a way that the latter term
will  again converge to $0$ as $M \to \infty$.
Therefore the third term on the r.h.s.\ of \eqref{est-9.5-M} is bounded.
It remains to choose $\varrho_2$ in such a way as to absorb the  term
$\int_0^t \|\w_M\|_{L^{2(q+1)}(\mathscr{A}_M^s)}^{2(q+1)} \dd s$ into the l.h.s.\ of
\eqref{est-9.4-M},
and  $\varrho_1$, $\varrho_3$ so that $\varrho_1+\varrho_3$ is  sufficiently small. Also applying the Gronwall Lemma, we conclude
that
\begin{equation}
\label{estimates-useful-M-2-bis}
\exists\, C>0 \  \forall\, M>0\, : \  \|\w_M \|_{L^2(0,T;H^1(\Omega))
\cap
 L^\infty(0,T; L^2(\Omega))} \leq C\,.
\end{equation}
With easy calculations we also find
\begin{equation}
\label{estimates-useful-M-3}
\exists\, C>0 \  \forall\, M>0\, : \| \w_M \|_{L^{2(q+1)} (0,T;L^{6(q+1)}(\Omega))}+
 \|\partial_t \w_M \|_{L^1(0,T;{W}^{2,\expo}(\Omega)')}
\leq C,
\end{equation}
with  $\expo$ as in \eqref{reg-w-rho}, the estimate for
$\partial_t \w_M $
following from a  comparison in \eqref{eq0d-trunc}.

We are now in
the position to obtain the further estimates
\begin{equation}
\label{estimates-useful-M-4}
\begin{array}{ll}
\exists\, C>0 \quad \forall\, M>0\, : \ \  & \| \uu_M\|_{H^1(0,T;H_0^2 (\Omega;\R^d)) \cap W^{1,\infty}(0,T;H_0^1 (\Omega;\R^d)) \cap H^2(0,T;L^2 (\Omega;\R^d)) }
 \\ & \quad  +\| \opchi(\chi_M)\|_{L^2(0,T; L^2 (\Omega))}+\| \xi_M\|_{L^2(0,T; L^2 (\Omega))}
\leq C\,,
\end{array}
\end{equation}
where $\xi_M $ is the selection in $\beta(\chi_M)$ fulfilling \eqref{eq-with-xi}. Indeed,
\eqref{estimates-useful-M-4} can be proved relying on the previously obtained
\eqref{estimates-useful-M} and \eqref{estimates-useful-M-2-bis}--\eqref{estimates-useful-M-3},
 by performing on the
truncated version of
system \eqref{eq0d}--\eqref{eq2d}, the time-continuous analogues of the \emph{Sixth} and \emph{Tenth a priori estimates}
(cf.\ Sec.\ \ref{ss:3.2}). 

Hence, we  can  carry out the passage to the limit argument as $M\to \infty$.
Relying on the bounds \eqref{estimates-useful-M}
and \eqref{estimates-useful-M-4} and on the
compactness tools already exploited in the proof of Thm.\ \ref{teor1}, we
find that there exist $(\w,\uu,\chi)$ and a (not relabeled) subsequence of
$(\w_M,\uu_M,\chi_M)_M$ such that (the time-continuous analogues of) convergences
 \eqref{e:convergences-u}--\eqref{e:converg-chi-better} hold as $M \to \infty$.
Furthermore,
estimates  \eqref{estimates-useful-M-2-bis}--\eqref{estimates-useful-M-3},
the aforementioned generalization \cite[Chap.\ 7, Cor.\ 7.9]{roub-NPDE} of the Aubin-Lions theorem to the case of
$\mathrm{BV}$-functions applies, and
convergences \eqref{up-to-extraction}
improve to
\begin{equation}
\label{better-up-to-extraction}
\begin{array}{lll}
& w_M  \weaksto \w & \text{ in
$L^2(0,T;H^1(\Omega)) \cap L^{2(q+1)} (0,T;L^{6(q+1)}(\Omega))\cap L^\infty (0,T;L^2(\Omega)),$}
\\
& \w_M  \to \w & \text{ in
$L^2(0,T;W^{1-\epsilon,2}(\Omega))\cap L^\sigma(0,T;L^2(\Omega))$ for all
$\epsilon \in (0,1]$ and  $1\leq \sigma<\infty$.}
\end{array}
\end{equation}
with  $\w \in \BV([0,T];{W}^{2,\expo}(\Omega)')$  and $\expo$  as in \eqref{reg-w-rho}.
Relying on \eqref{better-up-to-extraction} and on the Lipschitz continuity of $\Theta$, it is
not difficult to infer that
$\nabla \Theta_M (\w_M) \weakto \nabla \Theta(\w)$ in $L^2(0,T;L^2(\Omega))$ as $k \to \infty$.
Furthermore, combining \eqref{e:convej13} and the last of \eqref{e:convergences-u} we also have that
$\dive(\partial_t \uu_M)\Theta_M(\w_M) \weakto \dive(\uu_t
)\Theta(\w)$ in  $L^1(0,T;L^{3/2}(\Omega))$.
Therefore, we are able to pass to the limit in
\eqref{eq0d-trunc} (with test functions as in the statement of Thm.\ \ref{teor1bis}), and in the corresponding equations for $\uu$ and $\chi$
 in the case $\rho \neq 0$, which concludes the proof.
\fin


\section{Proof of Theorem \ref{teor2}}
\label{CD}

\noindent Let $(\ub_i, \chi_i)$, $i=1,2$, be two solution pairs like
in the statement of Theorem~\ref{teor2} and set $
(\ub,\chi):=(\ub_1-\ub_2, \chi_1-\chi_2,) $. Taking into account
that
$a$ is constant (cf.\ \eqref{constant-a}), hence
$a(\chi_i) \equiv \bar{a}\geq 0$ for $i=1,2$  (cf.\
\eqref{constant-a}), it is immediate to check that
 $(\ub,\chi)$ fulfill
a.e.~in $\, \Omega\times (0,T)$
\begin{align}
 \label{eq1dif}
&\ub_{tt}+\mathcal{E}(b(\chi_1)\ub)+\mathcal{E}((b(\chi_1)-b(\chi_2))\ub_2)+
\mathcal{V}((\bar a+\delta)\ub_t) + \ciro(\Theta_1^*-\Theta_2^*) =\mathbf{f}_1 -\mathbf{f}_2\,,
\\
  \label{eq2dif}
& \chi_t + \mathcal{B}\chi_1-\mathcal{B}\chi_2+\beta(\chi_1) -\beta(\chi_2)  +
\gamma(\chi_1)-\gamma(\chi_2) \\
\no &  \ni-b'(\chi_1)\left( \frac{\varepsilon(\uu_1)
\mathrm{R}_e\varepsilon(\uu_1)}{2} - \frac{\varepsilon(\uu_2)
\mathrm{R}_e\varepsilon(\uu_2)}{2}
\right)-(b'(\chi_1)-b'(\chi_2))\frac{\varepsilon(\uu_2)
\mathrm{R}_e\varepsilon(\uu_2)}{2}+\Theta_1^*-\Theta_2^*\,.
\end{align}
Now, we test~\eqref{eq1dif} by  $\ub_t$ and integrate  in time.
Recalling \eqref{korn}, it is not difficult to
infer
\begin{equation}\label{acca1}
 \mezzo\|\ub_t(t)\hn^2+\delta \itt \| \ub_t \vn^2\dd s \leq\frac12 \|\vv_0^1
 -\vv_0^2\hn^2 +\int_0^t \|\mathbf{f}_1 -\mathbf{f}_2\|_{H^{-1}(\Omega)}\, \|\uu_t\|_{\V}\dd s +
  I_{9}+I_{10}+I_{11},
\end{equation}
where
we have
\begin{align}
\label{acca2} &
\begin{aligned}
I_{9}=-\itt \pairing{}{H^1(\Omega;\R^d)}{   \mathcal{E} (b(\chi_1)
\ub)}{ \ub_t} \dd s  & \leq C \itt \|b(\chi_1) \|_{L^\infty (\Omega)} \: \|
\ub \vn \: \| \ub_t \vn \dd s
\\ & \leq \frac{\delta}{4} \itt \| \ub_t \vn^2 \dd s  + C \itt \| \ub
\vn^2 \dd s \,,
\end{aligned}
\end{align}
whereas, the Lipschitz continuity of  $b$ on bounded intervals (cf.\
\eqref{data-a}) and the H\"{o}lder inequality yield
\begin{align}
 \label{acca3}
\begin{aligned}
 I_{10} &  = \int_0^t \pairing{}{H^1(\Omega;\R^d)}{
 \mathcal{E}((b(\chi_1)-b(\chi_2))\ub_2)}{\uu_t} \dd s  \\
  &  \leq C \itt \| \ub_2 \|_{W^{1,6}(\Omega)} \| \chi\|_{L^3 (\Omega)} \| \ub_t \vn \dd s
\\
 & \leq
\frac{\delta}{4}\itt \| \ub_t \vn^2 \dd s  +C  \| \ub_2
\|_{L^\infty(0,T;H^2(\Omega))}^2 \int_0^t \|\chi\|_{L^3(\Omega)}^2 \dd s
\\
&
 \leq
\frac{\delta}{4}\itt \| \ub_t \vn^2 \dd s  + \nu
 \int_0^t \|\nabla\chi\|_{L^2(\Omega)}^2 \dd s
+C  \int_0^t \|\chi\|_{L^2(\Omega)}^2 \dd s\,,
\end{aligned}
\end{align}
where in the last inequality we have exploited the embeddings
$H^1(\Omega)\Subset L^3(\Omega) \subset L^2(\Omega)$ and \cite[Thm.
16.4, p. 102]{LM},
 with $\nu>0$  a suitable
constant  to be chosen later and the constant  $C$  also depending
on $\| \ub_2 \|_{L^\infty(0,T;H^2(\Omega))}^2 $.
 Moreover, we get
\begin{align}\no
I_{11}=&   \rho\int_0^t\io(\Theta_1^*-\Theta_2^*)\dive(\ub_t)\dd x\dd s  \\
\no
&\leq  \frac{\delta}{4}\int_0^t\|\ub_t\|_{H^1(\Omega)}^2\dd s+\rho^2C_\delta\itt \|\Theta_1^*-\Theta_2^*\|_{L^2(\Omega)}^2\dd s\,.
\end{align}
Noting that
$
 \|\uu(t)\vn^2 \leq 2 \|\uu_0^1 -\uu_0^2 \vn^2 + 2t
\int_0^t \|\uu_t (r)\vn^2 \, {\rm  d}r\,,
$
we obtain from~\eqref{acca1}--\eqref{acca3} that
\begin{equation}\label{uniu}
\begin{aligned}   & \frac12 \|\ub_t(t)\hn^2  +\frac{\delta}{4}\itt \|\ub_t \vn^2
\dd s \\ &
\leq \frac12 \|\vv_0^1
 -\vv_0^2\hn^2 +C\|\mathbf{f}_1 -\mathbf{f}_2\|_{L^2 (0,T; H^{-1}(\Omega))}^2
+\frac{\delta}8 \int_0^t \|\uu_t\vn^2 \dd s   +
 C\|\uu_0^1 -\uu_0^2\vn^2
\\ &
\quad  +
  C \int_0^t \left(\int_0^s
\|\uu_t (r) \vn^2 \, {\rm d}r\right)\, {\rm d}s + \nu
 \int_0^t \|\nabla\chi\|_{L^2(\Omega)}^2 \dd s
+C  \int_0^t \|\chi\|_{L^2(\Omega)}^2 \dd s\\
\no
&\quad+C\rho^2\itt \|\Theta_1^*-\Theta_2^*\|_{L^2(\Omega)}^2\dd s.
\end{aligned}
\end{equation}
Next, we test~\eqref{eq2dif} by $\chi$  integrate the resulting
equation in time. With elementary computations, also taking into
account the Lipschitz continuity of $\gamma$ \eqref{data-gamma},
 the monotonicity of $\beta$ \eqref{databeta},
and the crucial inequality \eqref{essential-inequality},
 we get
\begin{equation}\label{uchi1}
 \frac12 \|\chi(t)\hn^2+c_9 \kappa \itt\|\nabla\chi\|_{L^2 (\Omega)}^2 \dd s \leq \frac12 \|
\chi_0^1 -\chi_0^2 \|_{\Ha}^2 +  C \int_0^t \|\chi\hn^2 \dd s  +
I_{12}+I_{13}\,,
\end{equation}
 \[\begin{aligned} &I_{12}:=-\itt\io b'(\chi_1)\left(
\frac{\varepsilon(\uu_1) \mathrm{R}_e\varepsilon(\uu_1)}{2} -
\frac{\varepsilon(\uu_2)
\mathrm{R}_e\varepsilon(\uu_2)}{2} \right)\chi \dd x \dd s, \\
 &I_{13}:=-\itt\io \left( (b'(\chi_1)-b'(\chi_2))\frac{\varepsilon(\uu_2)
\mathrm{R}_e\varepsilon(\uu_2)}{2}\chi +
(\Theta_1^*-\Theta_2^*)\chi\right) \dd x \dd s\,.
\end{aligned}
\]
Using \eqref{hynek} and the fact that $b'(\chi_1)\in
L^\infty(\Omega)$ we get
\begin{equation}
\label{ei7}
\begin{aligned}
|I_{12}|&\leq \itt\io |b'(\chi_1)|\left(
\frac{|(\varepsilon(\uu_1)-\varepsilon(\ub_2))
\mathrm{R}_e(\varepsilon(\uu_1)+\varepsilon(\ub_2))|}{2}\right)|\chi| \dd x \dd s \\
&\leq C\itt\|\uu\|_{\V}(\|\uu_1\|_{W^{1,6}(\Omega)}+\|\uu_2\|_{W^{1,6}(\Omega)})\|\chi\|_{L^3(\Omega)}\dd s \\
 &\leq  \nu
 \int_0^t \|\nabla\chi\|_{L^2(\Omega)}^2 \dd s
+C  \int_0^t \|\chi\|_{L^2(\Omega)}^2
+C\itt\left(\|\uu_0^1-\uu_0^2\|_{\V}^2 \dd s
+\int_0^s\|\dt\uu\|_{\V}^2\right) \dd s\,,
\end{aligned}
\end{equation}
where the last inequality is obtained arguing as for \eqref{acca3}.
 Again exploiting the Lipschitz
continuity of $b'$ on bounded intervals and the bound for $\uu_2$ in
$L^\infty (0,T;\boY)$, we get
\begin{equation}
\label{ei8}
\begin{aligned}
 |I_{13}|&\leq C\itt\io|\chi|^2|\varepsilon(\uu_2)|^2 \dd x \dd s + \itt
\|\Theta_1^*-\Theta_2^*\|_{L^2(\Omega)}\|\chi\|_{L^2(\Omega)} \dd s
\\
 & \leq
\itt\|\chi\|_{L^6(\Omega)}\|\chi\|_{L^2(\Omega)}\|\uu_2\|^2_{W^{1,6}(\Omega)} \dd s
+\frac12 \itt\|\Theta_1^*-\Theta_2^*\|_{L^2(\Omega)}^2 \dd s+\frac12
\itt\|\chi\|_{L^2(\Omega)}^2 \dd s \,
\\
 &\leq \nu \itt\io \|\nabla\chi\|_{L^2(\Omega)}^2 \dd x \dd s+\frac12
\itt\|\Theta_1^*-\Theta_2^*\|_{L^2(\Omega)}^2 \dd s
+C\itt\|\chi\|_{L^2(\Omega)}^2\dd s\,,
\end{aligned}
\end{equation}
where the last estimate also follows from the continuous
embedding $H^1(\Omega)\subset L^6(\Omega)$ and the Young inequality.
Collecting now \eqref{uchi1}--\eqref{ei8}, we arrive at
\begin{equation}
\begin{aligned} \label{unichi}
 & \frac12 \itt\|\chi(t)\hn^2 \dd s+c_9 \kappa \itt\|\nabla
 \chi\|_{L^2(\Omega)}^2 \dd s
 \\& \leq \frac12  \| \chi_0^1 -\chi_0^2
\|_{L^2(\Omega)}^2+ 2\nu \itt\|\nabla
 \chi\|_{L^2(\Omega)}^2 \dd s \\ & \qquad +
 C \int_0^t \|\chi\|_{L^2(\Omega)}^2 \dd s
+C\itt\left(\|\uu_0^1-\uu_0^2\|_{\V}^2
+\int_0^s\|\dt\uu\|_{\V}^2 \dd r +\int_0^t\|\Theta_1^*-\Theta_2^*\|_{L^2(\Omega)}^2\right)\dd s \,.
\end{aligned}
\end{equation}
Summing up \eqref{uniu} and \eqref{unichi} and choosing $\nu \leq
c_9 \kappa/6$, we conclude
\begin{align}\no
 &\frac12 \|\ub_t(t)\hn^2 +\frac{\delta}{8}\itt \|\ub_t \vn^2 \dd s +\frac12 \|\chi(t)\hn^2
 +\frac{c_9\kappa}{6} \itt\|\nabla \chi\|_{L^2(\Omega)}^2 \dd s \\
 \no
 &\leq C\left(\|
\chi_0^1 -\chi_0^2 \hn^2 +\|\uu_0^1-\uu_0^2\|_{\V}^2+\|\vv_0^1
 -\vv_0^2\hn^2 +\|\mathbf{f}_1 -\mathbf{f}_2\|_{L^2 (0,T; H^{-1}(\Omega))}^2\right.\\
 \no
 & \left.\quad+ \int_0^t \|\chi\|_{L^2(\Omega)}^2\dd s
 +\itt\int_0^s\|\uu_t\|_{\V}^2 \dd r \dd s +\int_0^t\|\Theta_1^*-\Theta_2^*\|_{L^2(\Omega)}^2 \dd s \right)\,.
\end{align}
The application of the standard Gronwall lemma gives immediately the
desired continuous dependence estimate \eqref{cont-depe}.
\fin

\begin{remark}\label{sLaplCD}
\upshape
If we replace
 the
$p$-Laplacian \eqref{defAp} with
 the linear $s$-Laplacian \eqref{bilinear-s_a}   in the equation for $\chi$,
  the continuous
dependence estimate of Theorem \ref{teor2} can be performed without assuming $a$ to be constant (cf. \eqref{constant-a}). Indeed, in this case
we would be able to deal with the additional term
$\int_0^t\pairing{}{H^1(\Omega;\R^d)}{\mathcal{V}((a(\chi_1)-a(\chi_2))\partial_t\uu_2}{\partial_t\uu}\dd s$,
which results from subtracting the equations fulfilled by  solution pairs $(\uu_i,\chi_i)$, $i=1,2$.
It would be possible to
estimate it by means of  the
$H^s(\Omega)$-norm of $\chi=\chi_1-\chi_2$, which would pop in on  the left-hand side of \eqref{uchi1}.
\end{remark}


\section{Proofs of Theorems \ref{teor3} and \ref{teor4}}
\label{sec:proof34}
\subsection{Proof of Theorem \ref{teor3}}
\label{ss:5.1}
\emph{Step $0$: approximate equations.} Equations   \eqref{eq-discr-w-irr}
and \eqref{eq-discr-u-irr} can be rephrased in terms of the
interpolants $\pwc \w \tau, \, \upwc \w \tau, \, \pwc \uu \tau, \,
\upwc \uu \tau,  \,   \pwwll \uu \tau, \,  \pwc \chi \tau, \, \pwl
\chi \tau, $ in a way analogous to
\eqref{eq-w-interp}--\eqref{eq-u-interp}.

Furthermore, 
 taking into account that $\alpha=
\partial I_{(-\infty,0]}$ and that $\beta= \partial I_{[0,+\infty)}$
 we observe that the minimum problem \eqref{eq-discr-chi-irr} yields
 for every $k=1,\ldots, K_\tau$
  that
 \[
 \begin{aligned}
 \frac\tau 2 \int_\Omega\left( \left| \frac{\eta{-}\chitau {k-1}}\tau \right|^2 - \left| \frac{\chitau k{-}\chitau {k-1}}\tau \right|^2  \right) \dd x
 & +\Phi(\eta)-\Phi(\chitau k) +\int_\Omega \left(\widehat{\gamma}(\eta){-}\widehat{\gamma}(\chitau k)
 +h_{\tau}^{k-1}(\eta{-}\chitau k)
  \right)\dd x   \geq 0 \\ & \text{for all }\eta
\in W^{1,p}(\Omega) \text{ with } 0
\leq \eta  \text{ and } \eta \leq \chitau {k-1} \  \aein \,
\Omega\,,
\end{aligned}
 \]
 (recall the short-hand notation \eqref{not-h} for $h_{\tau}^{k-1}$).
  Writing necessary optimality conditions for the above minimum problem, we infer
\begin{equation}
\label{eq-chi-ineq}
\begin{aligned}
&
 \int_\Omega  \Big( \partial_t \pwl
\chi{\tau}(t) (\eta-\pwc \chi\tau (t)) + \mathbf{d}(x,\nabla{\pwc\chi\tau} (t) )
\cdot \nabla (\eta-\pwc \chi\tau (t) ) +   \gamma(\pwc \chi\tau(t) )  (\eta-\pwc
\chi\tau(t) )  + \upwc h\tau(t)  (\eta-\pwc \chi\tau(t) )
 \Big) \, \mathrm{d}x  \geq 0
\\
& \qquad \text{for all } t \in [0,T] \text{ and all }
 \eta \in W^{1,p}(\Omega) \text{ with } 0
\leq \eta  \text{ and } \eta \leq \upwc \chi\tau(t)   \  \aein \,
\Omega,
\end{aligned}
\end{equation}
 where we have used the short-hand notation (cf.\ \eqref{not-h})
\begin{equation}
\label{short-hand} \upwc h\tau (t) :=
{b}'(\upwc\chi\tau(t) )\frac{\varepsilon(\upwc\uu\tau(t) )
\mathrm{R}_e\varepsilon(\upwc\uu\tau(t) )}{2} -\Theta(\upwc\w\tau (t) ).
\end{equation}
Letting $\eta= \nu
\varphi + \pwc\chi\tau(t)$ in
\eqref{eq-chi-ineq}  and dividing the resulting inequality by
$\nu>0$,
  we deduce that
\begin{equation}
\label{eq-chi-ineq-better}
\begin{aligned}
&
  \int_\Omega  \Big( \partial_t \pwl
\chi{\tau}(t)  \varphi + \mathbf{d}(x,\nabla{\pwc\chi\tau}(t) ) \cdot \nabla
\varphi +     \gamma(\pwc \chi\tau (t) )   \varphi   + \upwc h\tau (t) \varphi
 \Big) \, \mathrm{d}x   \geq 0  \ \text{ for all } t \in [0,T]
\\
&   \text{ and all }
 \varphi \in W^{1,p}(\Omega) \text{ s.t.\ there exists $\nu>0$ with  } 0
\leq \nu \varphi + \pwc\chi\tau(t) \leq \upwc\chi\tau(t) \ \aein \,
\Omega.
\end{aligned}
\end{equation}
Choosing $\varphi= - \partial_t \pwl
\chi{\tau}(t) $ (observe that it complies with the constraint above, upon taking $\nu=\tau$), we therefore obtain
\begin{equation}
\label{eq:added-revised-ricky}
\begin{aligned}
 \int_{\Omega}
 \big( |\partial_t \pwl
\chi{\tau}(t)|^2    & +  \mathbf{d}(x,\nabla{\pwc\chi\tau}(t) ) \cdot \nabla(\partial_t \pwl
\chi{\tau}(t)) +   \gamma(\pwl \chi\tau (t) ) \partial_t \pwl
\chi{\tau}(t)
\big)
 \dd x\\
 & \leq - \int_\Omega \upwc h\tau (t)\partial_t \pwl
\chi{\tau}(t) \dd x + \int_\Omega (\gamma(\pwl \chi\tau (t)) - \gamma(\pwc \chi\tau (t)))\partial_t \pwl
\chi{\tau}(t) \dd x\,.
\end{aligned}
\end{equation}
Therefore,
 upon summing \eqref{eq:added-revised-ricky} over the index $k$ we deduce the \emph{discrete
version} of the energy inequality \eqref{energ-ineq} for all  $0
\leq s \leq t \leq T$, viz.
\begin{equation}
\label{energ-ineq-discrete}
\begin{aligned}
 &  \int_{\pwc{\mathsf{t}}{\tau}(s) }^{\pwc{\mathsf{t}}{\tau}(t) }
   \int_{\Omega} |\partial_t \pwl \chi\tau|^2 \dd x \dd r   +
 \Phi(\pwc\chi\tau(\pwc{\mathsf{t}}{\tau}(t))) +  \int_{\Omega} W(\pwc\chi\tau(\pwc{\mathsf{t}}{\tau}(t))) \dd x\\
 & \leq
 \Phi(\pwc\chi\tau(\pwc{\mathsf{t}}{\tau}(s))) + \int_{\Omega}  W(\pwc\chi\tau(\pwc{\mathsf{t}}{\tau}(s)))\dd x
  +  \int_{\pwc{\mathsf{t}}{\tau}(s) }^{\pwc{\mathsf{t}}{\tau}(t) } \int_\Omega \partial_t \pwl \chi\tau \left(- b'(\upwc\chi\tau)
  \frac{\varepsilon(\upwc\uu\tau)\mathrm{R_e}\varepsilon(\upwc\uu\tau)}2
+\Theta(\upwc\w\tau)\right)\dd x \dd r\\
& \qquad \qquad  + C \tau^{1/2} \|\partial_t \pwl \chi\tau \|_{L^2 (0,T; L^2(\Omega))}^2\,,
\end{aligned}
\medskip
\end{equation}
where we have estimated the last term on the right-hand side of \eqref{eq:added-revised-ricky} using that
$\| \gamma(\pwl \chi\tau (t)) - \gamma(\pwc \chi\tau (t))\|_{L^2 (0,T;L^2(\Omega))} \leq C\tau^{1/2} \|\partial_t \pwl \chi\tau \|_{L^2 (0,T; L^2(\Omega))} $,
thanks to the Lipschitz continuity of $\gamma$.

\noindent \emph{Step $1$: compactness.}
In view of the  a priori estimates from Proposition \ref{prop:aprio},
we infer that
 there exist a vanishing subsequence $(\tau_k)_k$ and limit
functions $(\w,\uu,\chi)$ such that convergences
\eqref{e:convergences-u}, \eqref{e:converg-chi},
\eqref{up-to-extraction}--\eqref{e:convej13} hold true as $k \to
\infty$. Observe that \eqref{e:converg-chi}  in particular yields
that $\chi \geq 0$ and $\chi_t \leq 0$ a.e.\ in $\Omega \times
(0,T)$.
 Arguing as in the proof of \cite[Lemma
5.11]{hk1}, we now prove that
\begin{equation}
\label{to-prove} \pwc \chi{\tau_k} \to \chi \qquad \text{in }
L^p(0,T;W^{1,p}(\Omega))\,.
\end{equation}
 Indeed,
\cite[Lemma 5.2]{hk1} gives a sequence $(\pwl \varphi{\tau_k})_k
\subset L^{p}(0,T;W_+^{1,p}(\Omega)) \cap L^\infty (Q)$ of test
functions for \eqref{eq-chi-ineq-better}, fulfilling
\begin{equation}
\label{characteristic-zeta} \pwl \varphi{\tau_k} \to \chi \text{ in
$L^{p}(0,T;W^{1,p}(\Omega))$,} \qquad 0 \leq \pwl \varphi{\tau_k} \leq
\upwc \chi{\tau_k} \ \aein \, \Omega \times (0,T).
\end{equation}
Observe that the first of \eqref{characteristic-zeta} and
convergences \eqref{e:converg-chi} yield in particular
\begin{equation}
\label{strong-convergence-L2} \pwc\chi{\tau_k} \to \chi \qquad
\text{in } L^2(0,T;L^2(\Omega)).
\end{equation}
 We have
\begin{equation}
\label{from-hk}
\begin{aligned}
 & c_7 \int_0^T \int_\Omega |\nabla \pwc \chi{\tau_k} - \nabla \chi|^p
\dd x \dd s
\\
   & \leq \int_0^T \int_\Omega \left(\mathbf{d}(x,\nabla \pwc \chi{\tau_k}) -\mathbf{d}(x,\nabla \chi)  \right)
   \cdot \nabla(\pwc \chi{\tau_k} -\chi) \dd x \dd s
   \\ &
   \leq
   \int_0^T \int_\Omega\mathbf{d}(x,\nabla \pwc \chi{\tau_k})
   \cdot \nabla(\pwc \chi{\tau_k} -\pwl\varphi{\tau_k}) \dd x \dd s
   +
    \int_0^T \int_\Omega\mathbf{d}(x,\nabla \pwc \chi{\tau_k})
   \cdot \nabla(\pwl\varphi{\tau_k}-\chi) \dd x \dd s
\\
& \qquad \qquad\qquad\qquad\qquad\qquad
    -
    \int_0^T \int_\Omega\mathbf{d}(x,\nabla \chi)
   \cdot \nabla(\pwc\chi{\tau_k}-\chi) \dd x \dd s
   \doteq I_{14}+ I_{15}+I_{16}
\end{aligned}
\end{equation}
where the first inequality follows from \eqref{pcoercive} and the
second one from elementary algebraic manipulations. Now, choosing
$\varphi:= \pwl \varphi{\tau_k}- \pwc \chi{\tau_k} $ in
\eqref{eq-chi-ineq-better} (which we are allowed to do thanks to
\eqref{characteristic-zeta}) and integrating in time,  we obtain
\[
\begin{aligned}
I_{14}   = \int_0^T \int_\Omega \left(\partial_t \pwl\chi{\tau_k} +
  \gamma(\pwc\chi{\tau_k})    + \upwc h{\tau_k}\right)(\pwl \varphi{\tau_k}{-}
\pwc \chi{\tau_k}) \dd x \dd s  \leq C \| \pwl \varphi{\tau_k}- \pwc
\chi{\tau_k}\|_{L^2(0,T;L^2(\Omega))} \to 0 \text{ as $k \to
\infty$},
\end{aligned}
\]
due to the bounds \eqref{aprio1bis} and \eqref{aprio4}, and to
\eqref{characteristic-zeta} and \eqref{strong-convergence-L2}. We also have
\[
|I_{15}| \leq \| \mathbf{d}(x,\nabla \pwc \chi{\tau_k})
\|_{L^{p'}(0,T;L^{p'}(\Omega))} \|
\nabla(\pwl\varphi{\tau_k}-\chi)\|_{L^{p}(0,T;L^{p}(\Omega))} \leq C\|
\nabla(\pwl\varphi{\tau_k}-\chi)\|_{L^{p}(0,T;L^{p}(\Omega))}  \to 0,
\]
where the second inequality follows from \eqref{datad2} and
\eqref{aprio4}, and the last passage is due to
\eqref{characteristic-zeta}. Taking into account that $\pwc
\chi{\tau_k} \weakto \chi$ in $L^{p}(0,T;W^{1,p}(\Omega))$ by
\eqref{e:converg-chi-better}, we also prove that $I_{16} \to 0$ as
$k \to \infty$. In this way, from \eqref{from-hk}  we conclude
\eqref{to-prove}. Observe that, \eqref{to-prove} combined with the bound
\eqref{aprio4} then yields \medskip  \eqref{e:converg-chi-better}.

\noindent \emph{Step $2$: passage to the limit.} Arguing in the very
same way as for the proof of Thm.\ \ref{teor1}, it is possible to
prove that $(\w,\uu,\chi)$ solve equations \eqref{eq0d} and
\eqref{eq1d}. It now remains to prove the variational inequality
\eqref{ineq-system2-integrated}, together with \eqref{xi-def}, and the energy
inequality \eqref{energ-ineq}. As for the latter, it is sufficient
to pass to the limit as $k \to \infty$ in
\eqref{energ-ineq-discrete}. For this, we use convergences
\eqref{e:convergences-u}, \eqref{e:converg-chi},
\eqref{up-to-extraction}--\eqref{e:convej13}, \eqref{conv-useful-1},
\eqref{later-on}, as well as \eqref{to-prove}, which in particular
yields
\[
\Phi(\pwc\chi\tau(\pwc{\mathsf{t}}{\tau}(s))) \to \Phi(\chi(s))
\qquad \foraa\, s \in (0,T).
\]
Clearly, the last term on the right-hand side of \eqref{energ-ineq-discrete} tends to zero.
 Since the argument  for \eqref{ineq-system2-integrated}--\eqref{xi-def}
  is perfectly analogous to the one developed in
the proof of \cite[Thm. 4.4]{hk1}, we refer the reader to \cite{hk1}
for all details and here just outline its main steps. Passing to the limit in \eqref{eq-chi-ineq-better}
as $\tau_k \down 0$ with suitable test functions from \cite[Lemma 5.2]{hk1}, we prove that for
almost all $t \in (0,T)$
\[
\begin{aligned}
&
 \int_\Omega  \Big( \chi_t(t) \tilde\varphi +
\mathbf{d}(x,\nabla\chi(t)) \cdot \nabla \tilde\varphi +
\gamma(\chi(t)) \tilde\varphi + b'(\chi(t))\frac{\varepsilon(\ub(t))
\mathrm{R}_e\varepsilon(\ub(t))}{2}\tilde\varphi  -\Theta(w(t))
\tilde\varphi \Big) \, \mathrm{d}x  \geq 0
\\ & \qquad \qquad \qquad \qquad\qquad \qquad \qquad \qquad\qquad \qquad
\text{for all } \tilde\varphi \in W_-^{1,p}(\Omega) \text{ with }
\{\tilde\varphi=0\} \supset \{\chi(t)=0\}.
\end{aligned}
\]
From this, arguing as in the proof of \cite[Thm. 4.4]{hk1} we deduce
that for almost all $t \in (0,T)$
\begin{equation}
\label{ineq-ssy2}
\begin{aligned}
&
 \int_\Omega  \Big( \chi_t(t) \varphi +
\mathbf{d}(x,\nabla\chi(t)) \cdot \nabla \varphi + \gamma(\chi(t))
\varphi
 + b'(\chi(t))\frac{\varepsilon(\ub(t))
\mathrm{R}_e\varepsilon(\ub(t))}{2}\varphi  -\Theta(w(t)) \varphi \Big)
\, \mathrm{d}x
\\ &
  \geq \int_{\{\chi(t)=0\}} \left( \gamma(\chi(t))  +
b'(\chi(t))\frac{\varepsilon(\ub(t))
\mathrm{R}_e\varepsilon(\ub(t))}{2}  -\Theta(w(t))  \right)^+ \varphi
\, \mathrm{d}x
 \ \ \text{for all } \varphi \in
W_-^{1,p}(\Omega).
\end{aligned}
\end{equation}
Therefore, we take
\begin{equation}
\label{xi-charact}
\begin{aligned}
  & \xi(x,t):= - \mathcal{I}_{\{\chi=0\}}(x,t) \left(
\gamma(\chi(x,t)) + b'(\chi(x,t))\frac{\varepsilon(\ub(x,t))
\mathrm{R}_e\varepsilon(\ub(x,t))}{2}    -\Theta(w(x,t)) \right)^+
\\  & \qquad \qquad \qquad \qquad \qquad \qquad \qquad \qquad  \qquad \foraa\, (x,t) \in  \Omega \times (0,T),
\end{aligned}
\end{equation}
$\mathcal{I}_{\{\chi=0\}}$ denoting the characteristic function of
the set $\{\chi=0\}$.
 From \eqref{ineq-ssy2} we deduce that, with this $\xi$ inequality
\eqref{ineq-system2-integrated} holds. Moreover, it is immediate to
check that $\xi$ also complies with \medskip \eqref{xi-def}.

 \noindent \emph{Step $3$: strict positivity
\eqref{strictpos} of the temperature.} Suppose that
\eqref{ipo:strictpos} holds: the discrete strict positivity
\eqref{strict-pos-wk} and convergences \eqref{up-to-extraction}
yield that, in the limit, $\w(x,t) \geq \underline{\w}_0(x) =
\Theta^{-1}(\underline{\teta}_0(x))$ for almost all $(x,t) \in \Omega \times (0,T)$.
Therefore, \eqref{strictpos} ensues.
 \fin

 \subsection{Proof of Theorem \ref{teor4}}
 \noindent
 \label{ss:5.2}
 \emph{Step $1$: compactness.}
For the interpolants $\pwc \uu\tau, \, \upwc \uu\tau, \, \pwl \uu
\tau, \, \pwwll \uu \tau, \, \pwc \chi \tau, \, \pwl \chi \tau,\, \pwc \xi \tau$ of
the solutions $(\utau k, \chitau k, \zetau
k)_{k=1}^{K_\tau}$ of the discrete Problem \ref{probk-irr-iso}, estimates
\eqref{aprio1}--\eqref{aprio4}  and \eqref{aprio11}--\eqref{aprio12bis} hold. Therefore, standard strong and
weak compactness results yield that there exist $(\uu,\chi)$
fulfilling \eqref{reg-u}--\eqref{reg-chi} and a subsequence $\tau_k
\downarrow 0$ such that convergences \eqref{e:convergences-u} and
\eqref{e:converg-chi} hold. Moreover, estimates
\eqref{aprio11}--\eqref{aprio12bis} also imply that
 $\chi$ has the enhanced regularity \eqref{enhanced-regu-chi}, and
 that
\begin{equation} \label{enhanced-converg-chi}
\begin{array}{lll}
& \pwl \chi {\tau_k} \weaksto \chi & \text{ in $L^\infty
(0,T;W^{1+\sigma,p}(\Omega))\cap W^{1,\infty}(0,T;L^2(\Omega))$ for
all $0<\sigma<\frac1p$,}
\\
& \pwl \chi {\tau_k},\, \pwc \chi {\tau_k}, \, \upwc \chi {\tau_k}
\to \chi & \text{ in $L^\infty (0,T;W^{1,p}(\Omega))$.}
\end{array}
 \end{equation}
Furthermore, there exist $\zeta$ and $\xi$ such that, possibly
along a further subsequence,
\begin{align}
\label{converg-omega} & \pwc \zeta{\tau_k} \weaksto \zeta & \text{
in $L^\infty(0,T;L^2(\Omega)),$}
\\
 \label{converg-xi} &
 \beta_{\tau_k} ( \pwc \chi {\tau_k})
\weaksto \xi & \text{ in $L^\infty(0,T;L^2(\Omega))$.} \medskip
\end{align}

\noindent
\emph{Step $2$: passage to the limit.} Relying on  convergences
\eqref{e:convergences-u}, \eqref{e:converg-chi}, and
\eqref{converg-interp-f},
 we take the limit of the discrete momentum equation
\eqref{eq-discr-u-irr}. As for \eqref{eq-discr-chi-irr-iso}, we
observe that, thanks to estimate \eqref{aprio11} and the second of \eqref{enhanced-converg-chi},
there holds
\begin{equation}
\label{converg-opchi} \opchi(\pwc \chi {\tau_k}) \to \opchi(\chi) \text{
weakly$^*$ in $L^\infty(0,T;L^2(\Omega))$ and strongly in $L^\infty(0,T;W^{1,p}(\Omega)^*)$.}
\end{equation}
Therefore, also taking into account \eqref{later-on} we pass to the
limit in \eqref{eq-discr-chi-irr-iso} and conclude
$(\uu,\chi,\xi,\zeta)$ fulfill \eqref{eq-with-zeta}, with
$\Theta(\w)$ replaced by $\Theta^*$. Furthermore, combining
\eqref{e:converg-xi} with \eqref{converg-xi} we have
\[
\limsup_{k \to \infty} \int_0^T \int_\Omega \beta_{\tau_k} ( \pwc
\chi {\tau_k}) \pwc \chi {\tau_k} \dd x \dd s\leq   \int_0^T
\int_\Omega \xi\chi\dd x \dd s.
\]
 Thanks to \cite[Prop. 2.5, p. 27]{brezis} we conclude that $\xi \in \beta(\chi)$ a.e.\ in
$\Omega \times (0,T).$ Finally, testing equation
\eqref{eq-discr-chi-irr-iso} by $\partial_t \pwl \chi{\tau_k}$ and
integrating in time, with calculations analogous to
\eqref{to-be-quoted-later} we find for all $t \in [0,T]$
\[
\begin{aligned}
 &  \limsup_{k \to \infty} \int_{0}^{\pwc{\mathsf{t}}{\tau}(t) }
\int_\Omega \partial_t \pwl \chi{\tau_k} \pwc \zeta{\tau_k} \dd x
\dd s
\\
&  \leq -\liminf_{k \to \infty} \int_{0}^{\pwc{\mathsf{t}}{\tau}(t)
} \int_\Omega | \partial_t \pwl \chi{\tau_k}|^2\dd x \dd s
-\liminf_{k \to \infty} \Phi(\pwc \chi{\tau_k}(t)) -\liminf_{k \to
\infty} \int_\Omega \widehat{\beta}_{\tau_k} (\pwc \chi{\tau_k}(t))
\\ & \qquad \qquad \qquad
-\liminf_{k \to \infty} \int_{0}^{\pwc{\mathsf{t}}{\tau}(t) }
\int_\Omega \gamma(\pwc \chi{\tau_k} ) \partial_t \pwl \chi{\tau_k}
\dd x \dd s -\liminf_{k \to \infty}
\int_{0}^{\pwc{\mathsf{t}}{\tau}(t) }\io \upwc h{\tau_k}
\partial_t \pwl \chi{\tau_k} \dd x \dd s
\\
&  \leq -\int_{0}^{t} \int_\Omega | \chi_t|^2\dd x \dd s -
\Phi(\chi(t)) -\int_\Omega \widehat{\beta}(\chi(t)) -\int_{0}^{t }
\int_\Omega \gamma(\chi)\chi_t \dd x \dd s \\
& \qquad \qquad \qquad- \int_0^t \io h
\partial_t \chi \dd x \dd s
= \int_0^t \int_\Omega \chi_t \zeta \dd x \dd s
\end{aligned}
\]
(here $ \upwc h{\tau}$ is as in \eqref{short-hand}, with $\upwc
{\Theta^*}{ \tau}$ in place of $\Theta(\upwc\w\tau)$, and
 $h:=
b'(\chi) \frac{\eps(\uu)\mathrm{R}_e \eps(\uu)}{2}-\Theta^*$), where the
last inequality is due to \eqref{e:converg-chi} and
\eqref{enhanced-converg-chi}, the Mosco-convergence of  $(\widehat{\beta}_{\tau_k})_{\tau_k}$
 to $\widehat{\beta}$, the Lipschitz continuity of
$\gamma$, and \eqref{later-on}. The
last identity follows from equation \eqref{eq-with-zeta}. The
aforementioned tokens of maximal monotone operator theory allow us
to deduce from the fact that
\[
\limsup_{k \to \infty} \int_0^t
\int_\Omega \partial_t \pwl \chi{\tau_k} \pwc \zeta{\tau_k} \dd x
\dd s\leq \int_0^t \int_\Omega \chi_t \zeta \dd x \dd s,
\]
that
$\zeta \in \alpha(\chi_t)$ a.e.\ in $\Omega \times (0,T)$, which
concludes the proof.
 \fin

 \begin{remark}
 \label{more-general-alpha}
\upshape Indeed, in the proof of Theorem \ref{teor4} the fact
that $\alpha=\partial I_{(-\infty,0]}$ has never been specifically
used, therefore Thm.\ \ref{teor4} extends to a maximal monotone
operator $\alpha$ as in \eqref{1-homog}, observing that, up to
perturbing $\widehat{\alpha} $ with an affine function, it is not
restrictive to suppose that \[ 0 \in \alpha(0) \text{ and }
\widehat{\alpha}(0)=0, \ \text{ whence } \ \widehat{\alpha}(x) \geq
0 \ \text{for all } x \in \R.
\]
 \end{remark}

 \section{Analysis of the degenerating system}
 \label{sec:5}
We now address the passage to the \emph{degenerate} limit $\delta
\downarrow 0$ in the full system \eqref{eq0d}--\eqref{eq2d}. For
technical reasons which will be clarified in Remark
\ref{rmk:explanation} later on,  we focus on the
\emph{irreversible}
 case $\mu=1$, 
 and  neglect the thermal expansion term in the
momentum equation, i.e.\ take  $\rho=0$.
Furthermore, we
confine the discussion to the case, in which, for $\delta >0$ the
coefficients of \emph{both} the elliptic operators in \eqref{eq1d}
are truncated, cf.\ Remark \ref{rmk:added} below.
  In particular, we will  take 
  the functions $a$ and $b$ of the form
\begin{equation}
\label{coefficients-a-b}  a(\chi)= \chi, \ \  b(\chi)= \chi,  \text{ and replace both coefficients by }
\chi+\delta.
\end{equation}
\begin{remark}
\label{trivial-case} \upshape The choice $a(\chi)=1-\chi$ and
$b(\chi)=\chi$ in \eqref{eq1d}  and the truncation of both coefficients
 would lead to the momentum equation
\begin{equation}
\label{trivial-asymo}
\ub_{tt}+\opj{(1-\chi+\delta)}{\ub_t}+\oph{(\chi+\delta)}{\uu}=\mathbf{f} \quad \text{in } H^{-1}(\Omega;\R^d),
\quad \text{a.e. in } (0,T) \end{equation}
 for which the asymptotic
analysis $\delta \down 0$ would be less meaningful in the  case of
an \emph{irreversible} evolution for $\chi$. For, starting from an
initial datum $\chi_0 \in W^{1,p}(\Omega)$ with $\max_{x \in
\overline\Omega} \chi_0(x)<1$, we   would  have $1-\chi(x,t) \geq 1-
\max_{x \in \overline\Omega} \chi_0(x)>0$ for all $(x,t) \in
\overline\Omega \times (0,T]$. Hence,  the limit $\delta \downarrow
0$  would not lead to elliptic degeneracy in \eqref{trivial-asymo}.

 Observe that the ensuing discussion  can be suitably adjusted to the choice $a(\chi)=\chi$, $b(\chi)=1-\chi$, which is meaningful for phase transition models.
\end{remark}
\begin{remark}
\upshape
\label{rmk:added}
It seems to us that \emph{both} the coefficients $a$ and $b$ need to be truncated  when
taking the degenerate limit in the momentum equation. Indeed, on the one hand
the truncation  of $a$ allows us to deal with the \emph{main part} of the elliptic operator
in \eqref{eq1d}. On the other hand, in order to pass to the limit in the quadratic term on the right-hand
side of \eqref{eq2d}, we will also need to truncate $b$, cf.\ \eqref{nice-argum} later on.
\end{remark}
Theorem \ref{teor3} guarantees that for every $\delta>0$  there
exists  a triple $(\w_\delta,\uu_\delta,\chi_\delta)$ as in
\eqref{reg-ental}--\eqref{reg-chi} fulfilling  the enthalpy equation
\eqref{eq0d}  with $\rho=0$,  the momentum equation
\begin{equation}
\label{momentum-eq-delta}
\partial_{tt}\ub_{\delta}-
\dive( (\chi+\delta)\mathrm{R}_v \eps(\partial_t \uu_\delta))-
\dive( (\chi+\delta) \mathrm{R}_e \eps(\uu_\delta)) =\mathbf{f}
\quad \text{in $H^{-1}(\Omega;\R^d)$, a.e.\ in } (0,T),
\end{equation}
(where for later convenience we have  dropped the operator notation
\eqref{operator-notation-quoted}),  as well as
\begin{align}
&
\label{ineq-system2-integrated-irrev}
\partial_t\chi_\delta(x,t) \leq 0 \quad \foraa\, (x,t) \in \Omega \times(0,T),
\\
 \label{ineq-system2-integrated-delta}
 &
\begin{aligned}
\int_\Omega \Big( \left(\partial_t\chi_\delta (t) +
 \xi_\delta(t)+\gamma(\chi_\delta(t))\right)\varphi
+  \mathbf{d}(x,\chi_\delta(t)) \cdot \nabla \varphi \Big)  \dd x
  & \leq \int_\Omega
\left(   -\frac{\varepsilon(\uu_\delta(t))
\mathrm{R}_e\varepsilon(\ub_\delta(t))}{2} + \Theta(w_\delta(t))
\right)  \varphi \dd x
   \\ &   \text{for all }  \varphi \in W_+^{1,p}(\Omega), \quad
   \foraa\, t \in (0,T),
\end{aligned}
\\
\label{ineq-system2-integrated-delta-2} & \text{with }
\xi_\delta(x,t) =- \mathcal{I}_{\{\chi_\delta=0\}}(x,t) \left(
\gamma(\chi_\delta(x,t)) + \frac{\varepsilon(\ub_\delta(x,t))
\mathrm{R}_e\varepsilon(\ub_\delta(x,t))}{2}  -\Theta(w_\delta(x,t))
 \right)^+
\end{align}
for almost all $ (x,t) \in \Omega \times (0,T),$   (changing
sign in \eqref{ineq-system2-integrated} and recalling
\eqref{xi-charact}), and the energy inequality
\eqref{energ-ineq}.
As observed in Section
\ref{glob-irrev}, the family
$(\w_\delta,\uu_\delta,\chi_\delta)_\delta$ then fulfills for all $t
\in (0,T]$ the \emph{energy inequality}
\begin{equation}
\label{total-energy-ineq-delta}
\begin{aligned}
 &
\int_\Omega w_\delta(t) (\mathrm{d} x) +\frac12 \int_\Omega
|\partial_t\uu_\delta (t)|^2\, \mathrm{d} x +\int_0^t \int_\Omega
|\partial_t\chi_\delta|^2 \, \mathrm{d} x\, \mathrm{d} r
+ \int_0^t
\bilj{(\chi_\delta+\delta)}{\partial_t\uu_\delta}{\partial_t\uu_\delta}
\, \mathrm{d} r
 \\  & \quad
+\frac12 \bilh{( \chi_\delta (t)
+\delta)}{\uu_\delta(t)}{\uu_\delta(t)}   +
\Phi(\chi_\delta(t))
+\int_\Omega W(\chi_\delta(t)) \, \mathrm{d} x
\\ &
 \leq \int_\Omega w_0\, \mathrm{d} x +\frac12 \int_\Omega |\vv_0|^2\, \mathrm{d} x
 +\frac12 \bilh{(\chi_0+\delta)}{\uu_0}{\uu_0} +
 \Phi(\chi_0)  +\int_\Omega W(\chi_0) \, \mathrm{d} x
\\ & \qquad \qquad \qquad
+\int_0^t \int_\Omega \mathbf{f} \, \cdot \, \partial_t\uu_\delta \,
\mathrm{d} x \mathrm{d} r + \int_0^t \int_\Omega g \, \mathrm{d} x.
\end{aligned}
\end{equation}

First of all, following \cite{mrz},  in Prop.\ \ref{prop:indepe-delta} below
we deduce from equations
\eqref{eq0d}, \eqref{momentum-eq-delta}, and from
\eqref{total-energy-ineq-delta} some a priori estimates for the
family $(\w_\delta,\uu_\delta,\chi_\delta)_\delta$,
\emph{independent} of $\delta>0$.

Let us  mention in advance  that
 estimate \eqref{aprio-delta-4} for
 $(w_\delta)_\delta$ holds true only for the solutions
 $(\w_\delta,\uu_\delta,\chi_\delta)_\delta$ obtained through the time-discretization procedure
 of Section \ref{ss:3.2}. Such solutions shall be referred to as \emph{approximable}.
Indeed, on the one hand, Remark \ref{rmk:est-indep-delta} ensures that  the
\emph{discrete}
 estimates \eqref{aprio3}--\eqref{aprio5} are valid with constants independent on $\delta$:
 hence they are inherited by the approximable solutions $(w_\delta)_\delta$, yielding estimate \eqref{aprio-delta-4} below.
 On the other hand, the calculations developed for the Fourth a priori estimate in Sec.\ \ref{ss:3.2}
 suggest that, in order to prove \eqref{aprio-delta-4} for \emph{all} weak solutions
 $(\w_\delta)_\delta$
  to \eqref{eq0d},
 it would be necessary to test \eqref{eq0d} by $\varphi=\Pi(\w_\delta)$
 with $\Pi$ as in \eqref{speciaL-test-func}. This is not
 an admissible choice  due to the poor regularity of $\w_\delta$.
 Since we do not dispose
 of a uniqueness result for the \emph{irreversible full} system, we
   cannot conclude  \eqref{aprio-delta-4} for \emph{all} weak  solutions
   (in the sense of Def.\ \ref{def-weak-sol})   $(w_\delta)_\delta$, and therefore we will restrict to \emph{approximable} solutions.

 As it will be clear from the proof of Prop.\ \ref{prop:indepe-delta},
 estimates \eqref{aprio-delta-1}--\eqref{aprio-delta-3} instead hold for \emph{all} weak solutions
 $(\uu_\delta,\chi_\delta)$.
\begin{proposition}
\label{prop:indepe-delta}
 Assume Hypotheses (I),  (II),  and   (IV)  with
$\widehat{\beta}=I_{[0,+\infty)}$,  conditions
\eqref{bulk-force}--\eqref{datochi}  on the data $\mathbf{f}, \, g,
\teta_0,  \, \uu_0, \, \vv_0, \, \chi_0$,  and suppose that
$a,\, b$ are given by \eqref{coefficients-a-b}.  
 Then, there exists a constant
$\overline{S}>0$ such that for all $\delta>0$ and for all
$(\w_\delta,\uu_\delta,\chi_\delta)_\delta$  \emph{(approximable)
weak} solutions
  to the irreversible full system,
the following estimates hold
\begin{align}
& \label{aprio-delta-1} \|\w_\delta \|_{L^\infty(0,T;L^1(\Omega))} +
\| \partial_t \uu_\delta \|_{L^\infty(0,T;L^2(\Omega;\R^d))}+
\|
\chi_\delta \|_{L^\infty(0,T;W^{1,p}(\Omega)) \cap H^1(0,T;L^2(\Omega))}  \\
\no
&\qquad\qquad+ \| W(\chi_\delta)\|_{L^\infty(0,T;L^1(\Omega))} \leq \overline{S},
\\
& \label{aprio-delta-2} \| \sqrt{\chi+\delta}\,\mathrm{R}_v\,
\eps(\partial_t \uu_\delta) \|_{L^2(0,T;L^2(\Omega;\R^{d \times
d}))} + \| \sqrt{\chi+\delta}\,\mathrm{R}_e\,\eps(\uu_\delta)
\|_{L^\infty(0,T;L^2(\Omega;\R^{d \times d}))} \leq \overline{S},
\\
& \label{aprio-delta-3} \|\partial_{tt} \ub_{\delta}\|_{L^2
(0,T;H^{-1}(\Omega;\R^d))} \leq \overline{S},
\\
& \label{aprio-delta-4} \| \w_\delta
\|_{L^r(0,T;W^{1,r}(\Omega))\cap
\mathrm{BV}([0,T];W^{1,r'}(\Omega)^*)} \leq \overline{S}.
\end{align}
\end{proposition}
\begin{proof}
Estimates \eqref{aprio-delta-1}--\eqref{aprio-delta-2} are
straightforward consequences of the energy inequality \eqref{total-energy-ineq-delta},
taking into account that
$\int_\Omega W(\chi_\delta(t))\dd x \geq -C$
for  a constant independent of $t \in [0,T]$,
 estimating
\[
\int_0^t \int_\Omega \mathbf{f} \, \cdot \, \partial_t\uu_\delta \,
\mathrm{d} x \mathrm{d} r \leq
 \frac12 \int_0^t \int_\Omega |\mathbf{f}|^2 \dd x  + \frac12 \int_0^t \int_\Omega|\partial_t\uu_\delta|^2
\dd x \mathrm{d} r,
\]
 and applying the Gronwall Lemma. Hence, \eqref{aprio-delta-3}
 follows from a comparison in \eqref{momentum-eq-delta}, in view of
 \eqref{reg-pavel-a}. Finally, \eqref{aprio-delta-4} can be proved
 by observing that the discrete estimates
 \eqref{aprio7}--\eqref{aprio5} are in fact \emph{independent} of the parameter
 $\delta>0$, hence they carry over to  the \emph{approximable} solutions
 $(\w_\delta)_\delta$.
\end{proof}

As pointed out in \cite{mrz} (see also \cite{bmr}), estimates
\eqref{aprio-delta-2} suggest that for the analysis $\delta \down 0$
it is meaningful to work with  the quantities
\begin{equation}
\label{mu-delta_eta-delta}
\mmu_\delta:=\sqrt{\chi_\delta+\delta}\, \eps(\partial_t \uu_\delta),
\qquad \eeta_\delta:= \sqrt{\chi_\delta+\delta}\,\eps(\uu_\delta)
\end{equation}
in terms of which \eqref{momentum-eq-delta} rewrites as
\begin{equation}
\label{momentum-eq-delta-2}
\partial_{tt}\ub_{\delta}-
\dive(  \sqrt{\chi_\delta+\delta} \, \mathrm{R}_v\,\mmu_\delta)- \dive(
 \sqrt{\chi_\delta+\delta} \,\mathrm{R}_e\,\eeta_\delta ) =\mathbf{f} \quad
\text{in $H^{-1}(\Omega;\R^d)$, a.e.\ in } (0,T).
\end{equation}
For later purposes, we also observe that, in the setting of \eqref{coefficients-a-b}
 and with notation \eqref{mu-delta_eta-delta},
 the total energy inequality \eqref{total-energy-ineq} for the triple $(\w_\delta,\uu_\delta,\chi_\delta)$
can be reformulated as
\begin{equation}
\label{pass-lim-delta-tot}
\begin{aligned}
 &
\int_\Omega w_\delta(t)(\mathrm{d} x) +\frac12 \int_\Omega |\partial_t\uu_\delta (t)|^2\,
\mathrm{d} x +\int_s^t \int_\Omega |\partial_t\chi_\delta|^2 \, \mathrm{d} x\,
\mathrm{d} r + \frac12   \int_s^t \int_\Omega
\mmu_\delta(r)\, \mathrm{R}_v \,\mmu_\delta(r)  \, \mathrm{d} x \, \mathrm{d} r  \\  & \quad
 +\frac12
 \int_\Omega \eeta_\delta(t) \,\mathrm{R}_e \,\eeta_\delta(t) \, \mathrm{d} x  +   \Phi(\chi_\delta(t))  +\int_\Omega
W(\chi_\delta(t)) \, \mathrm{d} x
\\ &
 \leq \int_\Omega w_\delta(s)(\mathrm{d} x) +\frac12 \int_\Omega |\partial_t\uu_\delta
(s)|^2\, \mathrm{d} x + \frac12  \int_\Omega
 \eeta_\delta(s) \,\mathrm{R}_e \,\eeta_\delta(s) \, \mathrm{d} x  +
  \Phi(\chi_\delta(s))
\\ & \qquad \qquad +\int_\Omega W(\chi_\delta(s)) \, \mathrm{d} x
+\int_s^t \int_\Omega \mathbf{f} \, \cdot \, \partial_t\uu_\delta \, \mathrm{d} x
\mathrm{d} r + \int_s^t \int_\Omega g \, \mathrm{d} x\mathrm{d} r \,.
\end{aligned}
\end{equation}

The following result shows that the limit $\delta \down 0$ preserves
the structure \eqref{momentum-eq-delta-2} of the momentum equation,
as well as    the enthalpy equation \eqref{eq0d}. The  weak formulation
\eqref{ineq-system2-integrated}--\eqref{energ-ineq} of the equation
for  $\chi$ is generalized by \eqref{variational-limit}--\eqref{energ-ineq-lim},
cf.\ Rmk.\ \ref{rmk:comparison-with-hk}.
\begin{maintheorem}
\label{teor5}
 Assume Hypotheses (I),  (II),     (IV)  with
$\widehat{\beta}=I_{[0,+\infty)}$,  and  (V) with  $\phi$ fulfilling \eqref{pcoercive}. Assume conditions
\eqref{bulk-force}--\eqref{datochi}  on the data $\mathbf{f}, \, g,
\teta_0,  \, \uu_0, \, \vv_0, \, \chi_0$,  and suppose that
$a,\, b$ are given by \eqref{coefficients-a-b}.
Then, there exist $\w$ as in
\eqref{reg-ental}, and
\begin{align}
& \label{reg-u-gene} \uu \in W^{1,\infty}(0,T;L^2(\Omega;\R^d))\cap
H^2(0,T;H^{-1}(\Omega;\R^d)), \ \mmu \in
L^2(0,T;L^2(\Omega;\R^{d\times d})), \\
\no
& \eeta \in
L^\infty(0,T;L^2(\Omega;\R^{d\times d})),
\\
& \label{reg-chi-gener}
 \chi \in L^\infty (0,T;W^{1,p}(\Omega))  \cap H^1 (0,T;L^2(\Omega)),
\quad \chi(x,t) \geq 0, \ \  \chi_t (x,t) \leq 0 \qquad \foraa\,
(x,t) \in  \Omega \times (0,T),
\end{align}
such that
\begin{equation}
\label{interesting}
 \mmu  = \sqrt{\chi} \,\eps(\uu_t), \ \eeta  = \sqrt{\chi} \, \eps(\uu)
 \ \text{a.e. in any open set } A \subset \Omega \times (0,T)
 \text{ s.t. } \chi >0 \ \aein \, A,
\end{equation}
 fulfilling the weak enthalpy equation \eqref{eq0d}  with $\rho=0$,  the
\emph{weak} momentum equation
\begin{equation}
\label{momentum-eq-delta-limit}
\partial_{tt}\ub-
\dive( \sqrt{\chi}\,\mathrm{R}_v\,\mmu)- \dive(
\sqrt{\chi}\mathrm{R}_e\,\,\eeta )) =\mathbf{f} \quad \text{in
$H^{-1}(\Omega;\R^d)$, a.e.\ in } (0,T),
\end{equation}
as well as 
\begin{equation}
\label{variational-limit}
\begin{aligned}
&  \int_0^T \int_\Omega \Big(  \left(\partial_t\chi+
\gamma(\chi)\right)\varphi + \mathbf{d}(x,\nabla\chi) \cdot \nabla
\varphi \Big)  \dd x  \dd t  \leq \int_0^T \int_\Omega \left(
 -\frac1{2 \chi} \eeta \,\mathrm{R}_e\, \eeta + \Theta(w)  \right) \varphi \dd x
\dd t
 \\ &
 \qquad   \quad   \text{for all }  \varphi \in   L^p
 (0,T;W_+^{1,p}(\Omega))
\cap L^\infty (Q) \text{ with } \mathrm{supp}(\varphi) \subset \{
\chi>0\},
\end{aligned}
\end{equation}
and the \emph{total energy inequality} for  almost all $t\in (0,T]$
\begin{align}
&
\label{energ-ineq-lim}
\begin{aligned}
 &
\mathcal{H}(t)
 +\int_0^t \int_\Omega |\chi_t|^2 \, \mathrm{d} x\,
\mathrm{d} r + \frac12 \int_0^t  \int_\Omega
\mmu(r) \, \mathrm{R}_v\, \mmu(r) \, \mathrm{d} x  \, \mathrm{d} r
\\
  &  \leq  \int_\Omega w_0\, \mathrm{d} x +\frac12 \int_\Omega |\vv_0|^2\, \mathrm{d} x
  +  \frac12 \bilh{\chi_0}{\uu_0}{\uu_0}  +
  \Phi(\chi_0)  +\int_\Omega W(\chi_0) \, \mathrm{d} x
 \\
 & \qquad \qquad \qquad \qquad \qquad
+\int_0^t \int_\Omega \mathbf{f} \, \cdot \, \uu_t \, \mathrm{d} x
\mathrm{d} r + \int_0^t \int_\Omega g \, \mathrm{d} x \mathrm{d} r
\end{aligned}
\\
&
\label{inequality-H-sec7}
\begin{aligned}
\text{with } \mathcal{H}(t)   &  \geq
\int_\Omega w(t)(\mathrm{d} x)   +\frac12 \int_\Omega |\partial_t\uu  (t)|^2\, \mathrm{d}x  
  +  \Phi(\chi(t))
  +\int_\Omega
W(\chi(t)) \, \mathrm{d} x + \mathcal{J}(t),
\\
&
\qquad \qquad \qquad \qquad \qquad
  \text{where }  \mathcal{J}(t):= \frac12 \liminf_{\delta_k \down 0} \int_\Omega
\eeta_{\delta_k}(t) \,\mathrm{R}_e \,\eeta_{\delta_k}(t) \dd x,
\end{aligned}
\end{align}
(with $(\eeta_{\delta_k})$ a suitable subsequence of
$(\eeta_\delta)$ from \eqref{mu-delta_eta-delta}),
and for all $0 \leq t_1\leq t_2 \leq T$ there holds
\begin{equation}
\label{integral-inequality-forH}
 \!\!\! \int_{t_1}^{t_2}  \mathcal{H}(r)    \dd r
  \geq  \int_{t_1}^{t_2}  \left(
\int_\Omega w(r)(\mathrm{d} x)   +  \Phi(\chi(r))  +
\int_\Omega\left( \frac12 |\partial_t\uu (r)|^2 +W(\chi(r) +\frac12
\eeta(r) \,\mathrm{R}_e \,\eeta(r) \right) \dd x    \right) \dd r\,.
\end{equation}
\end{maintheorem}
\begin{remark}
\upshape \label{rmk:comparison-with-hk} Let us briefly compare the
concept of weak solution (to the  \emph{degenerating} irreversible
full system \eqref{eq0d}--\eqref{eq2d}) arising from
\eqref{interesting}--\eqref{energ-ineq-lim}, with the notion of weak
solution (to the  \emph{non-degenerating} irreversible full system
\eqref{eq0d}--\eqref{eq2d}) given in Definition \ref{def-weak-sol},
in the case  in which  
$a(\chi)=b(\chi)=\chi$.  Suppose that the functions
$(\uu,\chi)$ in \eqref{reg-u-gene} and \eqref{reg-chi-gener}
 have
 further regularity properties
\eqref{reg-u}--\eqref{reg-chi}, and  that $\chi>0$ a.e.\ in $\Omega\times (0,T)$.
 Then, \eqref{interesting} holds a.e.\ in $\Omega\times (0,T)$, hence it is immediate to realize
 that \eqref{variational-limit} coincides with \eqref{ineq-system2-integrated}. Furthermore,
 subtracting from \eqref{energ-ineq-lim} the weak enthalpy equation
 \eqref{eq0d} tested by $1$, we obtain
 a generalized form of
  the energy inequality \eqref{energ-ineq} for almost all  $t \in (0,T]$
 and for $s=0$. 
\end{remark}

\noindent
\begin{proof}
It follows from estimates
\eqref{aprio-delta-1}--\eqref{aprio-delta-4} and the same
compactness arguments as in the proofs of Thms.\ \ref{teor1} and
\ref{teor3} that there exist a vanishing sequence $\delta_k \down 0$
and  functions $\w$ as in \eqref{reg-ental} and
$(\uu,\chi,\mmu,\eeta)$ as in
\eqref{reg-u-gene}--\eqref{reg-chi-gener}
 such that as $k \to \infty$
 \begin{align}
 &
\label{convergences-degen-w}
 \w_{\delta_k} \to w  \text{ in $L^r(0,T;W^{1-\epsilon,r}(\Omega)) \cap L^s (0,T;L^1(\Omega))$
  for all $\epsilon \in (0,1]$
 and all $1\leq s<\infty$,}
 \\
& \label{convergences-degen-u-1}
 \uu_{\delta_k} \weaksto \uu  \text{ in $W^{1,\infty}(0,T;L^2(\Omega;\R^d)) \cap H^{2}(0,T;H^{-1}(\Omega;\R^d)) $,}
 \\
 &
 \label{convergences-degen-mu}
 \mmu_{\delta_k} \weakto \mmu  \text{ in
 $L^2(0,T;L^2(\Omega;\R^{d\times d}))$,}
 \\
 &
 \label{convergences-degen-eta}
 \eeta_{\delta_k} \weaksto \eeta  \text{ in
 $L^\infty(0,T;L^2(\Omega;\R^{d\times d}))$,}
 \\
& \label{convergences-degen-chi-1} \chi_{\delta_k} \weaksto \chi
\text{ in
 $  L^\infty(0,T;W^{1,p}(\Omega)) \cap H^1 (0,T;L^2(\Omega))$,}
 \\
&
 \label{convergences-degen-chi-2}
\chi_{\delta_k} \to \chi  \text{ in
 $\mathrm{C}^0([0,T];\mathrm{C}^0(\overline\Omega))$, }
 \end{align}
the latter convergence due to the compactness results in
\cite{simon} and the compact embedding $W^{1,p}(\Omega)
\Subset \mathrm{C}^0(\overline\Omega)$.  Observe that
\eqref{convergences-degen-u-1} and \eqref{convergences-degen-chi-1}
respectively yield
\begin{equation}
\label{eq:-added-conve-sec7}
\partial_t  \uu_{\delta_k} \to \uu_t \text{ in $  \mathrm{C}^0_{\mathrm{weak}}([0,T];L^2 (\Omega;\R^d))$}, \qquad \chi_{\delta_k} \to \chi
\text{ in   $  \mathrm{C}^0_{\mathrm{weak}}([0,T];W^{1,p}
(\Omega))$.}
\end{equation}
From \eqref{convergences-degen-w}, exploiting \eqref{conseq-1} we
deduce that
\begin{equation}
\label{conv-teta-wk} \Theta(\w_{\delta_k}) \to \Theta(\w) \quad
\text{ in $L^2(0,T;L^2(\Omega))$.}
\end{equation}
 Thus, we are in the position of
passing to the limit as $\delta_k \down 0$ in \eqref{eq0d} for the
functions $(\w_{\delta_k},\chi_{\delta_k})$, and conclude
\eqref{eq0d} for $(\w,\chi)$.

Exploiting \eqref{convergences-degen-chi-2} and the fact that
$t\mapsto \chi_\delta (x,t)$ is nonincreasing for all $x \in
\overline{\Omega},$ with the very same argument as in the proof of
\cite[Prop. 4.3]{mrz} it is possible to prove that $\mmu$ and
$\eeta$ have the form \eqref{interesting}. In order to do that we can use the boundedness
of $\e(\uu_{\delta_k})$ and of $\e(\partial_t\uu_{\delta_k})$ in $L^2(K; \RR^{d\times d})$
for any compact cylinder $K$ of the form $ K_0 \times [0,t]$ on which $\chi>0$. {Notice that on these cylinders $\chi\geq \bar\delta>0$ for some $\bar\delta>0$. Hence, }
 exploiting convergence \eqref{convergences-degen-chi-2}, we infer that
 there exists  $\delta_0>0$ such that,
for any $0<\delta_k\leq \delta_0$, we have $\chi_{\delta_k}(x,t)+\delta_k\geq \bar\delta$ for all $x\in K_0$. Thus also
$\chi_{\delta_k}(x,s)+\delta_k\geq \bar\delta$ for all $(x,s)\in K= K_0 \times [0,t]$ because
$t\mapsto \chi_\delta (x,t)$ is nonincreasing for all $x \in
\overline{\Omega}$. Then we can identify at the limit $\mmu$ and $\eeta$ and cover $A$
in \eqref{interesting} by cylinders of the form $K$ above.
Hence, relying on
\eqref{convergences-degen-u-1}--\eqref{convergences-degen-chi-2} it
is immediate to pass to the limit in \eqref{momentum-eq-delta-2} and
conclude \eqref{momentum-eq-delta-limit}.

Next, we prove that
\begin{equation}
\label{to-prove-delta}  \chi_{\delta_k} \to \chi \qquad \text{in }
L^p(0,T;W^{1,p}(\Omega))\,.
\end{equation}
For this, we repeat the arguments from Step $2$ in the proof of  Thm.\ \ref{teor3}, based on  \cite[Lemma
5.11]{hk1}. Namely, we apply
\cite[Lemma 5.2]{hk1},
which gives
 a sequence $(\varphi_{\delta_k})_k
\subset L^{p}(0,T;W_+^{1,p}(\Omega)) \cap L^\infty (Q)$, such that
$  \varphi_{\delta_k} \to \chi $   in
$L^{p}(0,T;W^{1,p}(\Omega))$ and $ 0 \leq  \varphi_{\delta_k} \leq \chi_{\delta_k}  $ a.e.\ in  $\Omega \times (0,T).$
Relying on  assumption \eqref{pcoercive}, with the same calculations as in \eqref{from-hk} we then have
\begin{equation}
\label{from-hk-delta}
\begin{aligned}
 & c_7 \int_0^T \int_\Omega |\nabla \chi_{\delta_k} - \nabla \chi|^p
\dd x \dd s
   \\ &
   \leq
   \int_0^T \int_\Omega\mathbf{d}(x,\nabla  \chi_{\delta_k})
   \cdot \nabla( \chi_{\delta_k} -\varphi_{\delta_k}) \dd x \dd s
   +
    \int_0^T \int_\Omega\mathbf{d}(x,\nabla\chi_{\delta_k})
   \cdot \nabla(\varphi_{\delta_k}-\chi) \dd x \dd s
\\
& \qquad \qquad\qquad\qquad\qquad\qquad
    -
    \int_0^T \int_\Omega\mathbf{d}(x,\nabla \chi)
   \cdot \nabla(\chi_{\delta_k}-\chi) \dd x \dd s
   \doteq I_{17}+ I_{18}+I_{19}.
\end{aligned}
\end{equation}
 Now, choosing
$\widetilde{\varphi}_{\delta_k}:=  \varphi_{\delta_k}-  \chi_{\delta_k} $
as a test function for \eqref{ineq-system2-integrated-delta}
  and integrating in time,  we obtain
\[
\begin{aligned}
I_{17}   = \int_0^T \int_\Omega \left(\partial_t \chi_{\delta_k} + \xi_{\delta_k}+
  \gamma(\chi_{\delta_k}) +   \frac{\varepsilon(\uu_{\delta_k})
\mathrm{R}_e\varepsilon(\ub_{\delta_k})}{2} - \Theta(w_{\delta_k})\right)( \varphi_{\delta_k}{-}
 \chi_{\delta_k}) \dd x \dd s \doteq I_{17}^{a} + I_{17}^{b} +I_{17}^{c},
 \end{aligned}
 \]
 where
 \[
 \begin{aligned}
 &
 I_{17}^{a} = \int_0^T \int_\Omega \left(\partial_t \chi_{\delta_k} +
  \gamma(\chi_{\delta_k})  - \Theta(w_{\delta_k})\right)( \varphi_{\delta_k}{-}
 \chi_{\delta_k}) \dd x \dd s   \leq C \|  \varphi_{\delta_k}-
\chi_{\delta_k}\|_{L^2(0,T;L^2(\Omega))} \to 0 \text{ as $k \to
\infty$},
\\
&
 I_{17}^{b} =  \int_0^T  \pairing{}{W^{1,p}(\Omega)}{\xi_{\delta_k}}{\varphi_{\delta_k}{-}
 \chi_{\delta_k}} \dd s \leq 0
 \\
 &
  I_{17}^{c} =  \int_0^T \int_\Omega \frac{\varepsilon(\uu_{\delta_k})
\mathrm{R}_e\varepsilon(\ub_{\delta_k})}{2} \left(\varphi_{\delta_k}{-}
 \chi_{\delta_k}\right) \dd x \dd s \leq 0
 \end{aligned}
 \]
 the second inequality due to \eqref{xi-def}, and the third one to the fact that $\varphi_{\delta_k}\leq
 \chi_{\delta_k}$ a.e.\ in $\Omega \times (0,T)$. Calculations completely analogous to
 the ones developed in the proof of  Thm.\ \ref{teor3} yield that $I_{18}, \, I_{19} \to 0$ as $k\ \to \infty$.
 In this way, from \eqref{from-hk-delta}  we conclude
\eqref{to-prove-delta}.

We are now in the position  to pass to the limit in (the time-integrated version
of)  \eqref{ineq-system2-integrated-delta} and conclude
\eqref{variational-limit}. To this aim,  we observe that, for any fixed test
function $\varphi$ as in \eqref{variational-limit}, $\supp(\varphi)$
is a compact subset of $\overline{\Omega}\times [0,T]$. Hence  there
exists $\underbar{\chi}>0$ such that $\chi(x,t) \geq
\underbar{\chi}>0$ for all $(x,t )\in \supp(\varphi)$, and, by
\eqref{convergences-degen-chi-2}, there exists $\bar{k} \in
\mathbb{N}$ such that for $k \geq \bar k$
\begin{equation}
\label{preliminary} \chi_{\delta_k}(x,t) \geq \frac12
\underbar{\chi}>0 \quad \text{for all }(x,t) \in \supp(\varphi).
\end{equation}
 Therefore,
\begin{equation}
\label{escamotage} \lim_{k \to \infty} \int_0^T \int_\Omega
\xi_{\delta_k} \varphi \dd x \dd t =0,
\end{equation}
since $\supp(\xi_\delta) \subset \{ \chi_\delta=0\}$ by
\eqref{ineq-system2-integrated-delta-2}.
Also exploiting \eqref{to-prove-delta}, we succeed in taking the limit of the left-hand side of
\eqref{ineq-system2-integrated-delta}.
 As for the
right-hand side, we use \eqref{conv-teta-wk} and argue in the
following way
\begin{equation}
\label{nice-argum}
\begin{aligned}
  \limsup_{\delta_k \to 0} \left(-\int_0^T \int_\Omega
 \frac{\varepsilon(\uu_{\delta_k})
\mathrm{R}_e\varepsilon(\ub_{\delta_k})}{2}\varphi \dd x \dd t \right)
 &  = - \liminf_{\delta_k \to 0} \int_0^T \int_\Omega
\frac{1}{2(\chi_{\delta_k}+\delta_k)} \eeta_{\delta_k}
\,\mathrm{R}_e \,\eeta_{\delta_k} \varphi  \dd x \dd t
\\ &
\leq - \int_0^T \int_\Omega \frac{1}{2\chi} \eeta \,\mathrm{R}_e
\,\eeta \varphi \dd x \dd t\,,
\end{aligned}
\end{equation}
where   we have used that, thanks to
\eqref{convergences-degen-chi-2} and \eqref{preliminary},
$\frac{1}{2(\chi_{\delta_k}+\delta_k)}  \to \frac1{2\chi}$ uniformly
on $\supp(\varphi)$, thus the last inequality e.g.\ follows from the lower
semicontinuity result of  \cite{acerbi-buttazzo}.

Finally, \eqref{energ-ineq-lim} follows from taking  the limit as
$\delta_k \to 0$ of the total energy inequality \eqref{pass-lim-delta-tot}, written on the interval
$(0,t)$ for \emph{any} $t \in (0,T]$.
Observe  that, by \eqref{pass-lim-delta-tot}   (cf.\ also the
arguments in the proofs of \cite[Prop.\ 4.3]{mrz}), 
the map
\[
\begin{aligned}
t\mapsto \mathcal{H}_\delta(t):=
 \int_\Omega w_\delta(t)\, \mathrm{d}x & +\frac12 \int_\Omega |\partial_t\uu_\delta (t)|^2\, \mathrm{d}x
 +\frac12
\int_\Omega \eeta_\delta(t) \,\mathrm{R}_e \,\eeta_\delta(t) \, \mathrm{d}x
 +  \Phi(\chi_\delta(t)) +\int_\Omega
W(\chi_\delta(t))\, \mathrm{d}x
\end{aligned}
\]
has (uniformly) bounded variation.
 Therefore, by Helly's theorem up to a subsequence there exists $\mathcal{H}$ such that
$\mathcal{H}_{\delta_k}(t) \to \mathcal{H}(t)$   for  all  $t \in [0,T]$.   To identify $\mathcal{H}$,  we
 take the $\liminf$ as $\delta_k \to 0$
  of the first, second,   fourth,   and fifth term  in $\mathcal{H}_\delta(t)$,  exploiting
convergences \eqref{convergences-degen-w},  
\eqref{convergences-degen-chi-1}, \eqref{convergences-degen-chi-2},
 \eqref{eq:-added-conve-sec7},  as well as \eqref{to-prove-delta},
and relying on
lower semicontinuity arguments.
 Therefore we conclude that \eqref{inequality-H-sec7} holds.
 Finally,  inequality
\eqref{integral-inequality-forH} follows combining the following facts: on the one hand,
since $(\mathcal{H}_{\delta})_\delta \subset L^\infty (0,T)$ is uniformly bounded due to estimates
\eqref{aprio-delta-1}--\eqref{aprio-delta-4},
the  dominated convergence theorem ensures
\[
\int_{t_1}^{t_2} \mathcal{H}_{\delta_k}(r) \dd r \to  \int_{t_1}^{t_2} \mathcal{H}(r) \dd r \ \text{as } k \to \infty \quad \text{for all } 0 \leq t_1 \leq t_2 \leq T.
\]
On the other hand,
 on account
convergences
\eqref{convergences-degen-w}--\eqref{convergences-degen-chi-2},
by weak lower semicontinuity arguments we have that
$\liminf_{k \to \infty} \int_{t_1}^{t_2} \mathcal{H}_{\delta_k}(r) \dd r $ is greater or equal than the right-hand side of \eqref{integral-inequality-forH}. This concludes the proof.
\end{proof}

\begin{remark}
\upshape
\label{rmk:explanation}
As it is clear from the above lines, the proof of Thm.\ \ref{teor5}
  strongly relies on the following properties:
\begin{compactenum}
\item[1.] the compact embedding of  $W^{1,p}(\Omega)$   into $\mathrm{C}^0(\overline\Omega)$;
\item[2.] the fact that
$t\mapsto \chi_\delta (t,x)$ is nonincreasing for all $x \in\overline{\Omega}$, which follows from the
irreversibility constraint.
\end{compactenum}
These are the reasons why  we have restricted the analysis of the degenerate limit
 to the irreversible system. Within this setting, we further need to
assume $\rho=0$. Indeed, because of the lack of
 estimates on $\dive(\ub_t)$ for $\delta \down 0$,
 we would not be able to
 the limit in the term
 $\rho \dive(\ub_t) \Theta(\w)$ in
 \eqref{eq0d} as
$\delta\down 0$.

 We also point out that, seemingly, the total energy inequality \eqref{energ-ineq-lim}
cannot be improved to an inequality holding on any subinterval $(s,t)\subset (0,T)$. Indeed, for
the sequence $(\eeta_{\delta_k})_k$ only the weak convergence \eqref{convergences-degen-eta} is available,
which does not allow us to take the limit of the right-hand side of \eqref{pass-lim-delta-tot} but for $s=0$.

Finally, observe that the proof of Thm.\ \ref{teor5} simplifies if the operator $\mathcal{B} $ is given by the
nonlocal $s$-Laplacian operator $A_s$. In this case, in order to pass to the limit in \eqref{ineq-system2-integrated-delta} it is no longer necessary to prove the strong convergence \eqref{to-prove-delta}
for $(\chi_{\delta_k})_k$. In fact, the term $ \mathbf{d}(x,\chi_{\delta_k}) \cdot \nabla \varphi $ in  \eqref{ineq-system2-integrated-delta} is replaced by  $a_s (\chi_{\delta_k},\varphi)$,  which can be dealt with by
 weak convergence arguments due to the linearity of the operator $A_s$.

\end{remark}
\paragraph{Acknowledgments.}
The authors would like to thank all the referees for their careful reading of the paper, and in particular one of them for a very useful suggestion on how to improve Thm.\ \ref{teor5}.   They are also grateful to  Christiane Kraus and Christian Heinemann for  fruitful discussions on some topics related to this paper. 
\bibliographystyle{alpha}

\end{document}